\documentclass[12pt,a4paper]{article}
\usepackage[utf8]{inputenc}
\usepackage[T1]{fontenc}
\usepackage{graphicx}
\usepackage{amsmath, amsthm, amsfonts, amssymb}
\usepackage{comment}
\usepackage{mathrsfs}  
\usepackage[a4paper, margin=2.7cm]{geometry}
\usepackage[colorlinks=true,linkcolor=blue,citecolor=blue,urlcolor=blue,breaklinks]{hyperref}
\usepackage{dsfont}
\usepackage{caption}
\usepackage[round]{natbib}
\usepackage{cleveref}
\usepackage{stmaryrd}
\usepackage{algorithm} 
\usepackage{algpseudocode} 
\usepackage{framed}
\usepackage{lipsum}
\usepackage[toc,page]{appendix}
\usepackage{alltt}
\usepackage{multirow}
\usepackage{subcaption}
\usepackage{caption}
\usepackage{float} 
\usepackage[export]{adjustbox}

\newcommand{\fact}[1]{#1\mathpunct{}!}

\DeclareMathOperator*{\argmin}{arg\,min}

\def\R{\mathbb{R}}

\def\E{\mathbb{E}}

\def\U{\mathbb{U}}
\def\P{\mathbb{P}}
\def\v{{\rm Var}}
\def\d{{\textup d}}

\def\d{\,\mathrm{d}}

\newtheorem{thm}{Theorem}[section]

\newtheorem{cor}[thm]{Corollary}
\newtheorem{lem}[thm]{Lemma}
\newtheorem{prp}[thm]{Proposition}
\newtheorem{rem}[thm]{Remark}

\theoremstyle{definition}
\newtheorem{dfn}[thm]{Definition}

\newcommand{\1}{\mathds{1}}
\newcommand{\bm}[1]{\mathbf{#1}}

\title{\huge \textbf{ Infinite  random forests for imbalanced classification tasks}}
\author{Moria Mayala\footnote{This work was supported by a grant from the Ecole doctorale Science Mathématques de Paris Centre and partially funded by International
Emerging Actions 2022 and  the ANR T-REX project.}  , Olivier Wintenberger , Charles Tillier  and Clément Dombry\\ LPSM, UMR 8001, Sorbonne Université\\
LMV, UMR 8100, Université de Versailles\\
LMB, UMR 6623, Université de Franche-Comté}

\date{}
\AtEndDocument{\bigskip{\footnotesize
		\noindent
		\textsc{Sorbonne Université, CNRS, Laboratoire de Probabilités, Statistique et Modélisation
			F-$75005$ Paris, France}
		\textit{E-mail address}: \texttt{\href{mailto:email@example.com}{moria.mayala@sorbonne-universite.fr}}
		\par
		\addvspace{\medskipamount}
		\noindent
		\textsc{Sorbonne Université, CNRS, Laboratoire de Probabilités, Statistique et Modélisation
			F-$75005$ Paris, France}
		\textit{E-mail address}: \texttt{\href{mailto:olivier.wintenberger@sorbonne-universite.fr}{olivier.wintenberger@sorbonne-universite.fr}}
		\par
		\addvspace{\medskipamount}
		\noindent
		\textsc{Université Bourgogne Franche-Comté, École doctorale CARNOT-PASTEUR}
		\textit{E-mail address}: \texttt{\href{mailto:clement.dombry@univ-fcomte.fr}{clement.dombry@univ-fcomte.fr}}
		\par
		\addvspace{\medskipamount}
		\noindent
		\textsc{Université Paris-Saclay, Laboratoire de Mathématiques de Versailles, 78000, Versailles.}
		\textit{E-mail address}: \texttt{\href{mailto:charles.tillier@gmail.com}{charles.tillier@gmail.com}}
}}

\begin{document}
	\maketitle
	\begin{abstract}
We study predictive probability inference in classification tasks using random forests under class imbalance. We focus on two simplified variants of Breiman’s algorithm, namely subsampling Infinite Random Forests (IRFs) and under-sampling IRFs, and establish their asymptotic normality. In the under-sampling setting, training data from both classes are re-sampled to achieve balance, which enhances minority class representation but introduces a biased model. To correct this, we propose a debiasing procedure based on Importance Sampling (IS) using odds ratios. We instantiate our results using 1-Nearest Neighbor (1-NN) classifiers as base learners in the IRFs and prove the nearly minimax optimality of the approach for Lipschitz continuous objectives. We also show that the IS bagged 1-NN estimator matches the convergence rate of its subsampled counterpart while attaining lower asymptotic variance in most cases.  Our theoretical findings are supported by simulation studies, highlighting the empirical benefits of the proposed approach.\\

	
	\end{abstract}
    	
	\section{Introduction}
	
In binary classification, the goal is to predict a binary response variable $Y \in \{0,1\}$ based on covariates $X$ taking values in a feature space $\mathcal{X}$. The central object of interest is the regression function
\begin{equation}\label{intro:eq:px}
    \mu (\bm{x}) := \mathbb{P}(Y=1 \mid X=\bm{x}), \qquad \bm{x} \in \mathcal{X}\,,
\end{equation}
which represents the conditional probability of the positive class given input $\bm{x}$.

A key challenge in many classification problems is class imbalance, where one class (the majority) is heavily overrepresented compared to the other (the minority). This phenomenon is common across real-world applications and can arise for various reasons ranging from biased sampling or measurement errors during data collection, to inherent class asymmetries in the underlying population. The latter often occurs in settings involving rare events, such as image detection, network intrusion, fraud detection, or medical diagnosis of rare diseases; see, for example, \citet{kruegel2003bayesian, chanlearning, khalilia2011predicting, rahman2013addressing, makki2019experimental}.

From a predictive modeling perspective, class imbalance can significantly degrade performance, particularly when standard algorithms implicitly assume balanced class distributions. This is especially problematic when the minority class is of primary interest. As a result, handling class imbalance has become a central topic in classification research, with numerous methods developed to address its impact on predictive performance.

Two main strategies are commonly used to address class imbalance. The first consists of data-level methods, which aim to rebalance the class distribution by transforming the training dataset. Among these, oversampling and under-sampling techniques such as SMOTE are the most widely used. The second approach involves algorithm-level methods, which either adapt existing learning algorithms or design new ones to handle imbalance more effectively. Examples include SHRINK and Near-Bayesian Support Vector Machines (NBSVM); see \citet{kubat1997learning} and \citet{datta2015near}. For a broader overview of standard techniques for handling imbalance in classification, we refer the reader to \citet{brownlee2020imbalanced} and \citet{stefanowski2015dealing}.

Random Forests (RF), introduced by \citet{breiman2001random}, are ensemble learning algorithms that aggregate multiple decision trees to improve predictive accuracy. They have become popular tools for both classification and regression tasks due to their empirical success and flexibility. A detailed review of theoretical developments is provided by \citet{biau2016random}, shedding light on their behavior and limitations. Additional recent results in high-dimensional settings may be found in \citet{chi2022asymptotic,klusowski2024large}. 
Several tree-based algorithms have been specifically designed to handle imbalanced classification problems. For instance, \citet{o2019random} introduced the $q^*$-classifier, a RF variant tailored to class imbalance. \citet{cieslak2012hellinger} proposed using the Hellinger distance as a splitting criterion to improve minority class detection, while \citet{chen2004using} explored cost-sensitive learning approaches within the RF framework to account for imbalance.

The first objective of this paper is to investigate the asymptotic normality of \textit{Infinite Random Forest} (IRF) algorithms in the context of imbalanced classification. We begin by analyzing subsampled IRFs, which average predictions over all possible size-$s$ subsamples of the training data, following the framework of \citet{wager2018estimation} and \citet{peng2022rates}. We then introduce a generalization of this approach, referred to as under-sampled IRFs, that incorporates a downsampling strategy to mitigate class imbalance. Specifically, each subsample of size $s$ is constructed by independently drawing $s_0\le n_0$ elements from the majority class and $s_1\le n_1$ elements from the minority class, i.e. $n_1\ll n_0$ with sample size $n=n_0+n_1$. The resulting IRF is obtained by averaging trees built on all such  subsamples. A natural and practical choice, such as $s_0 = s_1$, ensures that the subsamples used for training are balanced across classes.

To address the bias introduced in the model by under-sampling, we propose a correction procedure based on importance sampling (IS). Specifically, we first construct a biased estimator by learning on under-sampled data and then we apply an IS adjustment based on odds ratios. We show that the resulting estimator is asymptotically normal in the original model, and we support our theoretical findings with simulation studies demonstrating its empirical effectiveness.

In a broader context, \citet{vogel2020weighted} demonstrated that debiasing strategies based on IS  preserve the generalization ability of learning algorithms in particular, in terms of excess risk within the empirical risk minimization framework. IS techniques have also been widely applied in transfer learning, where similar challenges to class imbalance often arise; see \citet{maddouri2022robust, gupta2023boosting} for recent developments. To the best of our knowledge, the present work is the first to investigate IS-based debiasing in the context of under-sampled IRFs.


To illustrate our theoretical results, we instantiate IRFs using 1-Nearest Neighbor (1-NN) classifiers as base learners. This leads to ensemble predictors constructed by aggregating 1-NN estimators over subsampled and under-sampled training sets. Prior work by \citet{biau2010rate} and \citet{xue2018achieving} has shown that bagging 1-NN with subsample size \( s = o(n) \) can achieve nearly minimax optimal convergence rates. 

Our second main contribution builds on this line of work. We first establish the asymptotic normality of the subsampling-based bagged 1-NN estimator, showing that it retains a nearly minimax convergence rate over the class of Lipschitz functions. We then extend this result to its under-sampled counterpart, which is generally inconsistent due to the bias induced by the under-sampling. Finally, we prove that our IS-corrected bagged 1-NN estimator is both consistent and asymptotically normal, while achieving the same nearly minimax rate. This enables a direct comparison of asymptotic variances, revealing the advantage of our approach in imbalanced classification settings.

The remainder of the introduction sets out the notation and shorthand used throughout the paper. In Section~\ref{section:2}, we present the statistical framework, introduce the problem, and motivate the models under study. Section~\ref{sec:Framework} provides a brief overview of the random forest algorithm and introduces the IRF model. Our main theoretical results are stated in Section~\ref{sec:results}, where we analyze the bias and establish central limit theorems for subsampling, under-sampling, and importance sampling IRF estimators. Section~\ref{sec:KNN} illustrates these results with an application to bagged 1-NN classifiers. Numerical experiments are presented in Section~\ref{sec:illustration}, followed by a conclusion and discussion in Section~\ref{sec: conclusion}. All proofs are deferred to Appendices~\ref{sec:appendix}–\ref{sec: proof5}.\\

\noindent \textbf{Notation and shortcuts}. The indicator function of the set $A$ is denoted by $\1\{A\}$ and its cardinality 
is denoted by $|A|$. Besides, i.i.d. stands for independent and identically distributed, r.h.s. for right-hand side, w.r.t for with respect to, c.d.f. for cumulative distribution function.

\section{The problem and the models}\label{section:2}

\subsection{Problem statement}\label{sub:motiv}
Let $(X_i, Y_i)_{1 \leq i \leq n}$ be i.i.d. observations drawn from the distribution of a generic pair $(X, Y) \in \mathcal{X} \times \{0,1\}$, where $\mathcal{X} \subset \mathbb{R}^d$ denotes the feature space. The joint distribution of $(X, Y)$ is fully characterized by the marginal distribution of $X$ and the regression function $\mu$, defined as
\begin{equation}\label{eq:px}
    \mu(\bm{x}) := \mathbb{P}(Y = 1 \mid X = \bm{x})\,, \qquad \bm{x} \in \mathcal{X}\,.
\end{equation}

In the context of imbalanced classification, our objective is to estimate $\mu$ when the class labels $\{0,1\}$ are unequally represented in the data. To quantify the level of imbalance, we rely on the standard notion of the \textit{imbalance ratio}, defined below. Without loss of generality, we assume that class $1$ corresponds to the minority class.

\begin{dfn}\label{def:ir}
Let $n_0$ and $n_1$ denote the number of observations from the majority and minority classes, respectively. The \emph{imbalance ratio} is defined as $IR = n_0 / n_1$, and the dataset is considered imbalanced whenever $IR \gg 1$.
\end{dfn}

A balanced dataset corresponds to the case where $IR = 1$. To mitigate the negative impact of class imbalance on the generalization performance of learning algorithms, our approach involves subsampling both the majority and minority classes   to create a balanced training set, following strategies similar to those proposed in \citet{chen2004using} and \citet{vogel2020weighted}. The regression function $\mu$, as defined in \Cref{eq:px}, is then estimated using an  IRF trained on this rebalanced data. This leads to what we refer to as an \emph{under-sampling IRF}. As a consequence, two distinct data-generating mechanisms are considered in the analysis of subsampling and under-sampling IRFs.




    \subsection{The models}\label{subsec:models}
\subsubsection{The original model}\label{subsec: orig_model}
This is the model of primary interest, where the goal is to accurately predict instances from the minority class. 
We define $p$ as the marginal probability of class $1$:
\begin{equation*}
	p := \mathbb{P}(Y = 1) = \int_{\mathcal{X}} \mu(\bm{x})\, \mathbb{P}(X \in d\bm{x})\,.
\end{equation*}

We assume $0 < p < 1$ and $p$ small that is, the class $1$ is under-represented. Standard machine learning algorithms often struggle to achieve reliable performance on the minority class.

\subsubsection{The biased classification model}\label{subsec: biased_moedel}
We now define the biased model, which results from applying an under-sampling strategy to both the majority and minority classes in the original dataset in order to balance class representation. Quantities associated with this model are indicated using a superscript $^*$. 

Let $p^*$ denote the proportion of class $1$ in the biased model. While the choice $p^* = 0.5$ (corresponding to a perfectly balanced case with $\mathrm{IR}^* = 1$) is a natural candidate to improve generalization performance, alternative values may sometimes yield better results in practice. In this work, we simply assume that $p < p^* < 1$. 

The distribution of a generic pair $(X^*, Y^*) \in \mathcal{X} \times \{0,1\}$ drawn from the biased model is characterized by:
\begin{equation}\label{syst:1}
\begin{cases}
\mathbb{P}(Y^* = 1) = p^* > p = \mathbb{P}(Y = 1)\,, \\
\mathbb{P}(X^* \in \cdot \mid Y^* = 1) = \mathbb{P}(X \in \cdot \mid Y = 1)\,, \\
\mathbb{P}(X^* \in \cdot \mid Y^* = 0) = \mathbb{P}(X \in \cdot \mid Y = 0)\,.
\end{cases}
\end{equation}

In other words, while the marginal distributions of $Y$ and $Y^*$ differ, the dependency structure between $X$ and $Y$ remains unchanged, as the conditional distributions $X \mid Y$ and $X^* \mid Y^*$ are identical.

Training an IRF on under-sampled data therefore yields a biased estimator of the original regression function $\mu(\bm{x})$, since the learning procedure is based on a different distribution. We denote by
\begin{align}\label{def:mustar}
    \mu^*(\bm{x}) := \mathbb{P}(Y^* = 1 \mid X^* = \bm{x})\,, \qquad \bm{x} \in \mathcal{X}\,,
\end{align}
the regression function associated with the biased model. This function $\mu^*$ is related to the true regression function $\mu$ through the odds ratio, as detailed below.

\subsubsection{The odds ratios and the Importance Sampling (IS) procedure}\label{subsec:odds}

From \eqref{syst:1}, we define the ratio 
\[
R(\bm x,p):=\dfrac{\mathbb{P}(X\in d\bm{x}\mid Y=1)}{\mathbb{P}(X\in d\bm{x}\mid Y=0)}=\dfrac{\mathbb{P}(X^*\in d\bm{x}\mid Y^*=1)}{\mathbb{P}(X^*\in d\bm{x}\mid Y^*=0)}\,,\qquad \bm{x}\in {\cal X}\,.
\]

Using Bayes's formula, odds ratios of the original and the biased models are equal. Then, we obtain the following identity
\begin{align}\label{eq:odds_ration}
\frac{p/(1-p)}{\mu(\textbf{x})/(1-\mu(\textbf{x}))} = \frac{p^{*}/(1-p^{*})}{\mu^{*}(\textbf{x})/(1-\mu^{*}(\textbf{x}))}\,,\qquad \textbf{x}\in{\cal X} \,.
\end{align}

From  \Cref{eq:odds_ration}, one observes that the odds of the two models are proportional up to a multiplicative factor  that depends only on class probabilities $p$ and $p^*$. More precisely, we have	
$$ \frac{\mu(\textbf{x})}{1-\mu(\textbf{x})}= \frac{p/(1-p)}{p^{*}/(1-p^{*})} \times \frac{\mu^{*}(\textbf{x})}{1-\mu^{*}(\textbf{x})}\,, \qquad\bm{x}\in{\cal X} \,.  $$

From the latter equation, the regression function $\mu$ for the original model and the regression function of the biased model $\mu^*$ are related as follows 
\begin{align}\label{est:mu_{I_n}S}
	\mu(\bm x)&= 1-\Big( 1 + \frac{p/(1-p)}{p^*/(1-p^*)} \frac{\mu^*(\bm x)}{1- \mu^*(\bm x)}\Big)^{-1}\nonumber\\
 &= \frac{(1-p^*)p\mu^*(\bm{x})}{p^*(1-p)(1-\mu^*(\bm{x}))+(1-p^*)p\mu^*(\bm{x})}\,, \qquad\bm{x}\in{\cal X} \,.
\end{align}
Equivalently we have
$$ \mu^*(\bm{x})= \frac{p^{*} (1-p)\mu(\bm x)}{p(1-p^{*})(1-\mu(\bm x)) +p^{*} (1-p)\mu(\bm x)}\,, \qquad\bm{x}\in{\cal X} \,.$$

Thanks to \Cref{est:mu_{I_n}S}, it is always possible to recover the original regression function $\mu(\bm{x})$ from the biased one $\mu^*(\bm{x})$ via an explicit function $g$, such that
\[
g\big(\mu^*(\bm{x})\big) = \mu(\bm{x}), \qquad \text{for all } \bm{x} \in \mathcal{X},
\]
where $g$ is defined as
\begin{equation}\label{eq:function_g}
	g(z):= \frac{(1-p^*)p z}{p^*(1-p)  (1-z) + (1-p^*)p
		z}\,,\qquad z\in [0,1] \,,
\end{equation}
and the other way round since the relation $\mu^*(\bm{x})=g^{-1}(\mu(\bm{x}))$ holds for all $\bm{x}\in{\cal X}$ with 
\begin{equation}\label{eq: inverse_g}
	g^{-1}(z):= \frac{p^{*} (1-p)z}{p(1-p^{*})(1-z) +p^{*} (1-p)z}\, \,,\qquad z\in [0,1]\,.
\end{equation}

When estimating the regression function in the biased model, we rely on a learning dataset with balanced classes. This setting allows us to construct an estimator $\widehat{\mu}^*$ of $\mu^*$ with strong predictive performance, as the balanced sampling ensures sufficient representation from both classes and prevents neglecting the minority class. 

By applying the function $g$ to $\widehat{\mu}^*$, we obtain an estimator for the original regression function $\mu$ that retains these favorable properties. The resulting estimator, defined as $\widehat{\mu}^{IS} = g(\widehat{\mu}^*)$, is called the IS estimator. Notably, this correction depends only on the class proportions $p$ and $p^*$ from the original and biased models, respectively.

\section{Framework and hypotheses}\label{sec:Framework}

\subsection{Random Forests (RF)}\label{section: back}
Random Forests (RF), introduced by \citet{breiman2001random}, aggregate the predictions of multiple decision trees to produce accurate and stable estimates. By training individual trees on bootstrap samples of the data (drawn with replacement), the ensemble yields a bagged predictor that reduces variance and improves generalization. 

Let $Z_i = (X_i, Y_i)$, for $i \in I_n := \{1, \ldots, n\}$, denote the training data, and let $\bm{x} \in \mathcal{X}$ be a test point. Following Breiman’s original formulation, a RF prediction at $\bm{x}$ is given by the bagged average of individual tree predictors $T$, defined as

\begin{align}\label{eq: random_brei}
	\widehat{\mu}_{\text{Breiman}}( \bm{x}):= \frac{1}{B} \sum_{b=1}^B T (\bm{x}; Z_{b1}^\star, \ldots, Z_{bn}^\star).   
\end{align}
In this setting, the variables $Z_{bi}^\star \overset{\text{iid}}{\sim} Z_1$ represent a bootstrap sample drawn from the original training data. The RF estimator defined in \Cref{eq: random_brei} has shown remarkably strong empirical performance across a wide range of classical statistical applications; see, for example, \citet{buja2006observations} and \citet{friedman2007bagging}. 

Throughout this work, we adopt an alternative approach based on subsampling, as proposed in \citet{buhlmann2002analyzing} and detailed below. Let $S \subset I := \{1, \ldots, n\}$ be an index set of cardinality $|S| = s$, with $1 \leq s \leq n$. Define $\bm{Z}_S := (Z_{i_1}, \ldots, Z_{i_s})$ as the $(\mathcal{X} \times \{0,1\})^s$-valued vector obtained by uniformly sampling, without replacement, $s$ observations from the training dataset $(Z_1, \ldots, Z_n)$. To implement the subsampling-based strategy, we consider $T^s(\bm{x}; \U; \bm{Z}_S)$, an individual tree predictor evaluated at the test point $\bm{x}$, constructed using the $s$-subsample $\bm{Z}_S$. It is defined as follows

\begin{equation}\label{eq:T_onesample}
	T^s(\bm{x}; \U; \bm Z_S):= \frac{1}{N_{L_{\U}(\bm{x})}( \bm X_S)} \sum_{i \in S} Y_{i}\1{\{X_{i} \in L_{\U}(\bm{x}) \}}.
\end{equation}

Here, $L_\U(\bm{x}) \subset \mathcal{X}$ denotes the leaf, i.e., the region of the feature space associated with the tree structure defined by the random variable $\U$, that contains the test point $\bm{x} \in \mathcal{X}$. The randomness $\U$ may depend on the covariates $\bm{X}_S$, but is assumed to be independent of the labels $\bm{Y}_S$. The quantity
\[
N_{L_{\U}(\bm{x})}(\bm{X}_S) := \left| \left\{ i \in S : X_i \in L_\U(\bm{x}) \right\} \right|
\]
represents the number of observations in the subsample whose features fall into the leaf $L_\U(\bm{x})$.

In this context, the RF estimator is defined as a bagged predictor, obtained by averaging the outputs of tree estimators built on subsamples $\bm{Z}_S$ of size $s$. Specifically, the bagged estimator is the empirical average over $B$ such trees and is given by
\begin{align*}
    \widehat{\mu}_B(\bm{x}) := \frac{1}{B} \sum_{b=1}^B T^s(\bm{x}; \U_b; \bm{Z}_{S_b}), \qquad \bm{x} \in \mathcal{X},
\end{align*}
where $(S_b)_{1 \leq b \leq B}$ denotes the collection of subsample indices, each of size $s$, and $(\U_b)_{1 \leq b \leq B}$ are i.i.d. random variables (conditional on $\bm{X}_{I_n} := (X_1, \ldots, X_n)$) representing the randomness in the tree construction.

It is important to note that the function $T^s(\bm{x}; \U; \bm{Z}_S)$ is invariant under permutations of the indices in $S$, which excludes the original Random Forest algorithm of \citet{breiman2001random} from our theoretical framework. To further simplify the analysis, we consider the Infinite Random Forest (IRF) model by letting $B \to \infty$; see \citet{wager2018estimation, scornet2016asymptotics} and \Cref{subsec: IRF tools} for details.

\subsection{Infinite Random Forests (IRFs)}\label{subsec: IRF tools}

We study the asymptotic properties of Infinite Random Forests (IRFs) under two distinct subsampling schemes: the standard subsampling approach, as considered in \citet{wager2018estimation}, and an under-sampling strategy designed to address class imbalance.

\subsubsection{Subsampling IRFs}\label{subsubsec:sub_sampling_{I_n}RF}

This framework aligns with the setups considered by \citet{wager2018estimation} and \citet{peng2022rates}. The \textit{subsampling strategy} consists of averaging individual tree predictors $T^s(\bm{x}; \U; \bm{Z}_S)$ over all possible subsamples of size $s$, where each subsample is obtained by randomly selecting $s$ observations from the training dataset $\bm{Z}_{I_n} = (Z_1, \ldots, Z_n)$. An explicit expression for the resulting estimator is provided in the next proposition.

\begin{prp} [IRF: the subsampling case]\label{prop:u_stat} 
The subsampling IRF \footnote{\citet{buhlmann2002analyzing} called it subbagging which is a nickname for subsample aggregating, where subsampling is used instead of bootstrap resampling.} estimator at point  $\bm{x}\in \mathcal{X}$ may be written as 
	\begin{align}
 \label{def:subs_classif}
		\widehat{\mu}^{s}(\bm x):= \binom{n}{s}^{-1} \sum_{S \subset I_n, |S|=s} \mathbb{E}[T^s(\bm{x}; \U;\bm Z_S) \mid \bm Z_S]\,,
	\end{align}

\noindent where $T^s(\bm{x}; \U;\bm Z_S)$ are individual trees evaluated at point  $\bm{x}\in \mathcal{X}$ as defined in \Cref{eq:T_onesample}.

\end{prp}
\begin{proof}
   See the detailed proof in \Cref{proof:prop:u_stat} for completeness.  
\end{proof}

Note that the sum in \Cref{def:subs_classif} contains $\binom{n}{s}$ elements since it runs over all subsets $S \subset I_n$ with cardinality $|S|=s$.  

\subsubsection{Under-sampling IRFs}\label{subsubsec:under_sampling_{I_n}RF}

We denote by $\{Z_j^0,\, 1 \le j \le n_0\} = \{(X_i, 0),\, i \in I_{n_0}\}$ and $\{Z_j^1,\, 1 \le j \le n_1\} = \{(X_i, 1),\, i \in I_{n_1}\}$ the sets of observations from class $0$ and class $1$, respectively, where $I_{n_0}$ and $I_{n_1}$ are the corresponding index sets, and $n_0$ and $n_1$ denote the total number of observations in each class. By construction, we have $n_0 + n_1 = n$.

In the imbalanced binary classification setting i.e., when the imbalance ratio $IR := n_0 / n_1 > 1$ the \textit{under-sampling strategy} proceeds as follows. We uniformly sample, without replacement, $s_0$ observations from the majority class and $s_1$ from the minority class, yielding the subsamples $(Z^0_{i_1}, \ldots, Z^0_{i_{s_0}})$ and $(Z^1_{i_1}, \ldots, Z^1_{i_{s_1}})$, respectively. Letting $S_0$ and $S_1$ denote the corresponding index sets (with $S_0 \cap S_1 = \emptyset$), we define $S := S_0 \cup S_1$ so that $|S| = s = s_0 + s_1$.

The resulting under-sample is denoted by
\[
\bm{Z}_S := \bm{Z}_{S_0 \cup S_1} := (Z^0_{i_1}, \ldots, Z^0_{i_{s_0}}, Z^1_{i_1}, \ldots, Z^1_{i_{s_1}}),
\]
and may follow a distribution that differs from the original data. When the adjusted imbalance ratio $IR^* := s_0 / s_1 = 1$, we obtain a balanced training sample, despite the original dataset being imbalanced with $IR = n_0 / n_1 > 1$.

We denote by $T^{s_0,s_1}(\bm{x}; \U; \bm{Z}_{S})$ an individual tree predictor evaluated at the test point $\bm{x}$, constructed from the under-sample $\bm{Z}_S$ of size $s = s_0 + s_1$. The corresponding RF estimator is then defined as the bagged average over all possible subsamples composed of $s_0$ observations from the $0$-class and $s_1$ from the $1$-class. Note that the individual tree $T^{s_0,s_1}$ can equivalently be expressed in the following form

 \begin{equation}\label{def:base_two_sampl}
		T^{s_0,s_1}(\bm{x};\U;\bm{Z}_S):= \frac{N_{L_{\U}(\bm{x})}(\bm X_{S_1}^1) }{N_{L_{\U}(\bm{x})}(\bm X_{S_1}^1) + N_{L_{\U}(\bm{x})}(\bm X_{S_0}^0)}\,,
	\end{equation}

	\noindent where   $\bm X_{S_j}^j$ corresponds to the vector consisting of all the covariates from the class $j=\{ 0,1\}$ and 
	$$N_{L_{\U}(\bm{x})}(\bm X_{S_j}^j):= |\{ i\in S_j:   X_i^j \in L_{\U}(\bm{x})\}|\,\qquad\bm{x}\in{\cal X},$$ 
 is the number of elements of the $j$-class that belong to the leaf $L_{\U}(\bm{x})$, $j=\{0,1\}$.
We adopt the convention $\frac{0}{0}= 0$ if $N_{L_{\U}(\bm{x})}(\bm X^0_{S_1})=0$ and $N_{L_{\U}(\bm{x})}(\bm X^1_{S_0})=0$. 

\begin{prp}[IRF: the under-sampling case]\label{prop: case_two}
	The  under-sampling IRF estimator at point  $\bm{x}\in \mathcal{X}$ may be written as
	\begin{align}\label{est: under_classif} 
		\widehat{\mu}^{s_0,s_1}(\bm{x}):= \left(\binom{n_0}{s_0} \binom{n_1}{s_1} \right)^{-1} \sum_{S_0\subset I_{n_0}} \sum_{S_1\subset I_{n_1}} \mathbb{E} [T^{s_0,s_1}(\bm{x}; \U; \bm{Z}_{S})\mid \bm Z_S]\,,
	\end{align}
 where $T^{s_0,s_1}(\bm{x}; \U; \bm{Z}_{S})$
are individual trees evaluated at point  $\bm{x}\in \mathcal{X}$ as defined in \Cref{def:base_two_sampl}. 
\end{prp}

\begin{proof}
    See the detailed proof in \Cref{subsec2:proof_bis} for completeness.
\end{proof}
 Here, the sums are taken over all subsets $S_0 \subset I_{n_0}$ and $S_1 \subset I_{n_1}$ of cardinalities $|S_0| = s_0$ and $|S_1| = s_1$, where $|I_{n_0}| = n_0$ and $|I_{n_1}| = n_1$ denote the total number of observations in class $0$ and class $1$, respectively. The total number of possible subsamples is thus given by $\binom{n_0}{s_0} \binom{n_1}{s_1}$. In summary, the under-sampling IRF estimator is constructed by drawing, without replacement, $s_0$ observations from the majority class (class $0$) and $s_1$ observations from the minority class (class $1$), and averaging the individual tree predictors $T^{s_0,s_1}(\bm{x}; \U; \bm{Z}_{S})$ built on the resulting balanced subsamples.

\section{Asymptotic properties of IRFs}\label{sec:results}
 
\subsection{CLT for subsampling and under-sampling IRFs}\label{subsec:mainresult_oneD}
\subsubsection{The subsampling case}

Noticing that $\bm Z_S$ is identically distributed for every $S\subset I_n$ with $|S|=s$, we fix on $S=I_s=\{1,\ldots, s\}$ in the sequel. For simplicity,  we denote $T^s(\bm x, \U, \bm Z_{I_s})$ by $T^s$, 
$
T^s_1:=\E[T^s\mid Z_1]-\E[T^s]\,,$  and $V_1^s:=\v(T_1^s)\,.
$
Our assumptions depend on the crucial quantity $V_1^s$ as $n\to \infty$. We set the following assumption.
 
\begin{enumerate}
	\item [ \textbf{(H1)}]  \textbf{Lindeberg condition:}  The individual trees $T^s$ satisfy  
	$$
	nV_1^s\to \infty\,, \qquad n\to \infty\,.
	$$ 

\end{enumerate}

We are now ready to state the central limit theorem for the subsampling IRF estimator $\widehat{\mu}^s$.

\begin{thm}\label{th2} Let $\widehat{\mu}^{s}(\bm x)$ be the subsampling IRF at point $\bm{x}\in\mathcal{X}$ defined in \Cref{def:subs_classif}.
Under the assumption \textbf{(H1)}  we have  
	\begin{align*}
\sqrt{\frac{n}{s^2 V_1^s}}( \widehat{\mu}^s(\bm{x})- \E[\widehat{\mu}^s(\bm x)])
		\overset{d}{\longrightarrow} \mathcal{N}(\, 0,1)\,,\qquad  n \to \infty\,.
	\end{align*}
\end{thm}

\begin{proof}
Theorem \ref{th2} is the an alternative of Theorem 1 in \citet{peng2022rates} where it is assumed the weaker assumption $\v(T^s)/(nV_1^s)\to 0$, as $n\to \infty$. However we were not able to check the Lindeberg condition under this weaker assumption and we  propose an alternative proof in  Appendix \ref{subsec2:proof}; see also \citet{mentch2014quantifying, wager2018estimation} for similar works.
\end{proof}

Assuming the following bias condition
\begin{itemize}
    \item[\bf(H2)] {\bf Bias condition}$$ \sqrt{\frac{n}{s^2 V_1^s}}| \E[\widehat{\mu}^s(\bm{x})]-\mu(\bm x)| \longrightarrow 0\,,\qquad n\to \infty\, ,\qquad \bm{x}\in\mathcal{X}\,,$$
\end{itemize}
we can derive, as a direct consequence of \Cref{th2}, the following central limit theorem.

\begin{cor}\label{corth2} 
Let $\widehat{\mu}^{s}(\bm x)$ be the subsampling IRF at point $\bm{x}\in\mathcal{X}$ defined in \Cref{def:subs_classif}. Under assumptions \textbf{(H1)} and \textbf{(H2)} we have  
	\begin{align*}
\sqrt{\frac{n}{s^2 V_1^s}}( \widehat{\mu}^s(\bm{x})- \mu(\bm x))
		\overset{d}{\longrightarrow} \mathcal{N}(\, 0,1)\,,\qquad  n \to \infty\,.
	\end{align*}
\end{cor}
\begin{proof}
   See the proof in \Cref{proof:corth2} for completeness.
\end{proof}
\color{black}
 
\subsubsection{The under-sampling case}\label{subsec: under-samp}
Assume that $n_0$ and $n_1$ are random integers such that $n_0 \to \infty$ and $n_1\to \infty$ a.s. as $n\to\infty$. The under-sampling procedure is well-defined given $n_0$ and $n_1$. In the remainder of the paper, all statements are understood a.s. given $n_0$ and $n_1$.  Let us recall $\bm Z_S=\bm{Z}_{S_0 \cup S_1}$  and notice that it is identically distributed for every $|S_0|=s_0$ and $|S_1|=s_1$. Therefore we fix $S_i=I_{s_i}=\{1,\ldots,s_i\}$, $i=0,1$, and denote $$T^{s_0,s_1}:=T^{s_0,s_1}(\bm x, \U, \bm Z_{I_{s_0}\cup I_{s_1}})$$ where $\bm Z_{I_{s_0}\cup I_{s_1}}:=(Z^0_1,\ldots,Z^0_{s_0},Z^1_1,\ldots,Z^1_{s_1})$,  and also 
\begin{align*}
   T^{s_0,s_1}_{1,0}:=\E[T^{s_0,s_1}\mid Z_1^0]-\E[T^{s_0,s_1}]\,,&\mbox{ and }V^{s_0,s_1}_{1,0}:=\v(T^{s_0,s_1}_{1,0})\,,\\
      T^{s_0,s_1}_{1,1}:=\E[T^{s_0,s_1}\mid Z_1^1]-\E[T^{s_0,s_1}]\,,&\mbox{ and }V^{s_0,s_1}_{1,1}:=\v(T^{s_0,s_1}_{1,1})\,.
\end{align*}
We need the following assumption. 
\begin{enumerate}
	\item [\textbf{(H1')}]  \textbf{Lindeberg condition:} The individual trees $T^{s_0,s_1}$ satisfy  
 $$  (n_0\wedge n_1) (V_{1,0} ^{s_0,s_1}\wedge V_{1,1} ^{s_0,s_1})  \longrightarrow \infty\, ,\qquad n \to \infty\,. $$
 \end{enumerate}
 
Assumption \textbf{(H1')} generalizes condition \textbf{(H1)} for the under-sampling case. We are now ready  to state the asymptotic normality of the under-sampling IRF  estimator $\widehat{\mu}^{s_0,s_1}$. 
\begin{thm}\label{th1}
Let $\widehat{\mu}^{s_0,s_1}(\bm x)$ be the under-sampling IRF at point $\bm{x}\in\mathcal{X}$ defined in \Cref{est: under_classif}. Under the assumption \textbf{(H1')}  we have
\begin{align*}
\sqrt{\frac{1}{s_0^2 V_{1,0} ^{s_0,s_1}/n_0+ s_1^2 V_{1,1} ^{s_0,s_1}/n_1}}(\widehat{\mu}^{s_0,s_1}(\bm x) - \mathbb{E}[\widehat{\mu}^{s_0,s_1}(\bm x)]) \overset{d}{\longrightarrow} \mathcal{N}(0,1 ) \,,\quad n \to \infty\,.
\end{align*}
\end{thm}
\begin{proof} \Cref{th1} extends the result of \Cref{th2} to the under-sampling setting, with a proof that follows similar arguments,  see \Cref{subsec2:proof2}. 
\end{proof}
  

\subsection{CLT for importance sampling IRFs} 
\subsubsection{Proportional classes framework}

Throughout this section, we make use of the following assumption. 
\begin{enumerate}
	\item [\textbf{(G0)}] {\bf Proportional classes:}  There exists $p,p^\ast \in(0,1)$ such that 
 \begin{align}\label{eq:p}
\widehat{p}&:=\dfrac{n_1}{n_0+n_1}\longrightarrow p\,,\qquad n\to\infty\,, \\
\label{eq:p^*}
 \widehat{p}^{*}&:= \frac{s_1}{s_1+s_0}\longrightarrow p^\ast, \qquad  n\to\infty\,. 
\end{align}
\end{enumerate}
\begin{rem}
 Here $0<\widehat p<1$ can be arbitrarily small but not negligible as $n\to \infty$. We refer to \citet{aghbalou2024sharp} for severe imbalanced settings $\widehat p_n\to 0$ as $n\to \infty$. Finally notice that  the convergences above is deterministic since we work conditionally to $n_0$ and $n_1$.
\end{rem}


\subsubsection{Asymptotic bias under proportional classes}\label{sec:bias}
 
In this section, we establish consistency results for both the subsampling and the under-sampling IRF estimators. Before proceeding, we introduce the notion of the diameter of a leaf $L_{\U}(\bm{x})$, defined as
\[
\text{Diam}(L_{\U}(\bm{x})) := \sup_{\bm{x}, \bm{x}' \in L_{\U}(\bm{x})} \|\bm{x} - \bm{x}'\|_\infty,
\]
where $\|\cdot\|_\infty$ denotes the maximum norm on $\mathbb{R}^d$. We now state two classical assumptions that will be used in the analysis.

\begin{enumerate}
\item [\textbf{(G1)}] \textbf{Smoothness:
    } The regression function $\mu$ is a L-lipschitz function w.r.t  the max-norm.
	
	\item [ \textbf{(G2)}] \textbf{Leaf assumption:} For any $\bm{x} \in \mathcal{X}$, the random leaf $L_{\U}(\bm{x})\subset \mathbb{R}^d$ satisfies as $n\to \infty$ $$\mathbb{E}  \big[  \textup{Diam}(L_{\U}(\bm{x})) \mid X\in L_{\U}(\bm{x}) \big] \to 0 \ \text{and} \  \mathbb{P}(N_{L_{\U}(\bm{x})}(X_S)=0)\to 0.$$	
\end{enumerate}
From these assumptions, we have the following result. 
\begin{prp}\label{rem:conv_mu}
    	Under \textbf{(G1)}-\textbf{(G2)}, for every $\bm{x}\in{\cal X}\,$, we have for the original model   \begin{align*}\mathbb{P}(Y=1 |  X \in L_{\U}(\bm{x})) &\longrightarrow \mathbb{P}(Y=1 |  X =\bm{x}) =\mu(\bm{x})\,,\qquad n \to \infty\,,
\end{align*}
and similarly for the biased model
     \begin{align*}
     \mathbb{P}(Y^*=1 |  X^* \in L_{\U}(\bm{x})) &\longrightarrow \mathbb{P}(Y^*=1 |  X^*=\bm{x}) =\mu^*(\bm{x})\,,\qquad n \to \infty\,.
   \end{align*}
   \end{prp}
\begin{proof}
    See the detailed proof in \Cref{proof:rem:conv_mu} for completeness.
\end{proof}
 We also study the consistency of the subsampling IRF estimator $\widehat{\mu}^s(\bm{x})$ defined in Equation \eqref{def:subs_classif}. This is stated in \Cref{rem:bias_mu_s} below.

\begin{prp}\label{rem:bias_mu_s}
    Under \textbf{(G1)}-\textbf{(G2)}, the subsampling IRF estimator $\widehat{\mu}^s(\bm{x})$ is an asymptotically unbiased estimator of $\mu(\bm{x})$.
\end{prp}

\begin{proof}
    See \Cref{proof:rem:bias_mu_s}.
\end{proof}

We now focus on the under-sampling case. In contrast to Proposition \ref{rem:bias_mu_s},  the centering term $\mathbb{E}[\widehat{\mu}^{s_0,s_1}( \bm x)]$ in Theorem \ref{th1} does not converge to the target $\mu(\bm x)$ but instead, drawing the parallel with the Section \ref{subsec: biased_moedel}, it actually converges  to the regression function for the biased model $\mu^*( \bm x):=\mathbb{P}( Y^*=1| X^* =\bm{x})$, $\bm x\in\mathcal{X}$. Indeed, under \textbf{(G0)} and a control on the minimum number of observations per leaf, it is expected that for every $\bm x\in\mathcal{X}$, as $n\to\infty$,
\begin{eqnarray*}
		\mathbb{E}[\widehat{\mu}^{s_0,s_1}(\bm{x})]  = \mathbb{E}[T^{s_0,s_1}] 
		&\sim & \frac{p^{*}(1-p) \mu (\bm{x})}{p(1-p^{*})(1-\mu(\bm{x})) +p^{*} (1-p)\mu(\bm{x})}= \mu ^{*} (\bm{x})\,.
	\end{eqnarray*}

 \noindent As a consequence, in the following we write $\widehat{\mu}^\ast( \bm x)=\widehat{\mu}^{s_0,s_1}( \bm x)$. If moreover
\begin{itemize}
   \item[\bf(H2')] {\bf Bias condition}$$ \sqrt{\frac{1}{s_0^2 V_{1,0} ^{s_0,s_1}/n_0+ s_1^2 V_{1,1} ^{s_0,s_1}/n_1}}| \mathbb{E}[\widehat{\mu}^\ast(\bm{x})]-\mu^{\ast}(\bm x)| \longrightarrow 0\,,\quad n\to \infty\, ,\,  \bm{x}\in\mathcal{X}\,,$$
\end{itemize}
holds, as a direct consequence of \Cref{th1}, we have the following central limit theorem for the under-sampling estimator $\widehat{\mu}$.
\begin{cor}\label{th: mu_start_TCL}
Let $\widehat{\mu}^*(\bm{x})$ be the under-sampling IRF estimator at point  $\bm{x}\in \mathcal{X}$ as defined in \Cref{est: under_classif}. 	
 Under the assumptions \textbf{(H1')}-\textbf{(H2')}, and \textbf{(G0)} we have  
$$
		\sqrt{\frac{1}{s_0^2 V_{1,0} ^{s_0,s_1}/n_0+ s_1^2 V_{1,1} ^{s_0,s_1}/n_1}}(\widehat{\mu}^*(\bm{x})- \mu^* 
			(\bm{x}))
		\overset{d}{\longrightarrow}
		\mathcal{N}(\, 0,  1 )\, ,\qquad n \to \infty\,.
$$
\end{cor}
\begin{proof} See the detailed proof in \Cref{subsec:proof_th_mu_start_TCL} for completeness.
\end{proof}

\subsubsection{Importance sampling IRF} 
As discussed in \Cref{sec:bias}, the under-sampling IRF estimator $\widehat{\mu}^* = \widehat{\mu}^{s_0,s_1}$ does not converge to the target function $\mu$, but rather to $\mu^*$. To correct for this bias, a natural strategy is to apply the odds ratio function $g$ defined in \Cref{eq:function_g} to $\widehat{\mu}^*$ 
 so that the resulting estimator
\begin{align}\label{eq:IS}
	g( \widehat{\mu}^*(\bm x))= 1-\Big( 1 + \frac{p/(1-p)}{p^*/(1-p^*)} \frac{\widehat{\mu}^*(\bm x)}{1- \widehat{\mu}^*(\bm x)}\Big)^{-1}
\end{align}
could consistently estimate $\mu(\bm x)$. By respectively estimating $p$ and $p^{*}$ with their empirical counterpart as in  \Cref{eq:p} and \Cref{eq:p^*}, we obtain the importance sampling IRF estimator, denoting $\widehat{\mu}_{IS}(\bm{x})$,   defined below.

\begin{dfn}[IRF: the importance sampling case]\label{prop:IS_{I_n}RF}
    The  IS-IRF estimator at point  $\bm{x}\in \mathcal{X}$ is defined by
    \begin{equation}\label{eq:defIS}
	\widehat{\mu}_{IS}(\bm{x}): =	  g_n( \widehat{\mu}^*(\bm x))=\frac{n_1s_0\widehat{\mu}^*(\bm x)}{n_0s_1(1- \widehat{\mu}^*(\bm x))+n_1s_0\widehat{\mu}^*(\bm x)}\,, \qquad \bm{x}\in \mathcal{X}\,,
\end{equation}
\noindent where $\widehat{\mu}^*$ is the under-sampling IRF estimator as defined in \Cref{est: under_classif}. 

\end{dfn}

The next theorem provides the central limit theorem for $\widehat{\mu}_{IS}(\bm{x})$. 
 
\begin{thm}\label{Two_case_th} Let $\widehat{\mu}_{IS}(\bm x)$ be the IS-IRF at point $\bm{x}\in\mathcal{X}$ defined in \Cref{eq:defIS}.  Under assumptions \textbf{(H1')}-\textbf{(H2')} and \textbf{(G0)} we have
	\begin{align}
		\sqrt{\frac{1}{s_0^2 V_{1,0} ^{s_0,s_1}/n_0+ s_1^2 V_{1,1} ^{s_0,s_1}/n_1}}( \widehat{\mu}_{IS}(\bm{x})- g_n(\mathbb{E}[\widehat{\mu}^\ast(\bm{x})]))
		\overset{d}{\longrightarrow} \mathcal{N}(\, 0, V^R(\bm x) )\,, 
	\end{align}
	$\mbox{as}\, \, n \to \infty$ where  
 $$V^R(\bm x)  := \frac{(p p^{*}(1-p^{*}) (1-p)) ^2}{((1-p)p^{*}(1- \mu(\bm{x}))+p
		(1-p^{*}) \mu(\bm{x}))^4}\,,\qquad \bm x\in \cal X.  $$
\end{thm}

\begin{proof}
    See \Cref{subsec2:proof3}.
\end{proof}
\Cref{Two_case_th} is a close analogous to \Cref{th2} in  \Cref{subsec:mainresult_oneD} for subsampling IRF. 
If moreover the assumption
\begin{itemize}
    \item[\bf(H2'')] {\bf Bias condition}
    
    $$\sqrt{\frac{1}{s_0^2 V_{1,0} ^{s_0,s_1}/n_0+ s_1^2 V_{1,1} ^{s_0,s_1}/n_1}}|g_n(\mathbb{E}[\widehat{\mu}^\ast(\bm{x})])-\mu(\bm x)| \longrightarrow 0\,,\,  n\to \infty\, ,$$
\end{itemize}
holds, then we obtain as a corollary of \Cref{Two_case_th} the following central limit theorem for $\widehat{\mu}_{IS}(\bm{x})$. 
\begin{cor}\label{cor:Two_case_th}	
Let $\widehat{\mu}_{IS}(\bm x)$ be the IS-IRF at point $\bm{x}\in\mathcal{X}$ defined in \Cref{eq:defIS}. Under assumptions \textbf{(H1')}, \textbf{(H2')}, \textbf{(H2'')} and \textbf{(G0)} we have
	\begin{align}
		\sqrt{\frac{1}{s_0^2 V_{1,0} ^{s_0,s_1}/n_0+ s_1^2 V_{1,1} ^{s_0,s_1}/n_1}}( \widehat{\mu}_{IS}(\bm{x})- \mu
			(\bm{x}))
		\overset{d}{\longrightarrow} \mathcal{N}(\, 0, V^R(\bm{x}) )\,, 
	\end{align}
as $ n \to \infty$ where  $V^R(\bm{x}) $ is defined in \Cref{Two_case_th}.
\end{cor}
\begin{proof}
   See \Cref{subsec3:proof4}
\end{proof}
\begin{rem}Consider the case $p^* = 1/2$, corresponding to a balanced scenario in the biased model. For $p$ enough small and assuming $\mu(\bm{x})<1$ so that $p\mu(\bm{x})$ is negligible compared to  $1- \mu(\bm{x})$, the variance term  in \Cref{cor:Two_case_th} for the importance sampling IRF yields 
    $$V^R(\bm x)  \sim  \frac{p^2}{(1- \mu(\bm{x}))^4}\,,\qquad p\to 0.$$
This expression shows that the importance sampling procedure not only removes the asymptotic bias but also systematically improves the limiting variance compared to the under-sampling IRF estimator.
\end{rem}
Up to this point, we have established general theoretical results for subsampling, under-sampling, and importance sampling IRF estimators under suitable assumptions. 
The next section is dedicated to applying these results to the specific case of bagged nearest neighbor classifiers.

\section{Asymptotic properties of the bagged 1-Nearest Neighbours}\label{sec:KNN}


We apply the results of the previous section to the $k$-Nearest Neighbors ($k$-NN) algorithm. 
In the remainder of this work, we focus on the case $k=1$ for simplicity.

\subsection{Main results}
\subsubsection{1-Nearest Neighbours}\label{subsec:1NN_standard}

Let $d$ be a distance on $\mathcal{X}$, and let $E$ denote a standard exponential random variable. Recall that $I_n = \{1, \ldots, n\}$. Suppose $(Z_i)_{i \in I_n}$ is a sample of size $n$, consisting of $\mathcal{X} \times \{0,1\}$-valued random elements, used to estimate the regression function $\mu$ defined in \Cref{eq:px}. We now introduce the following two assumptions, stated for a fixed point $\bm{x} \in \mathcal{X}$.

\begin{itemize}
\item[\bf(D)] {\bf Distance condition:} The   c.d.f.\ $F(x):=\P( d(\bm x,X)\le x)$ is nonatomic.

\item[\bf(X)] {\bf Design condition:} The design $X$ admits a density bounded from below by $m>0$ in a neighborhood of $\bm x \in \cal X$.
\end{itemize}
 
 The 1-Nearest Neighbour (1-NN) estimator $Y_{nn(\bm x)}$ is defined as
\begin{align}\label{def:1NN}
    Y_{nn(\bm x)}=\sum_{i \in I_n }Y_{i} \1(nn(\bm x)=i), \qquad \bm x \in \cal X\,,
\end{align}
where the index of the nearest neighbor is given by
\begin{align}\label{def:nnx}
nn(\bm x)=\argmin_{i \in I_n} \{ \d(\bm x,X_i)\}, \qquad \bm x \in \cal X.
\end{align}
Equivalently, for each $i \in I_n$, we have
\begin{align}\label{def:nnx_equals_i}
\{ nn(\bm x) =i\}=\{ \d(\bm x,X_i)\le  \d(\bm x,X_j), \forall j \in I_n\setminus{\{i\}} \}, \qquad \bm x \in \cal X.
\end{align}
 
 Note that under Condition {\bf (D)}, $\arg \min_{i \in I_n} \{ \d(\bm x,X_i)\}$ exists and is finite so that the 1-NN estimator $Y_{nn(\bm x)}$ in \Cref{def:nnx} is well defined. We now proceed with the following lemma, which gathers  properties that will be used throughout the remainder of the analysis.

\begin{lem}\label{lem:E_ynn}
Let $Y_{nn(\bm x)}$ be the 1-NN estimator at point $\bm x \in \cal X$ as defined in \Cref{def:1NN}. Under Conditions \textbf{(D)}, \textbf{(G1)} and \textbf{(X)} the following two assertions hold 
  \begin{align}\label{eq:Yinconsistent}
    \E[Y_{nn(\bm x)}]=\mu(\bm x)+O(n^{-1/d})\,,\qquad n\to \infty\,
  \end{align}
  and 
  \begin{align*}
    \v(Y_{nn(\bm x)})=\mu(\bm x)(1-\mu(\bm x))+O(n^{-1/d})\,,\qquad n\to \infty\,  .
\end{align*}
\end{lem}
\begin{proof}
    See the detailed proof in \Cref{sub:proof-of-lemma5.3} for completeness.
\end{proof}

  \Cref{lem:E_ynn} shows that the 1-NN estimator $Y_{nn(\bm{x})}$ defined in \Cref{def:1NN} is \textit{asymptotically unbiased}, as the bias vanishes at the rate $O(n^{-1/d})$. However, it remains \emph{inconsistent} when $\mu(\bm{x}) \in (0,1)$, since its variance does not vanish in the limit but converges to $\mu(\bm{x})(1 - \mu(\bm{x})) > 0$ as $s\to \infty$. This residual variance reflects the inherent randomness of the class label at the nearest neighbor and highlights a fundamental limitation of the 1-NN rule in classification
; see also \citet{biau2015lectures} for further discussion. This observation motivates the use of aggregation methods, to reduce variance and improve stability. This is the purpose of the next sections.

\subsubsection{The subsampling bagged 1-NN estimator}\label{subsubsec:subsampling}

We adopt the notation introduced in \Cref{subsubsec:sub_sampling_{I_n}RF} and briefly recall the setup. We work within the original model described in \Cref{subsec: orig_model}. Let $S \subset I_n := \{1, \ldots, n\}$ be an index set of size $|S| = s$, with $1 \leq s \leq n$, and define $\bm{Z}_S := (Z_{i_1}, \ldots, Z_{i_s})$ as a vector in $(\mathcal{X} \times \{0,1\})^s$, obtained by uniformly sampling $s$ observations without replacement from the full dataset $\bm{Z}_{I_n}$.

The  1-NN estimator based on the subsample $\bm{Z}_S$ (called subsampling 1-NN estimator) is then defined as follows

\begin{align}\label{def:sub_1NN}
    Y^S_{nn(\bm x)}=\sum_{i \in S  } Y_{i} \1(nn(\bm x)=i), \qquad \bm x \in \cal X
\end{align}
where for $i \in S$ the index of the nearest neighbor is given by  
\begin{align}\label{def:sub_nnx_equals_i}
\{ nn(\bm x) =i\}=\{ \d(\bm x,X_i)\le  \d(\bm x,X_j), \forall j \in S\setminus{\{ i\} } \}, \qquad \bm x \in \cal X.
\end{align}

Throughout \Cref{subsubsec:subsampling}, it is worth noting that all results and statements from \Cref{subsec:1NN_standard} remain valid when replacing the full sample size $n$ with the subsample size $s$. In particular, the subsampling 1-NN estimator $Y^S_{nn(\bm{x})}$ in \Cref{def:sub_1NN} is well defined and remains an asymptotically unbiased estimator of $\mu(\bm{x})$, although it is still inconsistent when $\mu(\bm{x}) \in (0,1)$. 

We now define the bagged 1-NN estimator obtained by aggregating the subsampling-based estimators introduced in \Cref{def:sub_1NN} and \Cref{def:sub_nnx_equals_i}.

\begin{dfn}[bagged 1-NN: the subsampling case]\label{def:Bagged1NN} 
The subsampling bagged 1-NN estimator at point  $\bm{x}\in \mathcal{X}$ is defined as
	
 \begin{align}
 \label{def:bagged_classif}
		\widehat{\mu}_{NN}^s(\bm x):= \binom{n}{s}^{-1} \sum_{S \subset I_n, |S|=s}  Y^S_{nn(\bm x)},
  \end{align}
where $Y^S_{nn(\bm x)}$ is  the subsampling 1-NN estimator based on $\bm Z_S$ as defined in \Cref{def:sub_1NN} and 
\Cref{def:sub_nnx_equals_i}.
\end{dfn}
Note that both $Y^S_{nn(\bm{x})}$ and $\widehat{\mu}_{NN}^s(\bm{x})$ are designed to estimate the regression function $\mu(\bm{x})$ defined in \Cref{eq:px}, corresponding to the original model described in \Cref{subsec: orig_model}. 

To establish a central limit theorem for $\widehat{\mu}_{NN}^s$, we begin by analyzing an approximation of its variance. Specifically, we consider $V_1^s := \operatorname{Var}(T_1^s)$, where $T_1^s := \mathbb{E}[Y_{nn(\bm{x})}^{I_s} \mid Z_1]$ and $I_s = \{1, \ldots, s\}$.

\begin{prp}\label{Prop:V1^s}
Under Conditions \textbf{(D)}, \textbf{(X)} and \textbf{(G1)} we have
$$
V_1^s=\dfrac{\mu(\bm x)(1-\mu(\bm x))}{2s}+O(s^{-(1+1/d)})\,, \qquad s\to \infty.
  $$
 
\end{prp}
\begin{proof}
    See \Cref{proof:Prop:V1^s}.
\end{proof}

This result provides an asymptotic approximation of the variance $V_1^s$ associated with the subsampling 1-NN estimator. The leading term, $\mu(\bm{x})(1 - \mu(\bm{x}))/ (2s)$, indicates that the variance decreases at rate $1/s$ as the subsample size $s$ increases. This reduction stems from the randomization introduced by subsampling, which provides a smoothing effect.


We are now ready to apply \Cref{th2} on the subsampling bagged 1-NN estimator defined in \Cref{def:Bagged1NN}.

\begin{cor}\label{Coro: TCL}Let $\widehat \mu_{NN}^s(\bm x)$ be the subsampling bagged 1-NN estimator at point $\bm x \in \cal X$ as defined in \Cref{def:bagged_classif}.
Assume Conditions \textbf{(D)}, \textbf{(X)} and \textbf{(G1)} hold.  If moreover $s/n\to 0$ and $n/s^{1+2/d}\to 0 $ as $n\to \infty$  we have
$$
\sqrt{\dfrac {2n}s}( \widehat \mu_{NN}^s(\bm x)-\mu(\bm x))\overset{d}{\longrightarrow} \mathcal N(0,\mu(\bm x)(1-\mu(\bm x)))\,,\qquad n\to \infty.
$$
\end{cor}
\begin{proof}
    See \Cref{sub: proof-of-cor5.7}.
\end{proof}

Note that the inverse of the subsample size $s^{-1}$ plays a role analogous to the bandwidth parameter in kernel regression methods. The upper bound $s = O(n^{d/(d+2)})$ corresponds to the minimax-optimal choice for achieving the best trade-off between bias and variance under Lipschitz assumptions on $\mu$.

Our result offers a detailed asymptotic analysis of the subsampling-based bagged 1-NN estimator, which retains favorable computational and statistical properties. In particular, it shares similarities with the denoised 1-NN approach proposed by \citet{xue2018achieving}, both in terms of variance reduction and efficiency.

\subsubsection{The under-sampling bagged 1-NN estimator}\label{subsubsec:under}

We adopt the notation from \Cref{subsubsec:under_sampling_{I_n}RF} and briefly recall the setup. Given the under-sample $\bm{Z}_{S_0 \cup S_1} := (Z^0_{i_1}, \ldots, Z^0_{i_{s_0}}, Z^1_{i_1}, \ldots, Z^1_{i_{s_1}})$, the under-sampling 1-NN estimator constructed from $\bm{Z}_{S_0 \cup S_1}$ is defined as follows

\begin{align}\label{def:1NN_under}
Y^{S_0 \cup S_1}_{nn(\bm x) }
&=\sum_{i \in S_1} \1(nn(\bm x)=i), \qquad \bm x \in \cal X,
\end{align}
   where for $i \in S_1$, the index of the nearest neighbor is given by 
\begin{align}\label{def:nnx_under}
\{nn(\bm x) =i\}&= \left\{  \d(\bm x,X_{i}^1)<\d(\bm x,X_{j}^1), \forall j \in S_1\setminus{\{i\}} \right\}\, \qquad \bm x \in \cal X
\end{align}
Here, $X^j$ denotes a covariate drawn from class $j$, with associated distribution function $F_j(x) := \mathbb{P}(d(\bm{x}, X^j) \le x)$. By definition, we have 
\[
F_j(x) = \mathbb{P}(d(\bm{x}, X) \le x \mid Y = j).
\]
Under Condition \textbf{(D)}, both $F_0$ and $F_1$ are nonatomic. In line with the setting of \Cref{subsubsec:subsampling}, we introduce the following assumption.

 \begin{itemize}
\item[\bf(X')] {\bf Design condition:} The designs $X^i$, $i=0,1$ admit  densities bounded from below by $m>0$  in a neighborhood of $\bm x \in \cal X$.
\end{itemize}
We give below a definition of the under sampling bagged 1-NN based on the under-sampling  estimator given in \Cref{def:1NN_under}.
\begin{dfn}[bagged 1-NN: the under-sampling case]\label{def: Bagged_1NN_under}
	The  under-sampling bagged 1-NN estimator at point  $\bm{x}\in \mathcal{X}$ is defined as
	\begin{align}\label{est:bagged_sans_remp}
		\widehat{\mu}_{NN}^{s_0,s_1}(\bm{x}):= \left(\binom{n_0}{s_0} \binom{n_1}{s_1} \right)^{-1} \sum_{S_0\subset I_{n_0}} \sum_{S_1\subset I_{n_1}} Y^{S_0 \cup S_1}_{nn(\bm x) },
	\end{align}
with $|S_0|=s_0$ and $|S_1|=s_1$ and $Y^{S_0 \cup S_1}_{nn(\bm x) }$ is the the under-sampling 1-NN estimator based on $\bm{Z}_{S_0 \cup S_1}$ as defined in \Cref{def:1NN_under} and \Cref{def:nnx_under}.
\end{dfn}
In contrast to the subsampling setting, where both $Y^S_{nn(\bm{x})}$ and $\widehat{\mu}_{NN}^s(\bm{x})$ estimate the regression function $\mu(\bm{x})$ from the original model, the under-sampling estimators $Y^{S_0 \cup S_1}_{nn(\bm{x})}$ and $\widehat{\mu}_{NN}^{s_0,s_1}(\bm{x})$  estimate $\mu^*(\bm{x})$, the regression function associated with the biased   model introduced in \Cref{subsec: biased_moedel} and defined in \Cref{def:mustar}.

As in \Cref{subsubsec:subsampling}, in order to establish a central limit theorem for $\widehat{\mu}_{NN}^{s_0,s_1}$, we begin by analyzing the conditional expectations
\[
T_{1,j}^{s_0,s_1} := \mathbb{E}\left[Y^{I_{s_0} \cup I_{s_1}}_{nn(\bm{x})} \mid Z_1^j\right], \qquad j = 0,1,
\]
where the sample $\bm{Z}_{I_{s_0} \cup I_{s_1}} = (Z^0_1, \ldots, Z^0_{s_0}, Z^1_1, \ldots, Z^1_{s_1})$ is drawn from the biased model.

\begin{prp}\label{Prop:E[EynnZ1]} 
Assume Conditions \textbf{(D)}, \textbf{(X')} and \textbf{(G1)} hold. If moreover $|s_1-p^*s|=O(s^{-(1+1/d)})$ as $s\to \infty$, we have
\begin{enumerate}
    \item[$i)$] $
\E[T_{1,0}^{s_0,s_1}]=\E[T_{1,1}^{s_0,s_1}]=\E[Y^{S_0 \cup S_1}_{nn(\bm x)}] = \mu^*(\bm x)\, + O(s^{-1/d});
$

 \item[$ii)$] $\v(Y^{S_0 \cup S_1}_{nn(\bm x)}) = \mu^*(\bm x)(1-\mu^*(\bm x))\, + O(s^{-1/d}).$
\end{enumerate}

\end{prp}
\begin{proof}
    See \Cref{sub: proof-of-Prop5.12}.
\end{proof}

   We observe the same phenomenon as in the subsampling 1-NN setting: according to \Cref{Prop:E[EynnZ1]}, the under-sampling 1-NN estimator $Y^{S_0 \cup S_1}_{nn(\bm{x})}$ is asymptotically unbiased for $\mu^*(\bm{x})$. However, it remains inconsistent when $\mu^*(\bm{x}) \in (0,1)$, as its variance converges to a strictly positive limit $\mu^*(\bm{x})$   as $s \to \infty$.
   
\begin{rem}
\noindent
This additional assumption $
|s_1 - p^* s| = O(s^{-(1 + 1/d)}) $
ensures that the proportion of class-$1$ samples in the under-sample remains sufficiently close to its target value $p^*$ at an appropriate rate. This technical assumption is automatically satisfied when $p^*=0.5$ and $s_1=s_0$, which is the common choice usually made for rebalancing the classes.

\end{rem}

The next proposition extends \Cref{Prop:V1^s} to the under-sampling case and provides an approximation of the variances $V_{1,j}^{s_0,s_1}= \v(T_{1,j}^{s_0,s_1})$ for $j=0,1$.

\begin{prp}\label{prop:1NN_V10^s}
   Assume Conditions \textbf{(D)}, \textbf{(X')} and \textbf{(G1)} hold. If moreover $|s_1-p^*s|=O(s^{-(1+1/d)})$,  as $s\to \infty$, we have
  \begin{align*}
V_{1,0}^{s_0,s_1}&=\E[(T_{1,0}^{s_0,s_1}-\E[T_{1,0}^{s_0,s_1}])^2]= \dfrac{\mu^*(\bm x)^2(1-\mu^*(\bm x))}{2(1-p^*)s}+O(s^{-(1+1/d)})\,,\\
V_{1,1}^{s_0,s_1}&=\E[(T_{1,1}^{s_0,s_1}-\E[T_{1,1}^{s_0,s_1}])^2]= \dfrac{\mu^*(\bm x)(1-\mu^*(\bm x))^2}{2p^*s}+O(s^{-(1+1/d)})\,.
\end{align*}
\end{prp}

\begin{proof}
    See \Cref{subsec:proof_prop:1NN_V10^s}
\end{proof}

Compared to \Cref{Prop:V1^s}, the denominators in the leading terms involve the class proportions $p^*$ and $1 - p^*$, indicating that imbalance in the biased model directly affects the variance of the estimator.

We are now ready to apply our \Cref{{th: mu_start_TCL}} on the under-sampling bagged  1-NN estimator.

\begin{cor}\label{cor:TCL-under_1NN}Let $\widehat \mu_{NN}^{s_0,s_1}(\bm x)$ be the under-sampling bagged  1-NN estimator at point $\bm x\in \cal X$ defined in \Cref{est:bagged_sans_remp}.
Assume Conditions \textbf{(D)}, \textbf{(X')} and \textbf{(G1)} hold and $|s_1-p^*s|=O(s^{-(1+1/d)})$. If moreover $s/n\to 0$ and $n/s^{1+2/d}\to 0 $  as $n\to \infty$, we have
$$
\sqrt{\dfrac {2n}s}(\widehat \mu_{NN}^{s_0,s_1}(\bm x)-\mu^*(\bm x))
\overset{d}{\longrightarrow} 
\mathcal N\Big(0,V^{US}_{NN}(\bm{x})\Big)\,
$$
as $n\to \infty\,$ where 
$$V^{US}_{NN}(\bm{x})=\mu^{*}(\bm{x}) (1-\mu^{*}(\bm{x}))\Big(\dfrac{(1-p^*) \mu^{*}(\bm{x})}{(1-p)} 
 + \dfrac{p^*(1-\mu^{*}(\bm{x}))}{p}\Big)\,,\qquad \bm x\in {\cal X}.$$
\end{cor}
\begin{proof}
    See \Cref{subsec:cor:TCL-under_1NN}
\end{proof}
\Cref{cor:TCL-under_1NN}  demonstrates that the under-sampling bagged 1-NN estimator is inconsistent. Next we apply the odds ratio formula \Cref{eq:IS} to debias it.

\subsubsection{The importance sampling bagged 1-NN estimator}

We now give a definition of the importance sampling bagged 1-NN estimator (IS 1-NN) based upon importance sampling IRF defined in \Cref{eq:defIS}. 
\begin{dfn}[IS 1-NN]\label{def:IS.1NN}
    The importance  sampling bagged 1-NN estimator at point  $\bm{x}\in \mathcal{X}$ is defined as
    \begin{equation}\label{eq:IS.1NN}
	\widehat \mu_{NN}^{IS}(\bm x): =\frac{n_1s_0\widehat \mu_{NN}^{s_0,s_1}(\bm x)}{n_0s_1(1- \widehat \mu_{NN}^{s_0,s_1}(\bm x))+n_1s_0\widehat \mu_{NN}^{s_0,s_1}(\bm x)}\,, \qquad \bm{x}\in \mathcal{X}\,,
\end{equation}
\noindent where $\widehat \mu_{NN}^{s_0,s_1}(\bm x)$ is the under-sampling bagged 1-NN estimator  as defined in \Cref{est:bagged_sans_remp}.
\end{dfn}
By an application of \Cref{cor:Two_case_th} we obtain the asymptotic normality of the IS 1-NN estimator:
\begin{cor}\label{cor:TCL-IS_1NN}
  Let $\widehat \mu_{NN}^{IS}(\bm x)$ be the IS 1-NN estimator at point $\bm x\in \cal X$ defined in \Cref{def:IS.1NN}. Assume Conditions \textbf{(D)}, \textbf{(X')} and \textbf{(G1)} hold and $|s_1-p^*s|=O(s^{-(1+1/d)})$. If moreover $s/n\to 0$ and $n/s^{1+2/d}\to 0 $ as $n\to \infty$,  we have  
  \begin{align*}
    \sqrt{\dfrac {2n}s}\,(\widehat \mu_{NN}^{IS}(\bm x)-\mu(\bm{x}) )\overset{d}{\longrightarrow} \mathcal{N}\big(\, 0, V_{NN}^{IS}(\bm x)\big)
\end{align*}
as $n\to \infty\,$, where 
$$
V_{NN}^{IS}(\bm x)=\frac{(p(1-p))^3(p^{*}(1-p^{*}))^4\mu(\bm{x})(1- \mu(\bm{x}))}{((1-p)p^{*}(1- \mu(\bm{x}))+p
		(1-p^{*}) \mu(\bm{x}))^7}\,,\qquad \bm x\in \cal X\,.
$$
\end{cor}
\begin{proof}
    See \Cref{subsec:cor:TCL-IS_1NN}
\end{proof}
\Cref{cor:TCL-IS_1NN} shows that the IS 1-NN estimator attains the same near-minimax convergence rate as the subsampling-based bagged 1-NN estimator. In what follows, we compute the asymptotic variances of both estimators in relevant cases, in order to compare their respective asymptotic efficiencies.

\subsection{Discussion of theoretical results}
We discuss the obtained results in \Cref{sec:results} and \Cref{sec:KNN} for the two specific situations $p^*=p$ and $p^*=0.5$.
\begin{rem}\label{rem:cor:TCL-under_1NN}
In the special case where $p^* = p$, we have $\mu = \mu^*$, and \Cref{cor:Two_case_th} implies that $V^R(\bm{x}) = 1$ for all $\bm{x} \in \mathcal{X}$. As a result, the under-sampling and IS IRF estimators, defined respectively in \Cref{prop: case_two} and \Cref{prop:IS_{I_n}RF}, share both the same convergence rate and the same asymptotic variance as the subsampling IRF estimator. Consequently, when $s_0 + s_1 = s$, the subsampling bagged 1-NN and IS bagged 1-NN estimators also exhibit identical rates of convergence and asymptotic variances.
\end{rem}

\begin{rem}\label{rem:cor:TCL-IS_1NN}
Specializing to 1-NN estimators, we can compare the limiting variance of the subsampling bagged 1-NN, given by $V^{SB}_{NN}(\bm{x}) = \mu(\bm{x})(1 - \mu(\bm{x}))$, with that of the IS bagged 1-NN estimator. For $p^*=0.5$ \Cref{cor:TCL-IS_1NN} yields 
$$  \sqrt{\dfrac {2n}s}\,(\widehat \mu_{NN}^{IS}(\bm x)-\mu(\bm{x}))\overset{d}{\longrightarrow} \mathcal{N}\Big(\, 0, V^{IS}_{NN}(\bm x) \Big),\,\qquad n\to \infty\,,\qquad \bm x\in\mathcal X, $$
with the asymptotic variance, for  $\mu(\bm x)<1$, $$V^{IS}_{NN} (\bm x)\sim \frac{p^3\mu(\bm{x})}{2(1- \mu(\bm{x}))^6}\,,\qquad p\to 0.$$
Hence, the ratio of the variance limits
\[
\frac{V^{IS}_{NN}(\bm{x})}{V^{SB}_{NN}(\bm{x})} \sim \frac{p^3}{2(1 - \mu(\bm{x}))^7}\,,\qquad p\to 0\,,
\]
is always in favor of the IS 1-NN estimator when $p$ is sufficiently small and $\mu(\bm{x}) < 1$. Interestingly, the fact that $V^{IS}_{NN}(\bm{x}) \to 0$ as $p \to 0$ suggests that the IS 1-NN estimator may exhibit a faster than $O(n/s)$ convergence rate in severe imbalance settings, where $\widehat{p} = n_1 / n \to 0$ as $n \to \infty$.

\end{rem}

    \section{Numerical Illustrations}\label{sec:illustration}

The objective of this section is to show by means of numerical illustrations that the inference procedure suggested in the paper and the theoretical results are empirically relevant. We  illustrate \Cref{th2}, \Cref{Coro: TCL}, \Cref{cor:TCL-under_1NN} and \Cref{cor:TCL-IS_1NN} building estimates of  the subsampling IRF $\widehat{\mu}^s(\bm{x})$, the subsampling bagged 1-NN $\widehat \mu_{NN}^s(\bm{x})$, the  under-sampling  bagged 1-NN $\widehat \mu_{NN}^{s_0,s_1}(\bm{x})$ and the importance sampling bagged 1-NN  $\widehat\mu_{NN}^{IS}(\bm{x})$  estimators defined respectively in \Cref{def:subs_classif}, \Cref{def:bagged_classif}, \Cref{est:bagged_sans_remp} and \Cref{eq:IS.1NN}. 

We consider the imbalance scenario in the dataset (see \Cref{def:ir}) on two  setups as in \citet{o2019random},  which we define as follows:\\

\noindent  \textbf{Setup 1.} The dataset is marginally imbalanced meaning that $\mu(\bm x)\ll 1/2$  for all $\bm x\in \mathcal{X}$.\\
 \noindent \textbf{Setup 2.}   The dataset is conditionally imbalanced  meaning  that there exists a set $L  \subset \mathcal{X}$ with nonzero probability $\mathbb{P}(X\in L  )>0$, such that $\mu(\bm x)\gg 1/2$ for $\bm x\in L$ and $\mu(\bm x)\ll 1/2$ for $\bm x\notin L$.\\

The evaluation scheme is based on $100$ Monte Carlo repetitions. In both setups, we simulate i.i.d.\ uniform covariates $X_i \sim  \mathcal{U}([-1, 1]^2)$ and $Y_i \in \{0,1\}$, with the sample sizes  $n=\{400,800,1200,1600,2000\}$. 


We use a test dataset composed of a grid of $100^2$ points  $\bm x$  over the  domain  $[-1,1]^2$. The integrated squared bias is estimated as follows $i)$ we compute  the predictions of the different models for each point in the test set $ii)$ for each test data point, the bias is calculated as the difference between the prediction and the true value then we sum the squared bias over the test data set. 

Similarly, to estimate the integrated  squared variance, we calculate the variance of the model predictions at each test point and then sum these variances over the entire test grid. Note that in both setups below, $\mu(\bm x)$ is a piecewise constant function. \\

\noindent In \textbf{Setup 1}, we introduce a dependency between \( Y \) and the covariates. Specifically, if $X_1$ and $X_2$ share the same sign, then $Y=1$ occurs with a higher probability of $ \frac{3}{4} \times p<1/2$. Conversely, if $X_1$ and $X_2$ have opposite signs, the probability of $Y=1$ drops to $\frac{1}{4} \times p<1/2$.

\noindent In \textbf{Setup 2}, we define a small region where class 1 is more likely to appear. Specifically, these regions include a small square in the top-right quadrant, where \( X_1 > 1 - \frac{1}{\sqrt{2}} \) and \( X_2 > 1 - \frac{1}{\sqrt{2}} \), as well as a small square in the bottom-left quadrant, where \( X_1 < -1 + \frac{1}{\sqrt{2}} \) and \( X_2 < -1 + \frac{1}{\sqrt{2}} \). In these areas, the function \( \mu(\bm{x}) \) is constant larger than $1/2$ whereas elsewhere it is close to $0$.\\

 In each of these two setups, we consider Imbalanced  Scenario \textbf{(ImB-Sc)}, which corresponds to   $IR\gg 1$, and we let the proportion of the target variable $(Y=0)$ is very much larger than the proportion of $(Y=1)$, namely $p=\P(Y=1)=10\%$ in both setups.

 We implemented our simulation in \texttt{R}  using    \texttt{randomForest  R-package} (considering default parameters), see \citet{liaw2002classification} for applying IRF estimators and \texttt{class  R-package} to apply the procedure for all the bagged 1-NN estimators. 

Throughout all the  simulations, the subsample size $s$ is selected using a  10-fold cross-validation method.  For each iteration, the optimal $s$ is chosen from the range $s=n^\alpha$, where  $\alpha\in \{0.3,0.4,0.5,0.6,0.7\}$.  We set the number of  subsamples $B$, drawn without replacement, equal to the subsample size, fixing $\texttt{ntree}= n$,  $\texttt{ntree}= n_0$ or $\texttt{ntree}= n_1$ depending on the sampling scheme. Some discussion on the choice or determination of number of subsamples was made by \citet{mentch2016quantifying}, \citet{biau2010rate}. \\

\Cref{alg:sub} provides the procedure that produces  subsampling bagged 1-NN predictions  at each point in feature space.\\

\begin{algorithm}
\caption{Subsampling bagged 1-NN}\label{alg:sub}
\begin{algorithmic}[1]
    \State Load training set \, $Z_i=\{(X_i,Y_i)\}_{1\leq i\leq n}$, with feature $X_i$ and response $Y_i$
    \For{$b$ in $1$ to $B$}
        \State Take subsample of size $s$ from training set without replacement 
        \State Build 1-NN using subsample
        \State Use 1-NN to predict at $\bm{x}$
    \EndFor
    \State Average the $B$ predictions to get  $\widehat{\mu}_{NN}^{s}(\bm x)$
\end{algorithmic}
\end{algorithm}

 \Cref{alg: under} details  the procedure for constructing an under-sampling bagged  1-NN and applying importance sampling debiasing procedure.  As in \Cref{alg:sub}, we  use subsamples of size $s$ determined  through 10 fold validation. Note that a quite similar algorithm called double-sample trees was developed by \citet{wager2018estimation}.
 \vspace{1cm}
\begin{algorithm}
\caption{Importance sampling bagged 1-NN}
\begin{algorithmic}[1]\label{alg: under}
\State Load training set \, $Z_i=\{(X_i,Y_i)\}_{1\leq i\leq n}$, with feature $X_i$ and response $Y_i$ and divide it into two disjoint datasets, the 0-class and the 1-class of size $n_0$ and $n_1$, respectively.
\For{$b$ in $1$ to $B$}
    \State Draw two random subsamples of sizes  $s_1=\min(s,n_1)$ from the $1$-class  without replacement (resp, $s_0=\min(s,n_0)$ from the $0$-class) and concatenate them
    \State Build 1-NN using subsample
    \State Use 1-NN to predict at $\bm x$
    \EndFor
    \State Average the $B$ predictions to get  $ \widehat{\mu}_{NN}^{s_0,s_1}(\bm x)$ 
    \State Retrieve $\widehat{\mu}^{IS}_{NN}(\bm x)$ by applying  $g_n$ function to $ \widehat{\mu}_{NN}^{s_0,s_1}(\bm x)$, with $g_n$ as in \cref{eq:defIS}
\end{algorithmic}
\end{algorithm} 
\vspace{1.5cm}
  In- this paragraph, and for the sake of simplicity, we adopt the following notations to refer to the estimators shown in the figures: 1NN.sub for the subsampling bagged 1-NN estimator, 1NN.under for the under-sampling bagged 1-NN estimator,  IS-1NN for the importance sampling bagged 1-NN estimator and IRF for Subsampling IRFs.
\subsection{Experiment Setup 1}

 \Cref{fig: setup1_sc3} displays the bias, variance, and Mean Integrated Squared Error (MISE) as functions of the logarithm of the sample size. The results highlight that the 1NN.under estimator is significantly affected in terms of both bias and variance. Additionally, the IRF estimator is also affected by bias and variance, but it achieves lower bias and variance compared to the 1NN.under estimator. We also observe that the 1NN.sub and IS-1NN estimators have similar MISE performance. The IS-1NN estimator exhibits  the lowest bias among all methods, including IRF, 1NN.under, and 1NN.sub. However, 1NN.sub estimator performs slightly better than IS-1NN in terms of variance. Overall, the IS-1NN estimator successfully corrects the bias of the under-sampling 1-NN estimator, though its performance can be somewhat unstable, likely due to too small number of subsamples.

\begin{figure}[H]
     {\centering                                       
     \begin{subfigure}[b]{0.32\textwidth}
         \centering
         \includegraphics[width=\textwidth]{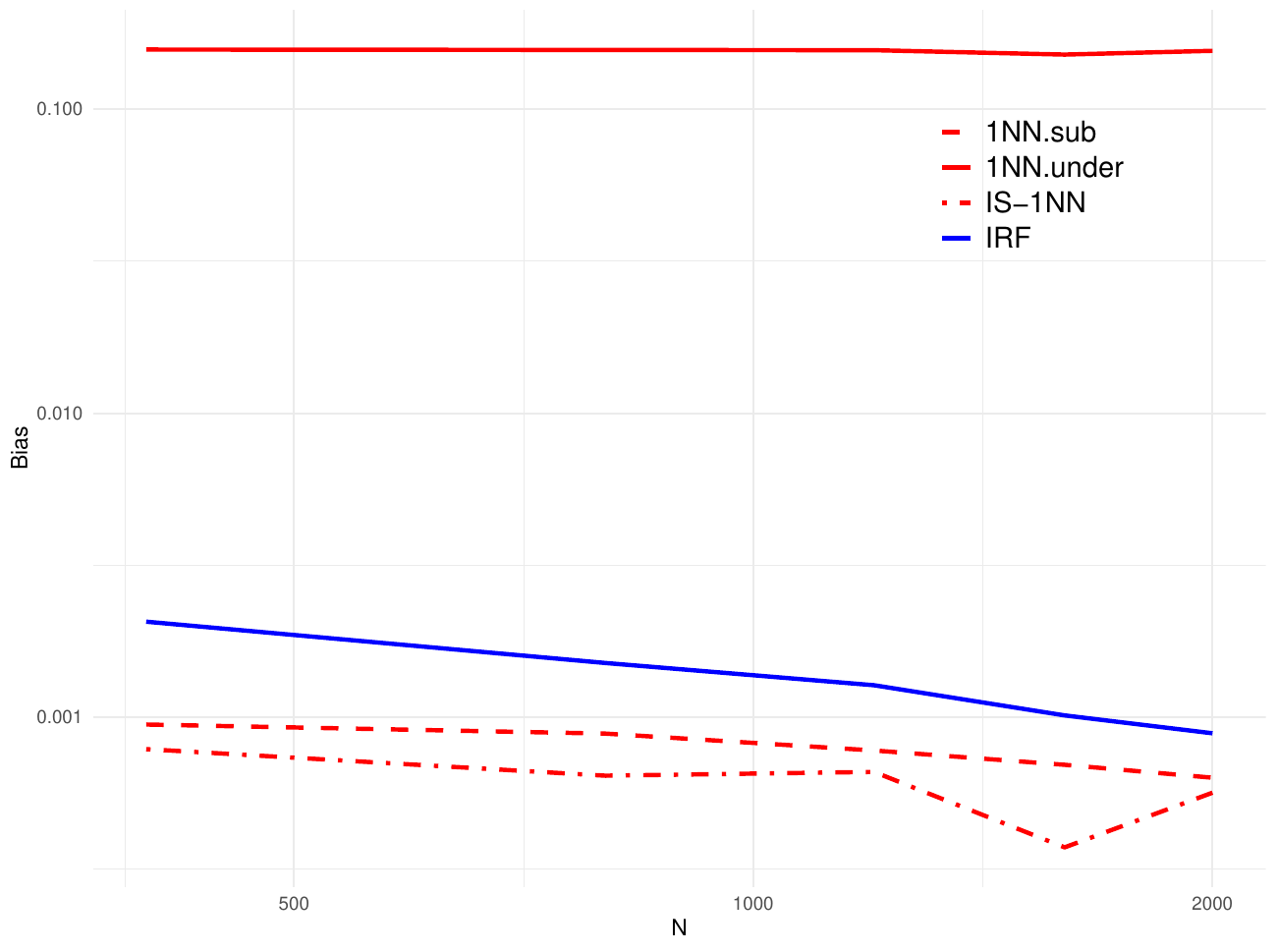}
         \caption{Bias}
         \label{subf:bias}
     \end{subfigure}
     \hfill
     \begin{subfigure}[b]{0.32\textwidth}
         \centering
         \includegraphics[width=\textwidth]{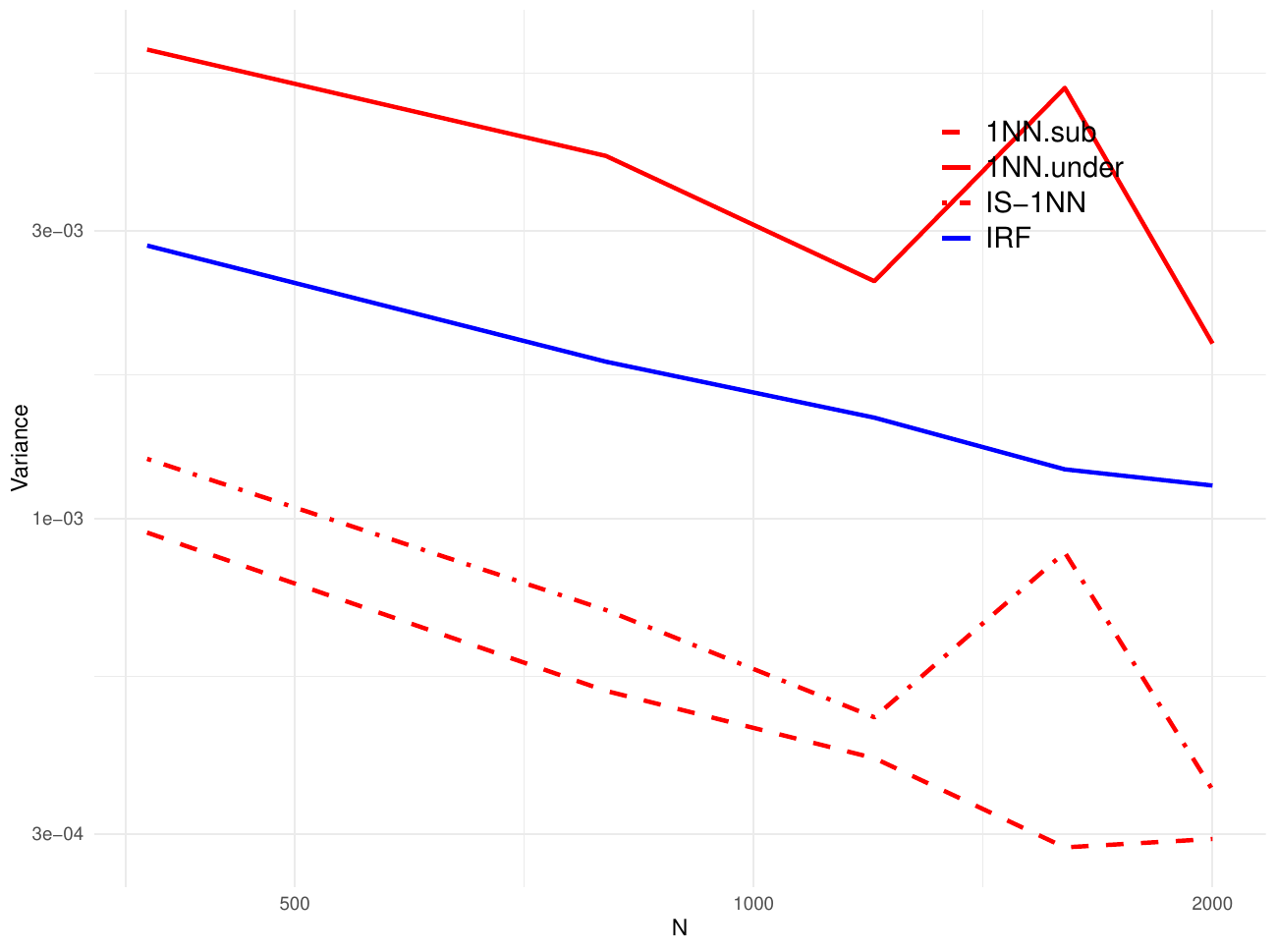}
         \caption{Variance}
         \label{subf:var}
     \end{subfigure}
     \begin{subfigure}[b]{0.32\textwidth}
         \centering
        \includegraphics[width=\textwidth]{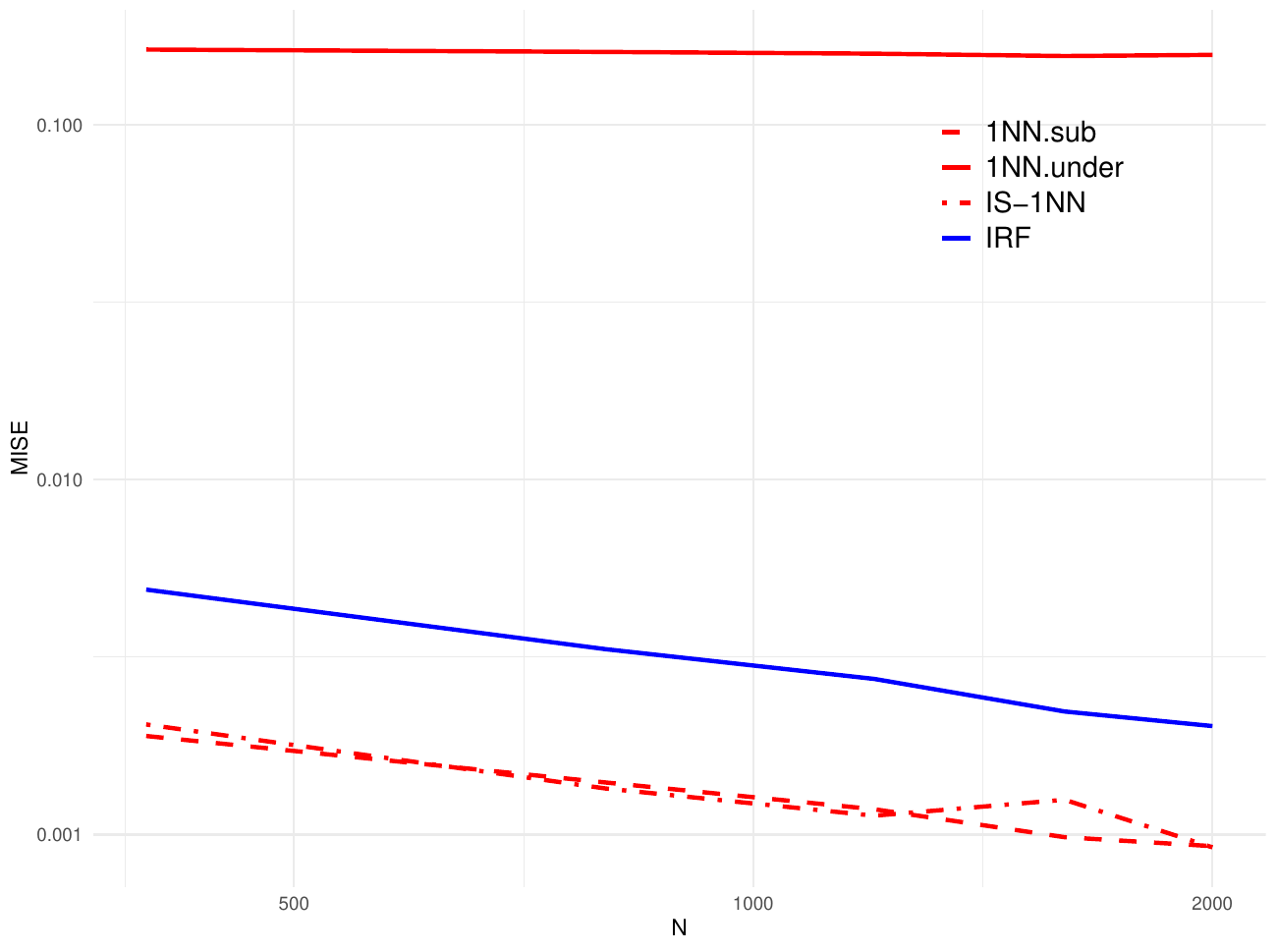}
         \caption{MISE }
         \label{subf:mise_sc1}
     \end{subfigure}
        \caption{Comparison of the MISE trend with that of the bias-variance decomposition w.r.t logarithm of the samples size for Setup 1.}
        \label{fig: setup1_sc3}}
\end{figure}

\subsection{Experiment Setup 2}

 \Cref{fig:two_Sc3} presents the results in Setup 2. Similar to \Cref{fig: setup1_sc3}, we observe that the 1NN.under  estimator is more impacted in terms of both bias and variance. Contrary to \Cref{fig: setup1_sc3}, we notice an opposite trend, that is, 1NN.sub estimator  performs slightly better than IS-1NN estimator in terms of bias, and IS-1NN estimator has a lower variance than 1NN.sub. As before, we find that the IS-1NN estimator perfectly corrects the bias of the 1NN.under estimator and appears to be stable. Furthermore, IRF estimator   outperforms the three other methods.

\begin{figure}[H]
     {\centering                                       
     \begin{subfigure}[b]{0.32\textwidth}
         \centering
         \includegraphics[width=\textwidth]{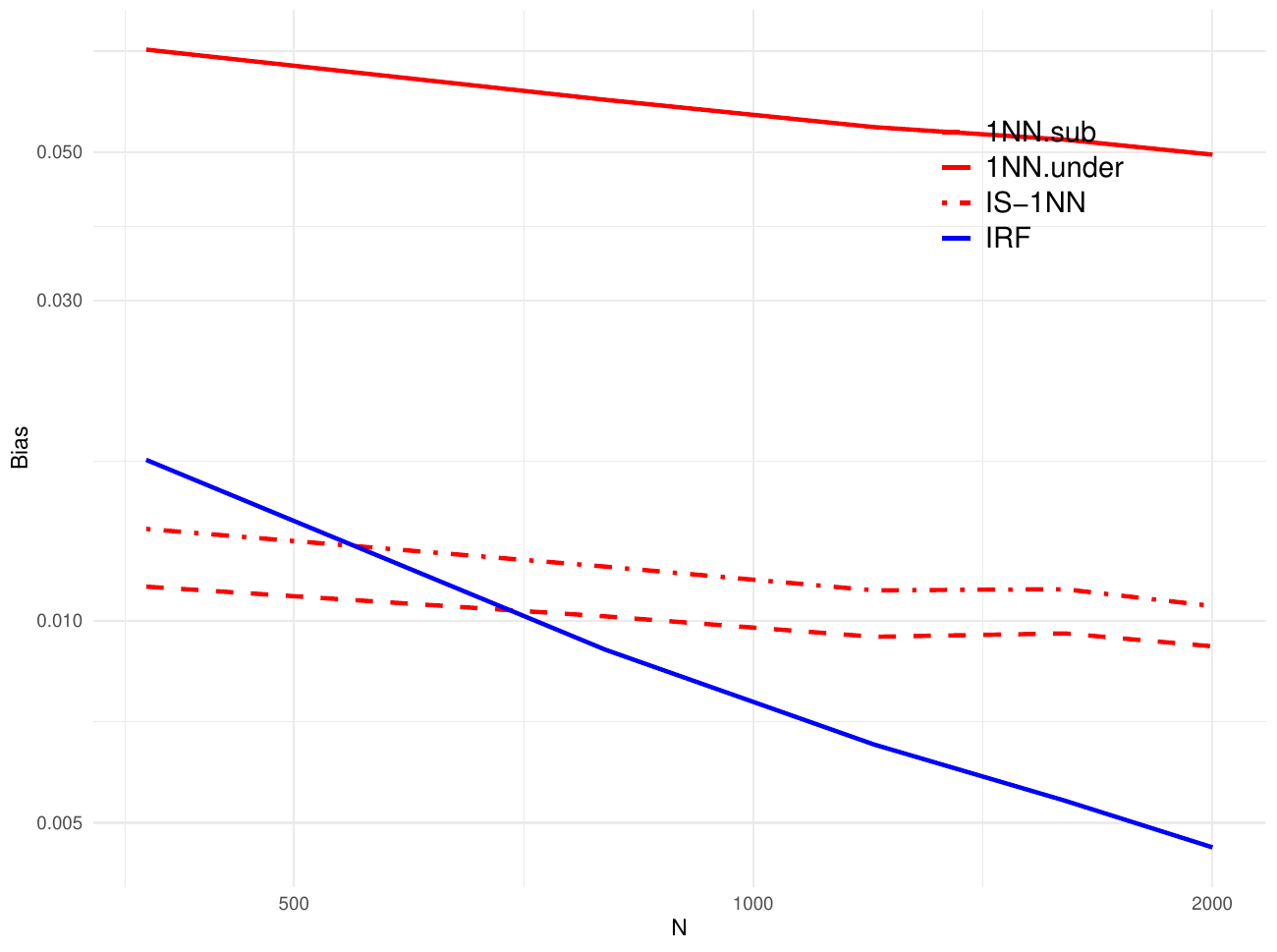}
         \caption{Bias}
         \label{subf:bias_setup2}
     \end{subfigure}
     \hfill
     \begin{subfigure}[b]{0.32\textwidth}
         \centering
         \includegraphics[width=\textwidth]{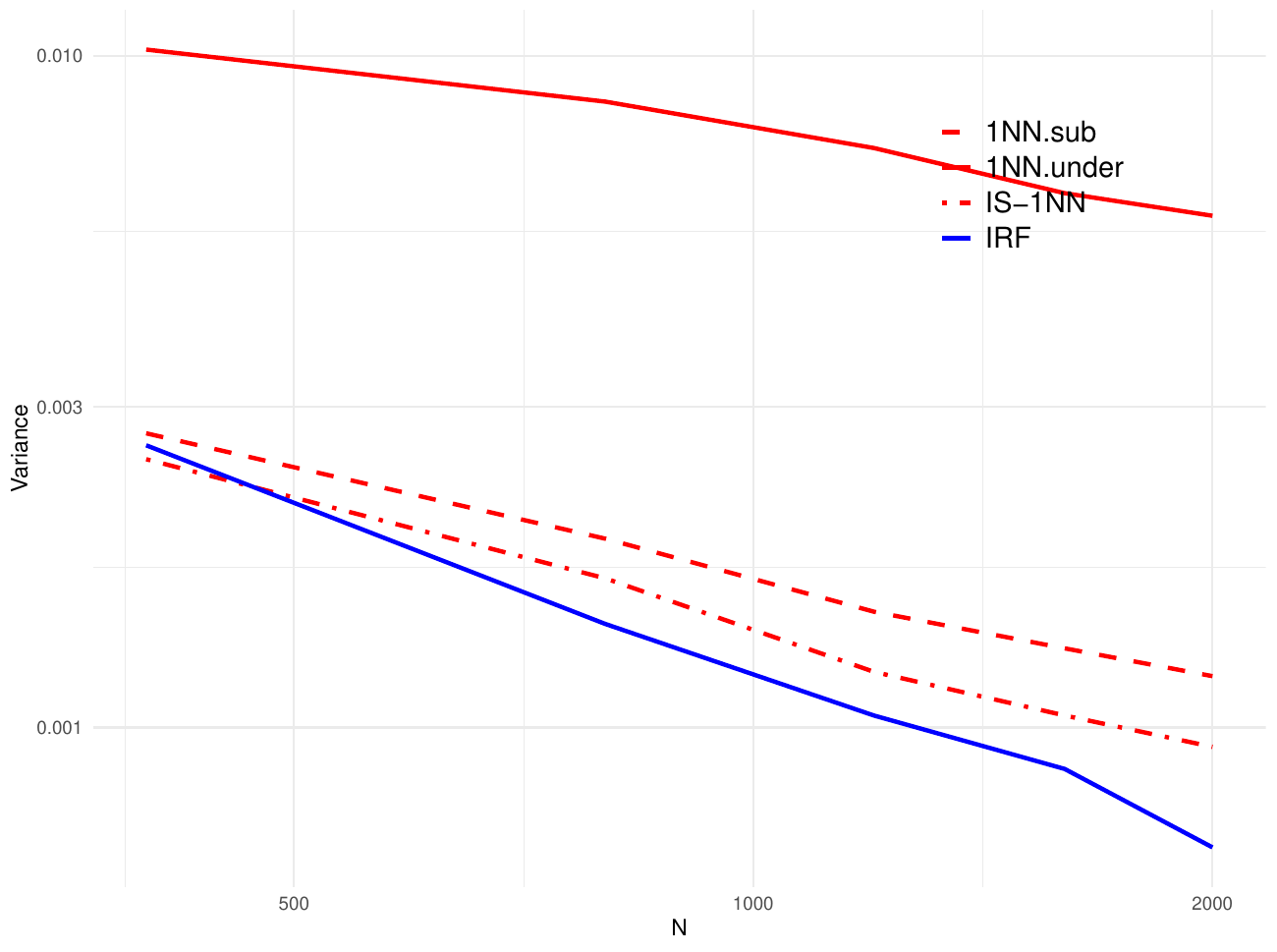}
         \caption{Variance}
         \label{subf:var_setup2}
     \end{subfigure}
     \begin{subfigure}[b]{0.32\textwidth}
         \centering
        \includegraphics[width=\textwidth]{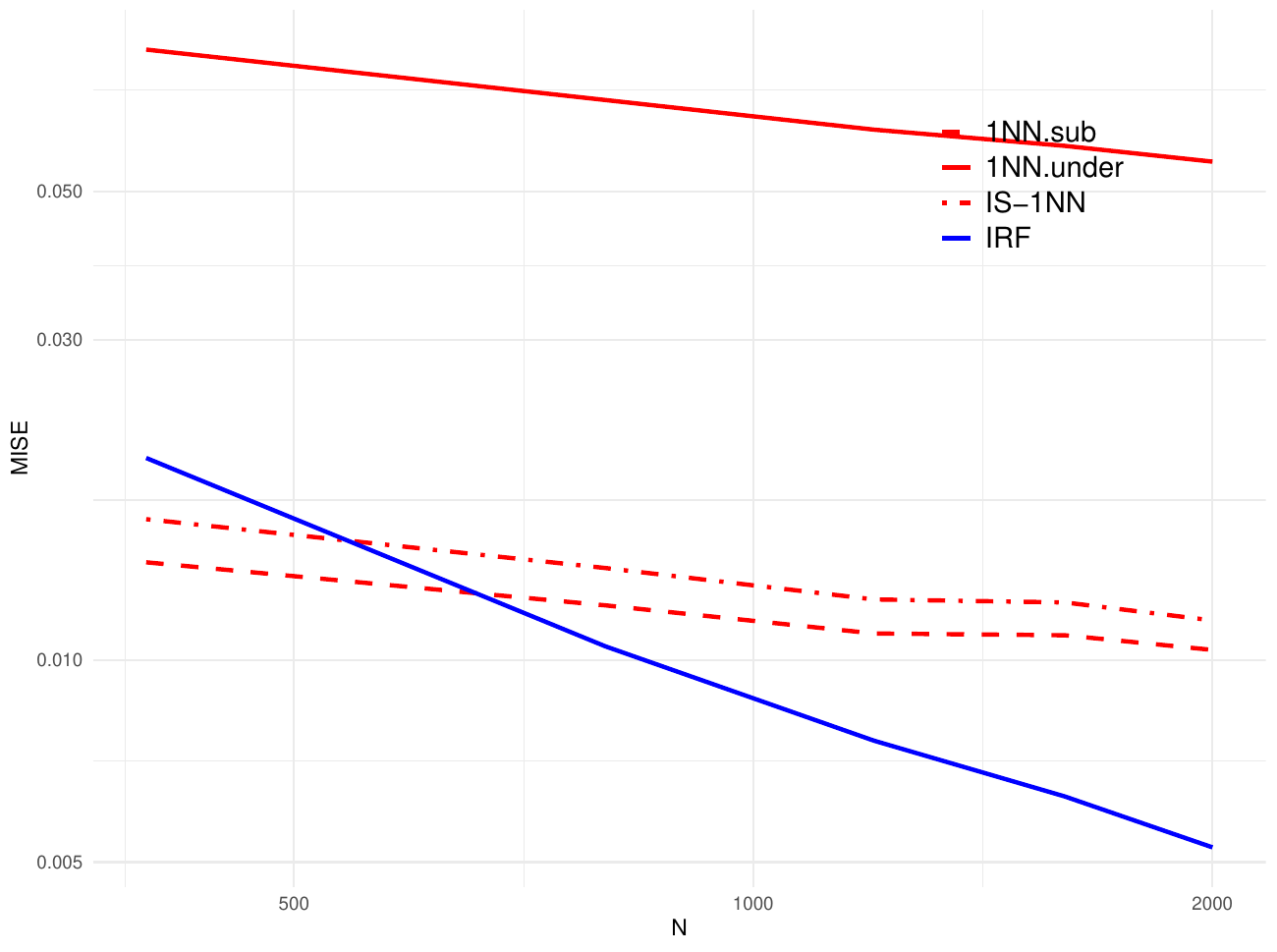}
         \caption{MISE }
         \label{subf: mise_setup2}
     \end{subfigure}
        \caption{Comparison of the MISE trend with that of the bias-variance decomposition w.r.t logarithm of the samples size  for Setup 2.}
        \label{fig:two_Sc3}}
\end{figure}

\section{Conclusion and perspectives}\label{sec: conclusion}
In this work, we have studied predictive probability inference in the context of imbalanced data, focusing on variants of random forests called IRFs. Similar algorithms have been previously employed  by \citet{buhlmann2002analyzing,mentch2016quantifying, wager2018estimation,peng2022rates}.  We have determined convergence rates of the subsampling and under-sampling IRFs estimators under mild regularity assumptions on the regression function. 

The latter estimator, that tackles imbalance  allowing to rebalance the dataset exhibits bias, which we have mitigated using an Importance Sampling (IS) function based on odds ratios. This approach not only improves prediction performance in practice but also ensures asymptotic normality. Additionally, the debiasing procedure we propose is robust, and its application to 1-NN algorithms is simple and effective in both theoretical and practical terms in   marginally imbalanced settings. However, forests based on 1-NN trees have good performances only for sufficiently large number of trees, increasing their effective computational costs.

Our importance sampling  debiasing method can be applied  and extended to any supervised classification and related problems. It would be interesting to apply our results to random forests with random trees independent to the data set, see for example \citet{biau2012analysis}. To satisfy our conditions (Lindeberg condition and bias hypotheses), the challenge would be to control the probability of an observation falling into the leaf given this additional randomness. This is left for a future work.

\begin{appendix}
\section{Proof of Section \ref{sec:Framework}}\label{sec:appendix}

\subsection{Proof of \Cref{prop:u_stat}}\label{proof:prop:u_stat}

Let $I_n := \{1, \ldots, n\}$, and let $S_b \subset I_n$ denote a subset of cardinality $s$.
Recall that in our framework, the RF estimator is defined as a bagged predictor obtained by averaging the outputs of $B$ individual trees $T^s(\bm{x}; \U_b; \bm{Z}_{S_b})$, each constructed from a subsample $\bm{Z}_{S_b}$ of size $s$ drawn without replacement in $\bm{Z}_{I_n}$. It is given by
\begin{align}\label{eq:proofbreiman}
\widehat{\mu}_B(\bm{x}) := \frac{1}{B} \sum_{b=1}^B T^s(\bm{x}; \U_b; \bm{Z}_{S_b}), \qquad \bm{x} \in \mathcal{X},
\end{align}
where  $(\U_b)_{1 \leq b \leq B}$ is a sequence of i.i.d.\ (conditionally on the input sample $\bm{X}_{I_n} := (X_1, \ldots, X_n)$) random variables representing the additional randomness used in the construction of the trees.

We apply the Law of Large Numbers (LLN) given the sample $\bm{Z}_{I_n}$. We obtain the subsampling IRF estimator as the almost sure limit, as $B\to \infty$, of the RF defined in \Cref{eq:proofbreiman}, i.e.\
$$ \widehat{\mu}_B( \bm{x})\stackrel{a.s.}{\longrightarrow} \mathbb{E}[T^s(\bm{x}; \U;\bm Z_{\mathbb S}) | \bm Z_{I_n}]=\widehat{\mu}^{s}(\bm x)\,, \qquad \bm{x}\in {\cal X}$$
where  $\mathbb{S}$ is a random element that corresponds to the sample scheme. Here the random variable $\mathbb{S}$ is uniformly distributed over the set $\{S \subset I_n: |S|=s \}$ since we sample at random $s$ individuals without replacement. Thus the distribution of  $\mathbb{S}$ is given by $$\mathbb{P}( \mathbb{S} = S )=\frac{1}{\binom{n}{s}}\,, \qquad S\subset I_n, |S|=s\,.$$ 
By independence between $\mathbb{S}$ and $(\U,\bm Z_{I_n})$, we apply Fubini's theorem to compute  $\widehat{\mu}^{s}$ as follows 
\begin{eqnarray*}
	\widehat{\mu}^{s}(\bm x) &=&   \mathbb{E}[\mathbb{E}_{ \mathbb{S}}[T^s(\bm{x}; \U;\bm Z_{\mathbb S})] | \bm Z_{I_n}]\\	
	&=&  \mathbb{E}\Big[\sum_{S\subset I_n,  |S|=s}T^s(\bm{x}; \U;\bm Z_S)\mathbb{P}(\mathbb{S} = S)\Big| \bm Z_{I_n}\Big] \\
	&=& \binom{n}{s}^{-1} \sum_{S\subset I_n,  |S|=s} \mathbb{E }\big[T^s(\bm{x}; \U;\bm Z_S)| \bm Z_{I_n}\big]
\end{eqnarray*}
concluding the proof.

\subsection{Proof of Proposition  \ref{prop: case_two}}\label{subsec2:proof_bis}
\begin{proof}

		The proof follows the same reasoning as that of \Cref{prop:u_stat}; see \Cref{proof:prop:u_stat} for details. By applying the LLN, conditional on the sample $\bm{Z}_{I_n}$, we obtain the under-sampling IRF estimator as the almost sure limit when $B \to \infty$ of the RF defined in \Cref{eq:proofbreiman}. The only difference is that the individual trees $T^s(\bm{x}; \U_b; \bm{Z}_{S_b})$ based on subsampling are replaced by the under-sampling trees $T^{s_0,s_1}(\bm{x}; \U_b; \bm{Z}_{S_b})$.

		Here we have  $\mathcal{I}=I_{n_0} \times I_{n_1}=\{(i,j)| i\in I_{n_0}, j\in I_{n_1}\}$, with $I_{n_0}:={\{1,\ldots,n_0\}}$, and  $I_{n_1}:={\{1,\ldots,n_1\}}$ and $\mathbb{S}$ is a random element that is uniformly distributed over the set $\{S: S\subset \mathcal{I}; |S_0|=s_0, |S_1|=s_1\}$ with $0\leq s_0 \leq n_0 \leq n$ and $0\leq s_1 \leq n_1 \leq n$. This corresponds to choose $\mathbb{S}$ as follows
	\begin{align*}
		\mathbb{P}\big( \mathbb{S} = (S_0, S_1)  \big)=\frac{1}{\binom{n_0}{s_0}\binom{n_1}{s_1}}, 
	\end{align*}
	and the proof follows easily using Fubini' s theorem as in  \Cref{proof:prop:u_stat}.
\end{proof}

\section{ Proofs of Section \ref{subsec:mainresult_oneD}}\label{proof1}

\subsection{Proof of \Cref{th2}}\label{subsec2:proof}

This proof serves as a complement to the one provided in Theorem 1 of \citet{peng2022rates}. First, observe that $\widehat{\mu}^{s}$ can be written as a U-statistic, which allows us to apply Hoeffding’s decomposition to facilitate the variance analysis, as originally introduced in \citet{hoeffding1948class}.

In particular, the leading term in the Hoeffding decomposition corresponds to the H\'ajek projection, denoted $\overset{\circ}{\widehat{\mu}}(\bm{x})$, and plays a central role in the proof. By showing that  $\widehat{\mu}^s(\bm{x}) - \overset{\circ}{\widehat{\mu}}(\bm{x})$ is asymptotically negligible in \Cref{lem:ratio_var_1D}, we can apply the Lindeberg Central Limit Theorem to the projection $\overset{\circ}{\widehat{\mu}}(\bm{x})$ instead of the full statistic $\widehat{\mu}^s(\bm{x})$.

We begin in \Cref{degenerate-ust} by presenting the Hoeffding decomposition of $\widehat{\mu}^{s}(\bm{x})$.

\begin{lem}[Hoeffding decomposition of $\widehat{\mu} ^{s}(\bm{x}) $] \label{degenerate-ust} Let
	$\widehat{\mu} ^{s}(\bm{x}) $ be the subsampling IRF estimator at point $\bm{x}\in {\cal X}$ as defined in \eqref{def:subs_classif}. Then, its Hoeffding decomposition \footnote{ see \citet{hoeffding1948class}.} writes
	\begin{align*}
		\widehat{\mu}^{s}(\bm{x}) = \overset{\circ}{\widehat{\mu}}(\bm{x}) + \sum_{r=2}^{s}\binom{s}{r}\widehat{\mu}_{n,r} ^{s}(\bm{x}).
	\end{align*}
	where \begin{align*}
	   \overset{\circ}{\widehat{\mu}}(\bm{x})&:= \mathbb{E}[T^s(\bm{x}; \U;\bm Z_S)] + \frac{s}{n}\sum_{i=1}^n T_1^s(\bm x;  Z_i)\, \qquad\text{and}\\\widehat{\mu}_{n,r} ^{s}(\bm{x})&:= \binom{n}{r}^{-1}  \sum_{A\subset I_n,|A|=r} T_r^s(\bm{x};  \bm Z_A),
	\end{align*}  
 are respectively  the H\'ajek projection of $\widehat{\mu}^{s}$ and  a degenerate U-statistic of order $r$ with
 $$T_r^s(\bm{x};  \bm Z_A):= \sum_{l=0}^r (-1)^{r-l} \sum_{ |B|=l} h_l^s(\bm x;\bm Z_B ), $$
where, for every $0\le l\le s$,
\begin{align*}
h_l^s(\bm x;z_1, \ldots, z_l) &= \mathbb{E}[T^s(\bm x; \U; Z_1, \ldots, Z_s) \ |  Z_i=z_i,   1\le  i \le l ]\\
& = \mathbb{E}[T^s(\bm x; \U ; z_1, \ldots, z_l, Z_{l+1}, \ldots, Z_s)]\,.
\end{align*}

\end{lem}

\begin{proof}[Proof of \Cref{degenerate-ust}]
    See \citet{van2000asymptotic} Chapter 11 and Chapter 12. 
\end{proof}

\noindent In order to show that $\overset{\circ}{\widehat{\mu}}(\bm x)$ and  $\widehat{\mu}^s(\bm x)$ are asymptotically equivalent, we need to prove the ratio of the variance $\widehat{\mu}^{s}(\bm{x})$ and $\overset{\circ}{\widehat{\mu}}(\bm{x})$ tends to $1$. This is the purpose of   \Cref{lem:ratio_var_1D}.

\begin{lem}[Asymptotic equivalence of variances] \label{lem:ratio_var_1D}
Let
	$\widehat{\mu} ^{s}(\bm{x}) $ be the subsampling IRF estimator  as defined in \eqref{def:subs_classif} and $\overset{\circ}{\widehat{\mu}}(\bm{x})$ its H\'ajek projection as in Lemma \ref{degenerate-ust}. Under \textbf{(H1)}, we have 
\begin{align}
    \dfrac{\v(\widehat{\mu} ^{s}(\bm{x}))}{\v(\overset{\circ}{\widehat{\mu}}(\bm{x}))}  \longrightarrow 1 , \qquad \, n \to \infty\, , \quad \bm{x}\in {\cal X}.
\end{align}

\end{lem}

\begin{proof}[Proof of \Cref{lem:ratio_var_1D}] 
    Denote $T^s=T^s(\bm{x}; \U;\bm Z_{I_s})$ and $V_r^s=\v(  T_r^s(\bm{x};  \bm Z_{I_r}))$,  $1\le r\le s$. Since the degenerate U-statistics $\widehat{\mu}_{n,r} ^{s}$ are uncorrelated,  we have 
	\begin{eqnarray*}
		\v(\widehat{\mu} ^{s}(\bm{x})) &=& \v\Big(\overset{\circ}{\widehat{\mu}}(\bm{x}) + \sum_{r=2}^{s}\binom{s}{r}\widehat{\mu}_{n,r} ^{s}(\bm{x}) \Big) \\
		&=&  \v(\overset{\circ}{\widehat{\mu}}(\bm{x})) + \sum_{r=2}^s \binom{s}{r} ^2 \binom{n}{r}^{-1}V_r^s \\
		&=& \frac{s^2}{n}V_1^s + \sum_{r=2}^s \binom{s}{r} ^2 \binom{n}{r}^{-1}V_r^s
	\end{eqnarray*}
    where in the last line we use the writing of the  H\'ajek projection of \Cref{degenerate-ust} that gives   $$\v(\overset{\circ}{\widehat{\mu}}(\bm{x}))=\v\Big(\mathbb{E}[T^s] + \frac{s}{n}\sum_{i=1}^n T_1^s(\bm x;  Z_i)\Big)= \frac{s^2}{n}\v( T_1^s(\bm x;  Z_1)):=\frac{s^2}{n}V_1^s .$$ Besides, we have  
	\begin{eqnarray*}
		\binom{s}{r} \binom{n}{r}^{-1}&=& \frac{\fact{s}}{\fact{r}\fact{(s-r)}}\times\frac{\fact{r}\fact{(n-r)}}{\fact{n}}\\
		&=& \frac{s \times (s-1) \times \ldots \times (s-r+1)}{n \times (n-1) \times \ldots \times (n-r+1)}\\
		&\leq & \frac{s^2}{n^2},\qquad  n\ge s\ge r\geq 2\,. 
	\end{eqnarray*}
 We obtain 
	\begin{eqnarray*}
		\dfrac{\v(\widehat{\mu} ^{s}(\bm{x}))}{\v(\overset{\circ}{\widehat{\mu}}(\bm{x}))}&=& 1+ \Big(\frac{s^2}{n}V_1^s\Big)^{-1} \sum_{r=2}^s \binom{s}{r} ^2 \binom{n}{r}^{-1}V_r^s \\
		& \leq &  1+ \Big(\frac{s^2}{n}V_1^s\Big)^{-1} \frac{s^2}{n^2}\sum_{r=2}^s \binom{s}{r}V_r^s\\
		&\leq & 1+ \frac{\v(T^s)}{nV_1^s}\leq  1+ \frac{1}{nV_1^s} \,. 
	\end{eqnarray*} 
The last line of the inequality above is held thanks to ANOVA decomposition of the tree $T^s$; see \citet{wager2018estimation} page 42. Finally, under \textbf{(H1)} we obtain 
 \begin{align*}
		\dfrac{\v(\widehat{\mu} ^{s}(\bm{x}))}{\v(\overset{\circ}{\widehat{\mu}}(\bm{x}))}  \longrightarrow & 1 , \qquad \, n \to \infty 
	\end{align*} 
 concluding the proof.
\end{proof}

We are now ready to prove \Cref{th2}.
 
\begin{proof}[Proof of \Cref{th2}]
From Lemma  \ref{lem:ratio_var_1D}, we have
\begin{align}
    \dfrac{\v(\widehat{\mu} ^{s}(\bm{x}))}{\v(\overset{\circ}{\widehat{\mu}}(\bm{x}))}  \longrightarrow 1 , \, \, \mbox{as} \quad \, n \to \infty\,,
\end{align}
and   Theorem 11.2 in \citet{van2000asymptotic} provides
	\begin{align}\label{eq:VdV_11.2}
		\frac{\widehat{\mu}^{s}(\bm{x}) - \mathbb{E}[\widehat{\mu}^{s}(\bm{x})]}{\sqrt{\v (\widehat{\mu}^{s}(\bm{x}))}} - \frac{\overset{\circ}{\widehat{\mu}}(\bm{x}) - \mathbb{E}[\overset{\circ}{\widehat{\mu}}(\bm{x})]}{\sqrt{\v (\overset{\circ}{\widehat{\mu}}(\bm{x}))}} \stackrel{\mathbb{P}}{\longrightarrow}  0\,, \qquad n \to \infty\,.
	\end{align}
	\Cref{eq:VdV_11.2} shows that asymptotically, the subsampling  IRF estimator $\widehat{\mu}^s$ and its  H\'ajek projection $\overset{\circ}{\widehat{\mu}}$ coincide.  It remains to show the asymptotic normality of $\overset{\circ}{\widehat{\mu}}$.
	Since $\widehat{\mu}^s(\bm{x})$ and $\overset{\circ}{\widehat{\mu}}(\bm{x})$  depend on the sizes $n$ and $s$ we need to show a central limit theorem for triangular arrays.   Therefore, we can use Lindeberg's central limit theorem; see Theorem 18.1 in \citet{billingsley2008probability}. We apply it to the centered version of the Hajèk projection 	\begin{align*}\label{Hakek_proj_centered}
		\overset{\circ}{\widehat{\mu}}(\bm{x})=\E[T^s] +  \frac{s}{n}\sum_{i=1}^n \big(h_1^s(\bm x; Z_i) - \mathbb{E}[T^s]\big)\,.
	\end{align*}
Using the notation $ T^s_1=\E[T^s(\bm x;\U;Z_{I_s})\mid Z_1]- \mathbb{E}[T^s]=h_1^s(\bm x; Z_1)- \mathbb{E}[T^s]$, the Lindeberg's condition becomes 
	 \begin{equation}\label{eq:lind_b}
	 	\lim_{n \to \infty} \ \frac{n }{\sigma_n^2}  \mathbb{E}\big[ \big( T_1^s\big)^2   \1{\big\{ |T_1^s| > \epsilon \sigma_n \Big\} }\Big]=0\,,
	\end{equation} 
 with 
 $$\sigma_n^2:=\dfrac{n^2}{s^2}\v (\overset{\circ}{\widehat{\mu}}(\bm{x}))=n \v(T_1^s)= nV_1^s\,,$$
as $n\to \infty$.
Equivalently \Cref{eq:lind_b} becomes
	 \begin{align}\label{eq:linderberg_f}
		\lim_{n \to \infty} \   \mathbb{E}\Big[ \big( T_1^s\big)^2   \1{\big\{ |T_1^s| > \epsilon \sqrt{nV_1^s}\big\} }\Big]/V_1^s=0\,.
	\end{align}
This above result holds because $T_1^s \le 1$ a.s.
\begin{align*}
    \mathbb{E}\Big[ \big( T_1^s\big)^2   \1{\big\{ |T_1^s| > \epsilon \sqrt{nV_1^s}\big\} }\Big]/V_1^s&\le\dfrac{\mathbb{E}\Big[ \big( T_1^s\big)^4\Big]}{\epsilon^2 n(V_1^s)^2}\\
    &\le\dfrac{\mathbb{E}\Big[ \big( T_1^s\big)^2\Big]}{\epsilon^2 n(V_1^s)^2}\\
    &\leq \frac{\text{Var}(T_1^s)}{\epsilon^2 n(V_1^s)^2} = \frac{1}{\epsilon^2 n V_1^s }\to 0
\end{align*}
for every $\epsilon>0$ as $nV_1^s\to \infty$. Therefore
$$
\dfrac{ \overset{\circ}{\widehat{\mu}}(\bm{x})-\E[\overset{\circ}{\widehat{\mu}}(\bm{x})]}{\sqrt{\v (\overset{\circ}{\widehat{\mu}}(\bm{x}))}}=\dfrac{\sum_{i=1}^n \big(h_1^s(\bm x; Z_i) - \mathbb{E}[T^s]\big)}{\sigma_n}\overset{d}{\longrightarrow} \mathcal{N}\big(\, 0,1\big)\,,\qquad  n \to \infty .
$$
Besides,  $\v (\overset{\circ}{\widehat{\mu}}(\bm{x}))= \frac{s^2}{n}V_1^s $  and 
$
		\mathbb{E}[ \overset{\circ}{\widehat{\mu}}(\bm{x})] =\E[T^s]=\E[\widehat{\mu}^s(\bm x)]\,.
$
Combined with \Cref{eq:VdV_11.2} we have proven
\begin{align*}
\sqrt{\frac{n}{s^2 V_1^s}}( \widehat{\mu}^s(\bm{x})- \E[\widehat{\mu}^s(\bm x)])
		\overset{d}{\longrightarrow} \mathcal{N}\big(\, 0,1\big)\,,\qquad  n \to \infty ,
	\end{align*}
 concluding the proof.
	
 \end{proof}

\subsection{Proof of \Cref{corth2}}\label{proof:corth2}

\begin{proof}

Under \textbf{(H1)} and \textbf{(H2)},  from \Cref{th2}, adding and subtracting $\mu(\bm{x})$, we obtain 

   \begin{align*}
		\sqrt{\frac{n}{s^2 V_1^s}}(\widehat{\mu}^s(\bm{x})- \E[\widehat{\mu}^s(\bm x)])& =\sqrt{\frac{n}{s^2 V_1^s}}( \widehat{\mu}^s(\bm x)- \mu(\bm x))+
 \sqrt{\frac{n}{s^2 V_1^s}}(\mu(\bm x)-\E[\widehat{\mu}^s(\bm x)]) \\
  &=		\sqrt{\frac{n}{s^2 V_1^s}}( \widehat{\mu}^s(\bm x)- \mu(\bm x)) +o(1)\\
		&\overset{d}{\longrightarrow} \mathcal{N}(\, 0,1 ).
	\end{align*}
 This concludes the proof.
\end{proof}

 \subsection{Proof of \Cref{th1}}\label{subsec2:proof2}
The proof of  \Cref{th1} follows the lines of the proof of \Cref{th2} for the subsampling IRF estimator. We start by providing the Hoeffding decomposition of  $\widehat{\mu}^{s_0,s_1}(\bm x)$. 

\begin{lem}[Hoeffding decomposition of $\widehat{\mu} ^{s_0,s_1}(\bm{x}) $] \label{degenerate-ust-2sample} Let
	$\widehat{\mu} ^{s_0,s_1}(\bm{x}) $ be the under-sampling  IRF estimator at point $\bm x\in \mathcal X$ as defined in \Cref{est: under_classif}. Then, its Hoeffding decomposition writes
\begin{align}
		\widehat{\mu}^{s_0,s_1}(\bm x) = \overset{\circ}{\widehat{\mu}}^{s_0,s_1}(\bm{x}) + \sum_{r_0=1}^{s_0} \sum_{r_1=1}^{s_1}\binom{s_0}{r_0} \binom{s_1}{r_1} \widehat{\sigma}_{r_0,r_1} ^{s_0,s_1}(\bm x),
\end{align}

where  $$\overset{\circ}{\widehat{\mu}}^{s_0,s_1}(\bm{x}):= \mathbb{E}[T^{s_0,s_1}] + \frac{s_0}{n_0}\sum_{i \in I_{n_0}} T_{1,0}^{s_0,s_1}(\bm{x}; Z^0_i)+ \frac{s_1}{n_1} \sum_{j \in I_{n_1}}T_{1,1}^{s_0,s_1}(\bm{x}; Z^1_j),$$ is the  H\'ajek projection of $\widehat{\mu}^{s_0,s_1}(\bm x)$ and 
 $\widehat{\sigma}_{r_0,r_1} ^{s_0,s_1}(\bm x)$ are degenerate U-statistic of kernel $T_{r_0,r_1}^{s_0, s_1}(\bm{x}; \bm{Z}_{A})$ of order $r_0, r_1$, defined by 
	\vspace{0.2cm}
	$$\widehat{\sigma}_{r_0,r_1} ^{s_0,s_1}(\bm x):= \frac{1}{ \binom{n_0}{r_0} \binom{n_1}{r_1}}\sum_{ |A_0|=r_0} \sum_{ |A_1|=r_1} T_{r_0,r_1}^{s_0, s_1}(\bm{x}; \bm{Z}_{A}) $$
for $A:=A_0 \cup A_1 \subset S_0 \cup S_1$,

$$T_{r_0, r_1}^{s_0,s_1}(\bm{x}; \bm{Z}_{A})= \sum_{\substack{0 \leq l_0 \leq r_0 \\ 0 \leq l_1 \leq r_1}} (-1)^{r_0+r_1 -l_0-l_1} \sum_{\substack{|B_0|=l_0 \\ |B_1|=l_1}}h_{l_0,l_1}^{s_0,s_1}(\bm x; \bm Z_{B_0\cup B_1})$$
and for every $0 \leq l_0 \leq r_0$ and $0 \leq l_1 \leq r_1$, 
\begin{align*}
    & h_{l_0,l_1}^{s_0,s_1}(\bm x; z^0_{1},\ldots,z^0_{l_0},z^1_{1},\ldots, z^1_{l_1})\\
    &=\mathbb{E}\Big[ T^{s_0,s_1} \Big(\bm x;\U;  \bm z^0_{1,\ldots, r_0}, \bm Z^0_{r_0+1,s_0}, \bm z^1_{1,\ldots, r_1},  \bm Z^1_{r_1+1,s_1} \Big) \Big],
\end{align*} 
for every  $A_0\subset I_{n_0}$,$A_1\subset I_{n_1}$, $|A_0|=r_0, \, |A_1|=r_1$, $1\le r_0\le s_0$, $1\le r_1\le s_1$. Also, $B_0\subset A_0$,$B_1\subset A_1$
and every $\bm{x}\in \mathcal{X}$, 
where $\bm z^j_{1\ldots,r_j}=(z^j_{1},\ldots,z^j_{r_j})$, \, $\bm Z^j_{r_j+1,s_j}=(Z^j_{r_j+1},\ldots, Z^j_{s_j})$, \, with $j=\{0,1\}$.

\begin{proof}[Proof of Lemma \ref{degenerate-ust-2sample}]  Similarly to the \Cref{degenerate-ust}, we have  
\begin{eqnarray*}
		\widehat{\mu}^{s_0,s_1}(\bm x)&=& \Big( \binom{n_0}{s_0} \binom{n_1}{s_1}\Big)^{-1}\sum_{\substack{0\leq r_0 \leq s_0 \\ 0 \leq r_1 \leq s_1}}  \sum_{\substack{|A_0|=r_0 \\ |A_1|=r_1}}  \binom{n_0-r_0}{s_0-r_0} \\
  & \times& \binom{n_1-r_1}{s_1-r_1} T_{r_0,r_1}^{s_0, s_1}(\bm{x}; \bm{Z}_{A})\\
		&=& \sum_{r_0=0}^{s_0} \sum_{r_1=0}^{s_1}\binom{s_0}{r_0} \binom{s_1}{r_1} \Big( \binom{n_0}{r_0} \binom{n_1}{r_1}\Big)^{-1}\sum_{ |A_0|=r_0} \sum_{ |A_1|=r_1} T_{r_0,r_1}^{s_0, s_1}(\bm{x}; \bm{Z}_{A})\\
		&=& \mathbb{E}[T^{s_0,s_1}] + \frac{s_0}{n_0}\sum_{i \in I_{n_0}} T_{1,0}^{s_0,s_1}(\bm{x}; Z^0_i)+ \frac{s_1}{n_1} \sum_{j \in I_{n_1}}T_{1,1}^{s_0,s_1}(\bm{x}; Z^1_j)\\ &+&  \sum_{r_0=1}^{s_0} \sum_{r_1=1}^{s_1}\binom{s_0}{r_0} \binom{s_1}{r_1} \widehat{\sigma}_{r_0,r_1} ^{s_0,s_1}(\bm x).
	\end{eqnarray*}

\end{proof}
 
 \end{lem}
Next, we state the asymptotic equivalence of variances of $\widehat{\mu}^{s_0,s_1}(\bm x)$ and \\ $\overset{\circ}{\widehat\mu}^{s_0,s_1}(\bm{x})$ in \Cref{lem:ratio_var_2D}.

\begin{lem}[Asymptotic equivalence of variances] \label{lem:ratio_var_2D}
Let
	$\widehat{\mu}^{s_0,s_1}(\bm x)$ be the under-sampling IRF estimator  as defined in \eqref{est: under_classif} and $\overset{\circ}{\widehat{\mu}}^{s_0,s_1}(\bm{x})$ its Hàjek projection as in Lemma \ref{degenerate-ust-2sample}. Under \textbf{(H1')}, we have
$$ \frac{\v( \widehat{\mu}^{s_0,s_1}(\bm x))}{\v\big(\overset{\circ}{\widehat\mu}^{s_0,s_1}(\bm{x})\big)}\longrightarrow 1, \quad \mathrm{as} \, \, (n_0\vee n_1) \to \infty, \quad \bm x\in \mathcal X.$$
\end{lem}

\begin{proof}[ Proof of \Cref{lem:ratio_var_2D}] 
Using Hoeffding's decomposition provided in   Lemma \ref{degenerate-ust-2sample}, since the terms in the sum of the degenerate $U$-statistics are uncorrelated  we have
	 
	\begin{eqnarray*}
		\v( \widehat{\mu}^{s_0,s_1}(\bm x)) :=  \v\big(\overset{\circ}{\widehat{\mu}}^{s_0,s_1}(\bm{x})\big)+ \sum_{r_0=1}^{s_0} \sum_{r_1=1}^{s_1}\binom{s_0}{r_0}^2 \binom{s_1}{r_1} ^2 \Big( \binom{n_0}{r_0} \binom{n_1}{r_1}\Big)^{-1} V_{r_0,r_1} ^{s_0,s_1}
	\end{eqnarray*}
	where $ V_{r_0,r_1} ^{s_0,s_1}= \v\big(T_{r_0,r_1}^{s_0, s_1}(\bm{x}; \bm Z_{
 A})\big) $. Besides,  
 
	\begin{align}\label{var:IS}
		\v(\overset{\circ}{\widehat{\mu}}^{s_0,s_1}(\bm{x}))=
		\frac{s_0^2}{n_0} V_{1,0} ^{s_0,s_1} + \frac{s_1^2}{n_1} V_{1,1} ^{s_0,s_1}
	\end{align}
	with $V_{1,0} ^{s_0,s_1}=\v\big(T_{1,0}^{s_0,s_1}(\bm{x}; Z_1^0) \big) $ and $V_{1,1} ^{s_0,s_1}=\v\big(T_{1,1}^{s_0,s_1}(\bm{x}; Z_1^1) \big) $. 
 Next, observe that 
	
	\begin{eqnarray*}
		\binom{s_0}{r_0} \binom{n_0}{r_0}^{-1}
		&\leq & \frac{s_0}{n_0}, \, \, \forall  \, r_0\geq 1, \\
		\binom{s_1}{r_1} \binom{n_1}{r_1}^{-1}
		&\leq & \frac{s_1}{n_1}, \, \, \forall  \, r_1\geq 1.
	\end{eqnarray*}
 
	\noindent Therefore, the ratio of variances  simplifies as
	
	\begin{eqnarray*}
		\frac{\v( \widehat{\mu}^{s_0,s_1}(\bm x))}{\v\big(\overset{\circ}{\widehat{\mu}}^{s_0,s_1}(\bm{x})\big)}&\leq & 1+ \Big( \frac{s_0^2}{n_0} V_{1,0} ^{s_0,s_1} + \frac{s_1^2}{n_1} V_{1,1} ^{s_0,s_1} \Big)^{-1}\frac{s_0\times s_1}{n_0\times n_1} \sum_{r_0=1}^{s_0} \sum_{r_1=1}^{s_1}\binom{s_0}{r_0} \\
  &  \quad\times& \binom{s_1}{r_1} V_{r_0,r_1} ^{s_0,s_1}\\
		&\leq& 1 + \Big(\frac{n_1}{s_1} s_0V_{1,0} ^{s_0,s_1} + \frac{n_0}{s_0} s_1V_{1,1} ^{s_0,s_1}\Big)^{-1} \v(T^{s_0,s_1}).
	\end{eqnarray*}

The variance on the right side of the above inequality comes from the ANOVA decomposition; see \citet{wager2018estimation} page 42.  Finally,
	\begin{eqnarray}
		\frac{\v( \widehat{\mu}^{s_0,s_1}(\bm x))}{\v\big(\overset{\circ}{\widehat{\mu}}^{s_0,s_1}(\bm{x})\big)}
		&\leq & 1+ \frac{\v(T^{s_0,s_1})}{n_1(s_0/s_1) V_{1,0} ^{s_0,s_1} +n_0(s_1/s_0)V_{1,1} ^{s_0,s_1}}  \nonumber \\
		& \longrightarrow & 1 \, \, \mathrm{as} \, \, n \to \infty, \nonumber
	\end{eqnarray}
under \textbf{(H1')}. This concludes the proof. 
\end{proof}

We are now ready to prove \Cref{th1}.

\begin{proof}[Proof of \Cref{th1}]
    
 From   \Cref{lem:ratio_var_2D} we have
$$ \frac{\v( \widehat{\mu}^{s_0,s_1}(\bm x))}{\v\big(\overset{\circ}{\widehat\mu}^{s_0,s_1}(\bm{x})\big)}\longrightarrow  1, \quad \mathrm{as} \, \, (n_0\vee n_1) \to \infty, \quad \bm x\in \mathcal X $$
 and Theorem  11.1 in \citet{van2000asymptotic}  provides, as $(n_0\vee n_1) \to \infty$,	\begin{align}\label{eq:VdV_2sample}
	\frac{\widehat{\mu}^{s_0,s_1}(\bm x) - \mathbb{E}[\widehat{\mu}^{s_0,s_1}(\bm x)]}{\sqrt{\v\big(\widehat{\mu}^{s_0,s_1}(\bm x)\big)}} - \frac{\overset{\circ}{\widehat{\mu}}^{s_0,s_1}(\bm{x})- \mathbb{E}[\overset{\circ}{\widehat{\mu}}^{s_0,s_1}(\bm{x})]}{\sqrt{\v\big(\overset{\circ}{\widehat{\mu}}^{s_0,s_1}(\bm{x})\big)}} \xrightarrow{\mathbb{P}} 0, \quad \bm x\in \mathcal X.
	\end{align}
	
Similarly to \Cref{th2},  we need to show   Lindeberg's central limit theorem for the triangular array see, \citet{billingsley2008probability}. To do this, we consider the centered H\'ajek projection as follows 
	
	\begin{align} \label{eq:somme_ H\'ajek}
	\overset{\circ}{\widehat{\mu}}^{s_0,s_1}(\bm{x})- \mathbb{E}[\overset{\circ}{\widehat{\mu}}^{s_0,s_1}(\bm{x})] &= \frac{s_0}{n_0}\sum_{i \in I_{n_0}} T_{1,0}^{s_0,s_1}(\bm{x}; Z^0_i)\\
	& + \frac{s_1}{n_1} \sum_{j \in I_{n_1}}T_{1,1}^{s_0,s_1}(\bm{x}; Z^1_j)\,,
	\end{align}
 with $ T_{1,0}^{s_0,s_1}(\bm{x}; Z^0_i)$, $ T_{1,0}^{s_0,s_1}(\bm{x}; Z^1_i)$, being  by definition iid copies of $T_{1,0}^{s_0,s_1}$ and $ T_{1,1}^{s_0,s_1}$, respectively. Here we recall the notation
\begin{align*}
T^{s_0,s_1}_{1,0}&=\E[T^{s_0,s_1}\mid Z_1^0]-\E[T^{s_0,s_1}],\\
T_{1,1}^{s_0,s_1} &= \mathbb{E}[T^{s_0,s_1} | \ Z_1^1] - \mathbb{E}[T^{s_0,s_1}]
\end{align*}
and  
  
	\begin{align}\label{eq:moyen_mu}
		\mathbb{E}[\overset{\circ}{\widehat{\mu}}^{s_0,s_1}(\bm{x})]& = \mathbb{E}[T^{s_0,s_1}] + \frac{s_0}{n_0}\sum_{i \in I_{n_0}} \underbrace{\E[ T_{1,0}^{s_0,s_1}(\bm{x}; Z^0_i)]}_{=0}+ \frac{s_1}{n_1} \sum_{j \in I_{n_1}} \underbrace{\E[ T_{1,1}^{s_0,s_1}(\bm{x}; Z^1_j)]}_{=0}.
	\end{align}

Assume \textbf{(H1')}. It implies that $n_0V_{1,0} ^{s_0,s_1}\to \infty$ and $n_1V_{1,1} ^{s_0,s_1}\to \infty$ as $n\to \infty$.  For the first term in the r.h.s  of  Equation \eqref{eq:somme_ H\'ajek}, Lindeberg's condition writes 
	\begin{align}
		\lim_{n_0 \to \infty} \ \frac{1 }{\delta_{n_0}^2} \sum_{i=1}^{n_0} \mathbb{E}\Big[ \Big(\frac{s_0}{n_0} T_{1,0}^{s_0,s_1}\Big)^2   \1{\{ |T_{1,0}^{s_0,s_1}| > \epsilon \delta_{n_0} \} }\Big]=0,
	\end{align}
	where  $\delta_{n_0}^2=\sum_{i=1}^{n_0} \frac{s _0^2}{n_0^2}\v( T_{1,0}^{s_0,s_1})= \frac{s_0^2}{n_0} V_{1,0} ^{s_0,s_1}$. This condition simplifies to
	\begin{align}\label{eq:lindeberg_two_sample1}
		\lim_{n_0 \to \infty} \   \frac{\mathbb{E}\big[  \big(T_{1,0}^{s_0,s_1} \big)^2   \mathds{1}{\{ |T_{1,0}^{s_0,s_1}| > \epsilon \sqrt{n_0 V_{1,0} ^{s_0,s_1}}\} }\big]}{V_{1,0} ^{s_0,s_1}}=0\,.
	\end{align}
	
Similar to \Cref{eq:linderberg_f}, for $s_0\le n_0\le n$, 
 \Cref{eq:lindeberg_two_sample1} holds by the same argument. That is,  $|T_{1,0}^{s_0,s_1}|\le 1$ a.s,  we have 

\begin{align*}
    \mathbb{E}\Big[\big(T_{1,0}^{s_0,s_1}\big)^2\1{\big\{ |T_1^s| > \epsilon \sqrt{nV_1^s}\big\} }\Big]/ V_{1,0} ^{s_0,s_1} &\le \dfrac{\mathbb{E}\Big[ \big(T_{1,0}^{s_0,s_1}\big)^2\Big]}{\epsilon^2 n_0 (V_{1,0} ^{s_0,s_1})^2}\\
    & \le  \frac{V_{1,0} ^{s_0,s_1}}{\epsilon^2 n_0 (V_{1,0} ^{s_0,s_1})^2} = \frac{1}{\epsilon^2 n_0 V_{1,0} ^{s_0,s_1}}\to 0
\end{align*}
 for every $\epsilon>0$ and $n_0 V_{1,0} ^{s_0,s_1}\to \infty$ as $n\to \infty$.   So, defining $$\overset{\circ}{\widehat{\mu}}_{1,j}^{s_0,s_1}(\bm{x}):=\frac{s_j}{n_j}\sum_{i \in I_{n_0}} T_{1j}^{s_0,s_1}(\bm{x}; Z^0_i)\,,\qquad j=0,1\,,$$ as a sum of i.i.d. centered random variables for each $s_j$, then we obtain
 \begin{align*}
		\sqrt{\frac{n_0}{s_0^2 V_{1,0} ^{s_0,s_1}}}\overset{\circ}{\widehat{\mu}}_{1,0}^{s_0,s_1}(\bm{x})  \overset{d}{\longrightarrow} \mathcal{N}\big(0,1 \big),\, \, \mbox{as}\, \, n_0 \to \infty, \quad \bm x\in \mathcal X.
	\end{align*}

Similarly Lindeberg's  condition for the second term in the r.h.s  of Equation \eqref{eq:somme_ H\'ajek} as $n_1V_{1,1} ^{s_0,s_1}\to \infty$ as $n\to \infty$ and
 \begin{align*}
		\sqrt{\frac{n_1}{s_1^2 V_{1,1} ^{s_0,s_1}}}\overset{\circ}{\widehat{\mu}}_{1,1}^{s_0,s_1}(\bm{x})  \overset{d}{\longrightarrow} \mathcal{N}\big(0,1 \big) ,\, \, \mbox{as}\, \, n_1 \to \infty, \quad \bm x\in \mathcal X.
	\end{align*}

 Thus, $\overset{\circ}{\widehat{\mu}}^{s_0,s_1}(\bm{x})$ is asymptotically normally distributed as the sum of two independent random variables which are asymptotically normally distributed. Moreover, using  \Cref{eq:VdV_2sample} we obtain 
 $$
		\sqrt{\frac{1}{s_0^2 V_{1,0} ^{s_0,s_1}/n_0+ s_1^2 V_{1,1} ^{s_0,s_1}/n_1}}(\widehat{\mu}^{s_0,s_1}(\bm x) - \mathbb{E}[\widehat{\mu}^{s_0,s_1}(\bm x)]) \overset{d}{\longrightarrow} \mathcal{N}(0,1 ) ,\qquad  n \to \infty\,.
$$
This concludes the proof. 

\end{proof}

  \subsection{Proof of Proposition \ref{rem:conv_mu}}\label{proof:rem:conv_mu}

Let  $\bm{x}\in {\cal X}$.  If \textbf{(G1)} holds, basic calculations  yield  
	\begin{eqnarray*}
		|\mathbb{P}(Y=1 |  X \in L_{\U}(\bm{x})) -  \mu(\bm{x}) | &=&  \bigg|\frac{\mathbb{E}\big[ \int_{ L_{\U}(\bm{x})} \mu(\bm{y})d\mathbb{P}_X(\bm{y})\big]}{\mathbb{P}(X \in L_{\U}(\bm{x})) } - \mu(\bm{x})\bigg|\\ 
		&\leq & \frac{\mathbb{E}\big[ \int_{ L_{\U}(\bm{x})} | \mu(\bm{y}) - \mu(\bm{x}) | d\mathbb{P}_X(\bm{y})\big]}{\mathbb{P}(X \in L_{\U}(\bm{x})) } \\
		&\leq & \frac{\mathbb{E}\big[ \int_{ L_{\U}(\bm{x})}  L||   \bm{y} - \bm{x} \| _{\infty} d\mathbb{P}_X(\bm{y})\big]}{\mathbb{P}(X \in L_{\U}(\bm{x})) } \\
  		&\leq & L\,\frac{\mathbb{E}\big[\textup{ Diam}(L_{\U}(\bm{x}))\1{\{X\in L_{\U}(\bm{x})\}} \big]}{\mathbb{P}(X \in L_{\U}(\bm{x})) } \\
		&\leq & L\,  \mathbb{E}\big[ \textup{Diam}(L_{\U}(\bm{x}))\big| X \in L_{\U}(\bm{x})\big]
	\end{eqnarray*}
 which goes to $0$ as $n \to \infty$ under Condition \textbf{(G2)}. Remark also that under \textbf{(G1)}, the regression function  $\mu^*$ satisfies the relation $\mu^*(\bm{x})=g^{-1}(\mu(\bm{x}))$ and is Lipschitz continuous, where the Lipschitz constant is $$ |(g^{-1}(z))^\prime|\le L^\prime= \frac{p/(1-p)}{p^{*}/(1-p^{*})}, \qquad z\in [0,1].$$
This Lipschitz constant corresponds to a  multiplicative factor in the odds-ratio procedures, see \Cref{subsec:odds}.
\subsection{Proof of Proposition \ref{rem:bias_mu_s}}\label{proof:rem:bias_mu_s}
  \begin{proof}
 We have  
\begin{eqnarray*}
		\mathbb{E}  [T^s]- \mu(\bm{x}) &=& \mathbb{E}\Big[\frac{1}{N_{L_{\U}(\bm{x})}(\bm X_S)} \sum_{i=1}^{s} Y_i\1{\{X_i \in L_{\U}(\bm{x}) \}}\Big] - \mu(\bm{x})\\
  &=&   \E\Big[\dfrac{\mathds{1}{\{ N_{L_{\U}(\bm{x})}(\bm X_S)>0\}}}{N_{L_{\U}(\bm{x})}(\bm X_S)}\sum_{i=1}^{s}( Y_i- \mu(\bm{x}) )  \1{\{ X_i \in L_{\U}(\bm{x}) \}}\Big]\\
  &&-\mu(\bm{x})\mathbb{P}(N_{L_{\U}(\bm{x})}(\bm X_S)=0),
  \end{eqnarray*}
and the second r.h.s term vanishes under Assumption \textbf{(G2)}. For the first one we use the tower property and the independence of $\1{\{ X_i \in L_{\U}(\bm{x}) \}}$, $i\in S$, given $\U$:
  \begin{align*}
		& \mathbb{E}\Big[\dfrac{\mathds{1}{\{ N_{L_{\U}(\bm{x})}(\bm X_S)>0\}}}{N_{L_{\U}(\bm{x})}(\bm X_S)} \E\Big[\sum_{i=1}^{s}( Y_i - \mu(\bm{x}) )  \1{\{ X_i \in L_{\U}(\bm{x}) \}}   | (\1{\{ X_i \in L_{\U}(\bm{x}) \}})_{i\in S},\\
  & \qquad \qquad \U\Big ]\Big]\\
    &= \mathbb{E}\Big[\dfrac{\1{\{ N_{L_{\U}(\bm{x})}(\bm X_S)>0\}}}{N_{L_{\U}(\bm{x})}(\bm X_S)} \sum_{i=1}^{s}\E[ Y_i - \mu(\bm{x})   | \mathds{1}{\{ X_i \in L_{\U}(\bm{x}) \}},\U]\1{\{ X_i \in L_{\U}(\bm{x}) \}}   \Big]\\
    &= \mathbb{E}\Big[\dfrac{\mathds{1}{\{ N_{L_{\U}(\bm{x})}(\bm X_S)>0\}}}{N_{L_{\U}(\bm{x})}(\bm X_S)} \sum_{i=1}^{s}( \mathbb{P}(Y_i=1 |  X_i \in L_{\U}(\bm{x}),\U) -  \mu(\bm{x}) )\1{\{ X_i \in L_{\U}(\bm{x}) \}}\Big].\,
\end{align*}	
In the last step we use the fact that $\U$ does not depend on $\bm Y_S$.	
Using a similar argument than in Section \ref{proof:rem:conv_mu} we derive under \textbf{(G1)} that for every $1\le i\le s$
$$   
| \mathbb{P}(Y_i=1 |  X_i \in L_{\U}(\bm{x}), \U) -  \mu(\bm{x}) |\le L\,   \mathbb{E}[ \textup{Diam}(L_{\U}(\bm{x}))| X_i \in L_{\U}(\bm{x}),\U]\,.
$$
Therefore we find the upper-bound
\begin{align*}
  & L\,\mathbb{E}\Big[\dfrac{\mathds{1}{\{ N_{L_{\U}(\bm{x})}(\bm X_S)>0\}}}{N_{L_{\U}(\bm{x})}(\bm X_S)} \sum_{i=1}^s \mathbb{E}[ \textup{Diam}(L_{\U}(\bm{x}))| X_i \in L_{\U}(\bm{x}),\U]\1{\{ X_i \in L_{\U}(\bm{x}) \}}\Big]\\
  &\leq  L \, \mathbb{E} [ \textup{Diam}(L_{\U}(\bm{x})) | X\in L_{\U}(\bm{x})]\,,
 \end{align*}
 and we conclude the proof thanks to \textbf{(G2)}.

\end{proof}

\subsection{Proof of \Cref{th: mu_start_TCL} }\label{subsec:proof_th_mu_start_TCL}

Since $ \widehat{\mu}^\ast(\bm{x})=\overset{\circ}{\widehat{\mu}}^{s_0,s_1}(\bm{x})$, from \Cref{th1}  we obtain under \textbf{(H2')}
	\begin{align*}
		&\sqrt{\frac{1}{s_0^2 V_{1,0} ^{s_0,s_1}/n_0+ s_1^2 V_{1,1} ^{s_0,s_1}/n_1}}\big(\widehat{\mu}^\ast(\bm{x})-  \mathbb{E}[\widehat{\mu}^\ast(\bm{x})]\big)\\
  &=\sqrt{\frac{1}{s_0^2 V_{1,0} ^{s_0,s_1}/n_0+ s_1^2 V_{1,1} ^{s_0,s_1}/n_1}}( \widehat{\mu}^\ast(\bm{x})- \mu^{\ast}(\bm x))\\
  & \qquad + \sqrt{\frac{1}{s_0^2 V_{1,0} ^{s_0,s_1}/n_0
   + s_1^2 V_{1,1} ^{s_0,s_1}/n_1}}(\mu^{\ast}(\bm x)-\E[\widehat{\mu}^\ast(\bm{x})])\\
  &=	\sqrt{\frac{1}{s_0^2 V_{1,0} ^{s_0,s_1}/n_0+ s_1^2 V_{1,1} ^{s_0,s_1}/n_1}}(\widehat{\mu}^\ast(\bm{x})- \mu^{\ast}(\bm x)) +o(1)\\
		&\overset{d}{\longrightarrow} \mathcal{N}(\, 0,1 )\,,\qquad  n \to \infty\,,\qquad \bm{x}\in\mathcal{X}.
	\end{align*}
 
 This concludes the proof.

\subsection{Proof of \Cref{Two_case_th}}\label{subsec2:proof3}
Defining $  v_n=s_0^2 V_{1,0} ^{s_0,s_1}/n_0+ s_1^2 V_{1,1} ^{s_0,s_1}/n_1$, \Cref{th1} yields
$$ \frac{1}{\sqrt{v_n}}\Big(\begin{pmatrix}
	\widehat{\mu}^*(\bm x)\\ 
	1- \widehat{\mu}^*(\bm x)
\end{pmatrix} -\begin{pmatrix}
	\E[\widehat{\mu}^*(\bm x)]\\ 
	1- \E [\widehat{\mu}^*(\bm x)]
\end{pmatrix}\Big) \overset{d}{\longrightarrow} \mathcal{N}\Big(\, 0, \, \, \begin{pmatrix}
	1 & -1\\ 
	-1 & 1
\end{pmatrix}  \Big).$$

\noindent Besides, since  by assumption it holds a.s.
\begin{align*}
n_1s_0/ns \to p(1-p^*)\,,\qquad 
n_0s_1/ns &\to (1-p)p^*\,,\qquad n\to \infty\,,
\end{align*}
Slutsky's theorem yields 
\begin{align*} \frac{1}{\sqrt{v_n}} \Big( \Big(\frac{n_1s_0}{ns}
	\dfrac{\widehat{\mu}^*(\bm x)}{p(1-p^*)},
	\dfrac{n_0s_1}{ns}\dfrac{1- \widehat{\mu}^*(\bm x)}{(1-p)p^*}\Big)&-\\
 &
	\Big(\frac{n_1s_0}{ns}
	\dfrac{ \E [\widehat{\mu}^*(\bm x)]}{p(1-p^*)},
	\dfrac{n_0s_1}{ns}\dfrac{1- \E [\widehat{\mu}^*(\bm x)]}{(1-p)p^*}\Big)\Big)^T\\
 &\overset{d}{\longrightarrow} \mathcal{N}\Big(\, 0, \, \, \begin{pmatrix}
	1 & -1\\ 
	-1 & 1\end{pmatrix}  \Big)\,, \qquad n\to \infty\,.
\end{align*}
Defining the function $h$ as
\begin{equation}\label{eq:funct_g}
	h: (z, z^\prime)\longmapsto \frac{p(1-p^{*}) z}{(1-p)p^{*}z ^\prime+p
		(1-p^{*}) z}\,,\qquad (z, z^\prime) \in [0,1]^2\,,
\end{equation}
we identify 
$ g_n\big(\widehat{\mu}^*(\bm x)\big)=h\Big( \frac{n_1s_0}{ns}\frac{\widehat{\mu}^*(\bm x)}{p(1-p^{*})}, \frac{n_0s_1}{ns}\frac{1- \widehat{\mu}^*(\bm x)}{(1-p)p^{*}}\Big)$. Therefore we can apply
delta's method  

\begin{multline*}
\frac{1}{\sqrt{v_n}}\Big(h\Big( \frac{n_1s_0}{ns}\frac{\widehat{\mu}^*(\bm x)}{p(1-p^{*})}, \frac{n_0s_1}{ns}\frac{1- \widehat{\mu}^*(\bm x)}{(1-p)p^{*}}\Big) \\
 -h\Big( \frac{n_1s_0}{ns}\frac{\E [\widehat{\mu}^*(\bm x)]}{p(1-p^{*})}, \frac{n_0s_1}{ns}\frac{1- \E [\widehat{\mu}^*(\bm x)]}{(1-p)p^{*}}\Big)\Big)
	\overset{d}{\longrightarrow} \mathcal{N}(\, 0, V^*(\bm x) )\,,\qquad n \to \infty\,,
\end{multline*}
where 
\begin{eqnarray*}
	V^* (\bm x)  &=& (  \partial_{\mu} h , \partial_{1-\mu} h) \begin{pmatrix}
		1 & -1\\ 
		-1 & 1
	\end{pmatrix} \begin{pmatrix}
		\partial_{\mu} h\\ 
		\partial_{1-\mu} h
	\end{pmatrix} \\
	&=& \frac{(p p^{*}(1-p^{*}) (1-p)) ^2}{((1-p)p^{*}(1- \mu(\bm{x}))+p
		(1-p^{*}) \mu(\bm{x}))^4} .
\end{eqnarray*}
We obtain
$$
\frac{1}{\sqrt{v_n}}\,(g_n(\widehat \mu^*(\bm x))-g_n( \E [\widehat{\mu}^*(\bm x)])\overset{d}{\longrightarrow} \mathcal{N}(\, 0,  V^*(\bm x) )\,,\qquad n \to \infty.\,
$$
Since $\widehat{\mu}_{IS}(\bm{x}):=g_n( \widehat{\mu}^*(\bm x))$ by definition, the desired result follows. 
\begin{flushright}
	$\Box$
\end{flushright}
\subsection{Proof of Corollary \ref{cor:Two_case_th}}\label{subsec3:proof4}
We use the notation of the proof of  \Cref{Two_case_th},  we have 
\begin{align*}
    \frac{1}{\sqrt{v_n}}\,(g_n(\widehat \mu^*(\bm x))-g_n( \E [\widehat{\mu}^*(\bm x)]))\overset{d}{\longrightarrow} \mathcal{N}(\, 0,  V^*(\bm x) )\,,\, n \to \infty .
\end{align*}
Adding subtracting by $\mu(\bm x)$, we obtain as $n \to \infty$
\begin{align*}
    \frac{1}{\sqrt{v_n}}\,(g_n(\widehat \mu^*(\bm x))-\mu(\bm x))- \frac{1}{\sqrt{v_n}}(g_n( \E [\widehat{\mu}^*(\bm x)])-\mu(\bm x))\overset{d}{\longrightarrow} \mathcal{N}(\, 0,  V^*(\bm x) )\,.
\end{align*}
 Under \textbf{(H2'')}, we have proven that
\begin{align*}
  \frac{1}{\sqrt{v_n}}\,(g_n(\widehat \mu^*(\bm x))-\mu(\bm x))\overset{d}{\longrightarrow} \mathcal{N}(\, 0,  V^*(\bm x) )\,,\qquad n \to \infty .
\end{align*}
This concludes the proof.
 \begin{flushright}
	$\Box$
\end{flushright}
\section{Proofs of \Cref{sec:KNN}}\label{sec: proof5}
\subsection{Proofs of \Cref{subsec:1NN_standard}}
We start this Section with some useful lemmas.
\begin{lem}\label{lem:1NN}
Let $Y_{nn(\bm x)}$ be the 1-NN estimator at point $\bm x \in \cal X$ as defined in \Cref{def:1NN} and \Cref{def:nnx}. Under Condition \textbf{(D)}, for $\bm x \in \cal X$ and  $i \in I$,  it holds
\begin{enumerate}
    \item[$i)$] The random variables $\overline F(\d(\bm x,X_i))$  are i.i.d. and uniformly distributed on $(0,1)$;
    \item[$ii)$]   $\P(nn(\bm x)=i\mid\overline F(\d(\bm x,X_i)))=\overline F(\d(\bm x,X_i))^{n-1}$;
    \item[$iii)$]   $\E[Y_i\mid nn(\bm x)=i, \overline F(\d(\bm x,X_i))=u]=g(\bm x,1-u)$,\quad for\,   $u\in (0,1)$
\end{enumerate} 
where $g(\bm x,1-u):=\E[\mu(X)\mid  \overline F(\d(\bm x,X))=u]$, \, for\,  $u\in (0,1)$.
\end{lem}

\begin{proof}
    See \Cref{sub: proof-of-lemma5.1}.
\end{proof}
\noindent Observe by construction $g(\bm x,1-u)\to \mu(\bm x)$ when $u\to 1$ for $\bm x \in \cal X$.
\begin{rem}\label{rem:Renyi} By definition, since $E \sim \mathcal{E}(1)$, the random variable $e^{-E/n}$ is $\beta(n,1)$-distributed. As a consequence,  we have  the identity 
\[
\E[g(\bm x,1-e^{-E/n})]= n\int_0^1g(\bm x,1-u) u^{n-1}du \,,\qquad n\ge 1.
\]
In the sequel we will widely use this identity together with the domination argument $0\le n(1-e^{-E/n})\le E$ a.s.\
\end{rem}

\begin{lem}\label{lem:E_ynn2}
Let $Y_{nn(\bm x)}$ be the 1-NN estimator at point $\bm x \in \cal X$ as defined in \Cref{def:1NN} and \Cref{def:nnx}  and let $g$ as defined in \Cref{lem:1NN} iii). Then,  under Condition \textbf{(D)}, the following two assertions hold
\begin{align*}
   \E[Y_{nn(\bm x)}]=\E[g(\bm x,1-e^{-E/n})]\,, 
\end{align*}
and \begin{align*}
    \v(Y_{nn(\bm x)})=\E[g(\bm x,1-e^{-E/n})]-\E[g(\bm x,1-e^{-E/n})]^2\,.
\end{align*}
\end{lem}
\begin{proof}
    See \Cref{sub: proof-of-lemmaC2}.
\end{proof}
\subsubsection{Proof of \Cref{lem:1NN}}\label{sub: proof-of-lemma5.1}

Assume Condition \textbf{(D)} holds. Then,
\begin{enumerate}
\item[$i)$] Since $F$  is the c.d.f of the random variable $\d(\bm x,X)$, thanks to inverse transform sampling, $U=F(\d(\bm x,X))$  transforms $\d(\bm x,X)$ into a uniform random variable over $(0, 1)$. By the symmetry of the uniform distribution,  $\overline F(\d(\bm x,X))$ is uniformly distributed over $(0, 1)$ as well.

\item[$ii)$] From the monotonicity of $\overline F$ together with \Cref{def:nnx_equals_i}, for $\bm x \in \cal X$ and  $i \in I$, we have
\begin{align*}
\{nn(\bm x)=i\}&=\{\d(\bm x,X_i)<\d(\bm x,X_j), \forall j \in I\setminus{\{ i\} } \}\\
&=\{\overline F(\d(\bm x,X_i))>\overline F(\d(\bm x,X_j)), \forall j \in I\setminus{\{ i\} } \}\,.
\end{align*}
Denoting $U$ a standard uniform random variable and  using the independence of $\overline F(\d(\bm x,X_j)), j \in I\setminus{\{ i\} }$, we have for $\bm x \in \cal X$ and  $i \in I$
\begin{align*}
\P(nn(\bm x)=i|\overline F(\d(\bm x,X_i))&=\prod_{j \in I\setminus{\{ i\} }}\P(\overline F(\d(\bm x,X_j))<\\
& \qquad \overline F(\d(\bm x,X_i))\mid \overline F(\d(\bm x,X_i)))\\
&=\P(U< \overline F(\d(\bm x,X_i))\mid \overline F(\d(\bm x,X_i)))^{n-1}\\
&=(\overline F(\d(\bm x,X_i)))^{n-1}\,.
\end{align*}

\item[$iii)$] Let $u\in (0,1)$, from the definition of $\mu(\bm x)=\E[Y\mid X=\bm x]$ and the independence of $(X_j)_{j \in I}$, we have for $\bm x \in \cal X$ and  $i \in I$
\begin{align*}
&\E[Y_i\mid nn(\bm x)=i, \overline F(\d(\bm x,X_i))=u]\\
&=\E[\E[Y_i\mid nn(\bm x)=i, \overline F(\d(\bm x,X_i)),X_i]\mid nn(\bm x)=i, \overline F(\d(\bm x,X_i))=u]\\&=\E[\mu(X_i)\mid nn(\bm x)=i, \overline F(\d(\bm x,X_i))=u]\\&=\E[\mu(X_i)\mid \overline F(\d(\bm x,X_i))=u, u>\overline F(\d(\bm x,X_j)),\forall j \in I\setminus{\{ i\} }]\\&=\E[\mu(X_i)\mid  \overline F(\d(\bm x,X_i))=u]\,.
\end{align*}
\end{enumerate}
This concludes the proof.
\subsubsection{Proof of \Cref{lem:E_ynn2}}\label{sub: proof-of-lemmaC2}
Let $Y_{nn(\bm x)}$ be the 1-NN estimator at point $\bm x \in \cal X$ as defined in \Cref{def:1NN}.  Assume Condition \textbf{(D)} holds. For $\bm x \in \cal X$, applying \Cref{lem:1NN} combined with exchangeability,    we have
\begin{align*}
\E[Y_{nn(\bm x)}]&=\sum_{i \in I}\E[Y_i\1(nn(\bm x)=i)]\\
&=n\E[\E[Y_1\1(nn(\bm x)=1)\mid \overline F(\d(\bm x,X_1))]]\\
&=n\int_0^1\E[Y_1\1(nn(\bm x)=1)\mid \overline F(\d(\bm x,X_1))=u]du\\
&=n\int_0^1\E[Y_1\mid nn(\bm x)=1, \overline F(\d(\bm x,X_1))=u]\\
& \qquad \times \P(nn(\bm x)=1\mid \overline F(\d(\bm x,X_1))=u)du\\
&=n\int_0^1\E[\mu(X_1)\mid  \overline F(\d(\bm x,X_1))=u] u^{n-1}du \\
& = \E[g(\bm x,1-e^{-E/n})]
\end{align*}
where the last equality comes from \Cref{rem:Renyi}.
As a consequence of $Y_i$ following a Bernouilli distribution, 
the computation of the variance $\v(Y_{nn(\bm x)})$ is straightforward  since then $Y^2_{nn(\bm x)}=Y_{nn(\bm x)}$. 

\subsubsection{Proof of \Cref{lem:E_ynn}}\label{sub:proof-of-lemma5.3}
We keep the notation and the assumptions of \Cref{sub: proof-of-lemmaC2}. If moreover \textbf{(G1)} holds, for $u \in (0,1)$ and $\bm x \in \cal X$, since $g(\bm x,1-u):=\E[\mu(X)\mid  \d(\bm x,X)=F^{-1}(1-u)]$, using \Cref{rem:Renyi},  we have 

\begin{align*}
    |\E[Y_{nn(\bm x)}] -\mu(\bm x)|&\le   n\int_0^1 |g(\bm x,1-u) - \mu(\bm x) | u^{n-1} du\\
    &  \le n\int_0^1    \E[\E[|\mu(X)-\mu(\bm x)|\mid \d(\bm x,X)=F^{-1}(1-u), X]]\\
    & \qquad \qquad \times u^{n-1}du\\
    &  \le Ln\int_0^1    \E[\E[\d(\bm x,X)\mid \d(\bm x,X)=F^{-1}(1-u), X]] \\
    & \qquad \qquad \times u^{n-1}du\\
    & = L n\int_0^1F^{-1}(1-u)u^{n-1}du\\
    & = L \E[F^{-1}(1-e^{-E/n})].
\end{align*}

\noindent With no loss of generality consider $d$ to be the max-norm $\|\cdot \|_\infty$. Under Condition \textbf{(X)}, for $v\ge 0$ sufficiently small
\begin{align*}
F(v)=\P(\|\bm x-X\|_\infty\le v)&= \int_{\|\bm x-\bm y \|_\infty\le v}f(\bm y)d\bm y\\
&\ge m \text{Vol}(\{\bm y\in \R^d : \|\bm x-\bm y \|_\infty\le v\} )\\
&\ge m 2^dv^d\,.
\end{align*}
Thus by monotonicity we obtain
$F^{-1}(m 2^dv^d)\le v$, i.e. \\
$F^{-1}(v)\le v^{1/d}/(2m^{1/d})$ for $v$ sufficiently close to $0$. This implies
$$
F^{-1}(1-e^{E/n})\le F^{-1}(-E/n)\le (E/n)^{1/d}/(2m^{1/d})\,,
$$
given $E$ for $n$ sufficiently large. Therefore we obtain a.s.
\begin{align*}
\lim\sup_{n\to \infty}n^{1/d}F^{-1}(1-e^{-E/n})\le   \dfrac{E^{1/d}}{2m^{1/d}}\,.
\end{align*}
Since $\E[E^{1/d}]<\infty$, one may apply Fatou's Lemma to get 
 \begin{align*}
    \E[Y_{nn(\bm x)}]=\mu(\bm x)+O(n^{-1/d})\,,\qquad n\to \infty.  
  \end{align*}
This also proves $\v(Y_{nn(\bm x)})=\mu(\bm x)(1-\mu(\bm x))+O(n^{-1/d})$  since from i) we have proven that
$\v(Y_{nn(\bm x)})=\E[g(\bm x,1-e^{-E/n})]-\E[g(\bm x,1-e^{-E/n})]^2.$

\subsection{Proofs of \Cref{subsubsec:subsampling} }
In the following, denote $$U_1=\overline F(\d(\bm x,X_1))$$ the random variable, which, from \Cref{lem:1NN}, follows a standard uniform distribution. In order to prove results of \Cref{subsubsec:subsampling} we first need several preliminary results. This is the purpose of the next section. 
\subsubsection{Preliminary results}

\begin{prp}\label{Prop:E_ynn_Z1}Under Condition \textbf{(D)}, for $\bm x \in \cal X$,  we have
\begin{align*}
T^s_1 &=Y_1 U_1^{s-1}+\int_{U_1}^1 g(\bm x,1-u)(s-1)u^{s-2}\mathrm{d}u.
\end{align*}
\end{prp}

\begin{proof}[Proof of \Cref{Prop:E_ynn_Z1}]

Assume Condition \textbf{(D)} holds. Recall that from \\\Cref{lem:1NN} i), the random variable $U_1$ has standard uniform distribution. Next, for $\bm x \in \cal X$ we have 
\begin{align*}
\E[Y_{nn(\bm x)}\mid Z_1] 
&=\sum_{i=1}^s \E[Y_i\1_{nn(\bm x)=i}\mid Z_1]\\
&=Y_1\P(nn(\bm x)=1 \mid Z_1)+\sum_{i=2}^s \E[Y_i\1_{nn(\bm x)=i}\mid Z_1]\\
&=Y_1\P(nn(\bm x)=1 \mid Z_1)+(s-1)\E[Y_2\1_{nn(\bm x)=2}\mid Z_1]
\end{align*}
where the last line follows from exchangeability. 
For the first term, \Cref{lem:1NN}  iii) gives 
$$
Y_1\P(nn(\bm x)=1\mid  Z_1)=Y_1\overline F(\d(\bm x,X_1))^{s-1}=Y_1U_1^{s-1}.
$$
For the second term, we compute 
\begin{align*}
\E[Y_2\1_{nn(\bm x)=2}\mid Z_1]
=&\E\Big[Y_2\1_{\d(x,X_2)<\d(x,X_1)}\prod_{j=3}^s\1_{\d(x,X_2)<\d(x,X_j)}\mid Z_1\Big]\\
=&\E[Y_2\1_{\bar F(\d(x,X_2))>U_1)}\bar F(\d(x,X_2))^{s-2}\mid X_1]\\
=& \int_{U_1}^1 g(\bm x,1-u)u^{s-2}\mathrm{d}u\,
\end{align*}
where the last equality comes from the independence between $(X_2,Y_2)$ and $X_1$. This concludes the proof. 

\end{proof}

\subsubsection{Proof of \Cref{Prop:V1^s}}\label{proof:Prop:V1^s}

Assume Conditions \textbf{(D)}, \textbf{(X)} and \textbf{(G1)} hold. From \Cref{Prop:E_ynn_Z1}, for $\bm x \in \cal X$ we have 
 \begin{align*}
 T^s_1 =&(Y_1-g(\bm x,1-U_1)) U_1^{s-1} \\
 &+g(\bm x,1-U_1)) U_1^{s-1}+\int_{U_1}^1 g(\bm x,1-u)(s-1)u^{s-2}\mathrm{d}u\\
=&\underbrace{(Y_1-g(\bm x,1-U_1)) U_1^{s-1}}_{=:W}+ \underbrace{\int_0^1 g(\bm x,1-u)(s-1)u^{s-2}\mathrm{d}u}_{=:R_0} \\
 &+ \underbrace{g(\bm x,1-U_1)) U_1^{s-1}-\int_{0}^{U_1} g(\bm x,1-u)(s-1)u^{s-2}\mathrm{d}u\,}_{=: R_1(U_1)}.
 \end{align*}
Note that $Y_1$ is a Bernoulli random variable with parameter $g(\bm x,1-U_1)$ given $U_1$, then $\E[W\mid U_1]=0$. Besides,  since $R_0$ is deterministic,  
\[
V_1^s:=\v(T_1^s \mid Z_1)= \E[W^2]+\v(R_1(U_1)).
\]
On the one hand, since $(2s-1)\int_{0}^1 g(x,1-U_1)u^{2(s-1)}du$ is $\beta(2s-1,1)$-distributed, using \Cref{rem:Renyi} we have

\begin{align*}
\E[W^2]&=
\E\Big[\Big(Y_1-g(\bm x,1-U_1)\Big)^2U_1^{2(s-1)}\Big]\\
&=\E\Big[U_1^{2(s-1)}g(\bm x,1-U_1)(1-g(\bm x,1-U_1))\Big]\\
&=\dfrac{\E[g(\bm x,1-e^{-E/(2s-1)})(1-g(\bm x,1-e^{-E/(2s-1)}))]}{2s-1}\\
&=\dfrac{\E[g(\bm x,1-e^{-E/(2s-1)})-g^2(\bm x,1-e^{-E/(2s-1)}))]}{2s-1}\,.
\end{align*}
From \Cref{lem:E_ynn} i), since, 
\begin{align*}
\E[g^2(\bm x,1-e^{-E/(s-1)})- \mu^2(\bm x)] 
&\leq \E[(g(\bm x,1-e^{-E/(s-1)})- \mu(\bm x))\\
& \quad (g(\bm x,1-e^{-E/(s-1)})+\mu(\bm x))]\\
&\leq 2 \E[g(\bm x,1-e^{-E/(s-1)})- \mu(\bm x)]
\end{align*}

\noindent it holds that
 $$
  \E[g(\bm x,1-e^{-E/(s-1)})]-\E[g^2(\bm x,1-e^{-E/(s-1)})]=\mu(\bm x)-\mu^2(\bm x)+O(s^{-1/d}),\,\,  s\to \infty,
  $$
or equivalently
\begin{align}\label{eq:Wcarre}
    (2s-1)\E[W^2]=\frac{1}{2}\mu(\bm x)(1-\mu(\bm x))+O(s^{-1/d-1}),
    \end{align}
which is the expected result. On the other hand,  letting $E'$ a standard exponentially distributed random variable independent of $E$, since 
\begin{align*}
R_1(v)&=g(\bm x,1-v)) v^{s-1}+\int_{0}^{v} g(\bm x,1-u)(s-1)u^{s-2}\mathrm{d}u\\
&=v^{s-1}\Big(g(\bm x,1-v) +\int_{0}^{1} g(\bm x,1-vu)(s-1)u^{s-2}\mathrm{d}u\Big)\\
&=v^{s-1}\Big(g(\bm x,1-v) -\E[ g(\bm x,1-ve^{-E/(s-1)})]\Big)\,.
\end{align*}
This gives 
\begin{align*}
\E[R_1(U_1)^2]& = \E\Big[ U_1^{2(s-1)}\Big(g(\bm x,1-U_1))-\E[ g(\bm x,1-U_1e^{-E/(s-1)})\mid U_1]\Big)^2\Big]\\
& = \dfrac{\int_0^1 \big(\int_{0}^1(g(\bm x,1-u )-\E[ g(\bm x,1-ue^{-E/(s-1)})] \big)^2(2s-1)u^{2(s-1)}du}{2s-1}\\
& = \dfrac{\E[(g(\bm x,1-e^{-E'/(2s-1)})-\E[g(\bm x,1-e^{-E'/(2s-1)-E/(s-1)})|E '])^2]}{2s-1}
\,.
\end{align*}

Using Jensen's inequality, we also get
\begin{align*}
&\E[(\E[g(\bm x,1-e^{-E'/(2s-1)-E/(s-1)})-g(\bm x,1-e^{-E'/(2s-1)})|E'])^2]\\&\le \E[(g(\bm x,1-e^{-E'/(2s-1)-E/(s-1)})-g(\bm x,1-e^{-E'/(2s-1)}))^2]
\\&\le 2(\E[(g(\bm x,1-e^{-E'/(2s-1)-E/(s-1)})-\mu(\bm x))^2]\\
& \qquad+\E[(g(\bm x,1-e^{-E'/(2s-1)})-\mu(\bm x))^2])\,.
\end{align*}
We obtain the rate $O(s^{-2/d})$ as  $s\to \infty$ with a similar argument as above, which, together with \Cref{eq:Wcarre} concludes the proof.

\subsubsection{Proof of \Cref{Coro: TCL}}\label{sub: proof-of-cor5.7}

Assume Conditions \textbf{(D)}, \textbf{(X)} and \textbf{(G1)} hold. \Cref{Coro: TCL} is a direct application of \Cref{corth2}. We only  have to check that the three conditions {\bf (H1)}, {\bf (H2)}  are satisfied. Note that here $T^s=Y_{nn(\bm x)}^S$. First, from \Cref{lem:E_ynn} ii) we have
$$
\v(T^s)=\v(Y_{nn(\bm x)})\to \mu(\bm x)(1-\mu(\bm x))\,,\qquad s\to \infty.
$$
Thus $$\frac{\v(T^s)}{nV_1^s}=O(s/n)=o(1)$$ and {\bf (H1)} follows. 
$$
$$
Finally, using again \Cref{lem:E_ynn} ii),  {\bf (H2)} is satisfied because $n/s^{1+2/d}\to 0$. This concludes the proof.

\subsection{Proofs of \Cref{subsubsec:subsampling}}
Again we first need preliminary results.  In the following, we denote $$U_1^j=\overline F_j(\d(\bm x,X_1^j)), \ \ \ j=0,1.$$  Again, \Cref{lem:1NN} i) applies and $U_1^0$ and $U_1^1$ are independent standard uniform random variables. 
\subsubsection{Preliminary results}
The next proposition extends \Cref{Prop:E_ynn_Z1} to the under-sampling case.

\begin{prp}\label{Prop:EynnZ1^0} Under Condition \textbf{(D)},  we have

\begin{enumerate}
    \item[$i)$] $T_{1,0}^{s_0,s_1 }=\int_{\overline F_1(\overline F_0^{-1}(U_1^0))}^1\overline F_0(\overline F_1^{-1}(v))^{s_0-1}s_1v^{s_1-1}dv$;

    \item[$ii)$] $T_{1,1}^{s_0,s_1} =(U_1^1)^{s_1-1}\overline F_0(\overline F_1^{-1}(U_1^1))^{s_0}+\int_{U_1^1}^1\overline F_0(\overline F_1^{-1}(v))^{s_0}(s_1-1)v^{s_1-2}dv\,.$
\end{enumerate}

\end{prp}

\begin{proof}[Proof of \Cref{Prop:EynnZ1^0}] Assume Condition \textbf{(D)} holds. 
\begin{itemize}
    \item[$i)$] By definition of $Z_i^0$ and independence,  we have
\begin{align*}
T_{1,0}^{s_0,s_1 } & := \E[Y^{S_0 \cup S_1}_{nn(\bm x) }\mid Z_1^0] \\ 
&=\sum_{i \in S_1} \P(nn(\bm x)=i \mid Z_1^0)\\
&=s_1\P(\d(\bm x,X_i^1)<\d(\bm x,X_{j}^1), \forall j \in S_1\setminus{\{i \} },\\
& \qquad \d(\bm x,X_i^1)<\d(\bm x,X_{k}^0), \forall k \in S_0 \mid Z_1^0)\\
&=s_1\E\Big[\prod_{i \in S_1\setminus{\{1 \} }}(1-\P(\d(\bm x,X_{i}^1)\le \d(\bm x,X_1^1)\mid X_1^1))\Big.\\& \Big.\times\prod_{k \in  S_0\setminus{\{1 \} }}\P(\d(\bm x,X_{k}^0) > \d(\bm x,X_1^1)\mid X_1^1))\\
& \qquad \qquad \times \P(\d(\bm x,X_{1}^1)< \d(\bm x,X_1^0)\mid X_1^1,X_1^0)\Big| X_1^0\Big].
\end{align*}
Using several times the exchangeability among the subsamples of the classes $0$ and $1$,  since $U_1=\overline F_1(\d(\bm x,X_1^1))$ is a standard random variable for \Cref{lem:1NN} i), we obtain 
\begin{align*}
T_{1,0}^{s_0,s_1 } &=s_1\E\Big[\overline F_1(\d(\bm x,X_1^1))^{s_1-1}\overline F_0(\d(\bm x,X_1^1))^{s_0-1}\\
& \qquad \qquad \1(\d(\bm x,X_{1}^1)\le \d(\bm x,X_1^0))\Big| X_1^0\Big]\\
&=s_1\E\Big[U_1^{s_1-1}\overline F_0(\overline F_1^{-1}(U_1))^{s_0-1}\1(U_1>\overline F_1( \d(\bm x,X_1^0))\Big| X_1^0\Big],
\end{align*}
which concludes the proof. 

\item[$ii)$] Noticing that
\begin{align*}
T_{1,1}^{s_0,s_1 } & := \E[Y^{S_0 \cup S_1}_{nn(\bm x) } \mid Z_1^1] \\ 
&=\sum_{i \in S_1} \P(nn(\bm x)=i \mid Z_1^1)\\
&=\P(nn(\bm x)=1 \mid Z_1^1)+\sum_{i \in S_1\setminus{\{1 \} }} \P(nn(\bm x)=i \mid Z_1^1)
\end{align*}
and proceeding as in i), one gets 
\begin{align*}
T_{1,1}^{s_0,s_1} &=(U_1^1)^{s_1-1}\overline F_0(\overline F_1^{-1}(U_1^1))^{s_0}\\
& \qquad+(s_1-1)\E\Big[U_1^{s_1-2}\overline F_0(\overline F_1^{-1}(U_1))^{s_0}\1(U_1>U_1^1) \Big| X_1^1\Big],
\end{align*}
that is the desired result.
\end{itemize}
\end{proof}
 
\begin{lem}\label{lem=barF0barF} Assume Conditions \textbf{(D)}, \textbf{(X*)} and \textbf{(G1)} hold. For $p \in (0,1)$ and using \Cref{eq:odds_ration} we have 
\begin{align*}
\overline F_0(\overline F_1^{-1}(v))=:1-R(\bm x,p)(1-v)+O((1-v)^{1+1/d}).
\end{align*}
\end{lem}
\begin{proof}[Proof of \Cref{lem=barF0barF}]
    From the  definition of \\
    $\overline F_i(v)=\P(\d(\bm x,X_1)> v \mid Y_1=i),\,  i=0,1$, we deduce for $p \neq 0$
\begin{align}\label{eq:lem_barF0barF_1}
\overline F_1(v) =\dfrac{\P(Y_1=1,\overline F(\d(\bm x,X_1))\le \overline F(v) )}{\P(Y_1=1)}
=\dfrac{\int_{0}^{\overline F(v)}g(\bm x,u)du}{p}
\end{align}
which, since $\int_0^1g(\bm x,u)du=p$,  yields to
\begin{align}\label{eq:barF1_barF-1}
\overline F_1(\overline F^{-1}(v))=\dfrac{\int_0^vg(\bm x,u)du}p=1-\dfrac{\int_v^1g(\bm x,u)du}{p}=1-\dfrac{\int_0^{1-v}g(\bm x,1-u)du}{p}\,.
\end{align}
Using similar arguments than for the proof of \Cref{lem:E_ynn} ii), namely
$$|g(\bm x,1-u)-\mu(\bm x)| \leq L F^{-1}(1-e^{-E/n})=O((1-u)^{1/d}), \qquad u\sim 1, $$
combined with \Cref{eq:barF1_barF-1} we have for $p \neq 0$
\[
\overline F_1(\overline F^{-1}(v))=1-\dfrac{\mu(\bm x)}{p}(1-v)+O((1-v)^{1+1/d})\,.
\]
Thus when $\mu(\bm x)>0$ we deduce that

\begin{align*}
    \overline F(\overline F_1^{-1}(v))=y & \Leftrightarrow \overline F_1(\overline F^{-1}(y))=v\\
    & \Leftrightarrow 1-\dfrac{\mu(\bm x)}{p}(1-v)+O((1-v)^{1+1/d})=y\\
    & \Leftrightarrow 1-\dfrac{p}{\mu(\bm x)}(1-y)+O((1-y)^{1+1/d})=v,
\end{align*}
so that
\begin{align}\label{eq:inter}
\overline F(\overline F_1^{-1}(v))=1-\dfrac{p}{\mu(\bm x)}(1-v)+O((1-v)^{1+1/d})\,.
\end{align}
In the same fashion than \Cref{eq:lem_barF0barF_1}, we have for $p<1$
\[
\overline F_0(v)=\dfrac{\int_{0}^{\overline F(v)}(1-g(\bm x,u))du}{1-p}=\dfrac{\overline F(v)-p\overline F_1(v)}{1-p}
\]
that gives
\[
\overline F_0(\overline F_1^{-1}(v))=\dfrac{\overline F(\overline F_1^{-1}(v))-pv}{1-p}.\]
Plugging \Cref{eq:inter} into the latter identity we obtain for $p<1$
\begin{align}\label{eq:inter2}
\overline F_0(\overline F_1^{-1}(v))&= 1-\dfrac{p(1-\mu(\bm x))}{(1-p)\mu(\bm x)}(1-v)+O((1-v)^{1+1/d})\\
&=1-R(\bm x,p)(1-v)+O((1-v)^{1+1/d})\,.\nonumber
\end{align}

\end{proof}

\begin{lem}\label{lem:ProofFV}
    Assume Conditions \textbf{(D)}, \textbf{(X*)} and \textbf{(G1)} hold and let $\alpha,\beta >0$ such that $\alpha=O(\beta)$. Then we have 
    \begin{align*}
    \E\big[\big|\overline F_0(\overline F_1^{-1}(e^{-E/\beta}))^\alpha-e^{-\alpha R(\bm x,p)E/\beta}\big|\big]=O(\beta^{-1/d}).
    \end{align*}
\end{lem}

\begin{proof}[Proof of \Cref{lem:ProofFV}]
  First, a Taylor approximation $$1-e^{-E/\beta}=E/\beta+O((E/\beta)^2)$$ together with \Cref{lem=barF0barF} yield
\begin{align}\label{eq:grandTau}
|\overline F_0(\overline F_1^{-1}(e^{-E/\beta}))-e^{-R(\bm x,p)E/\beta}|&=|\overline F_0(\overline F_1^{-1}(e^{-E/\beta}))-1+R(\bm x,p)E/\beta|\nonumber\\
& \qquad +O((E/\beta)^{2})\nonumber\\
&\le O((E/\beta)^{1+1/d})+ O((E/\beta)^{2})\nonumber\\
& = O((E/\beta)^{1+1/d}).
\end{align}
Moreover, by convexity of the function $x\to x^{\alpha}$ for every $\alpha\ge 1$ we assert that 
\begin{equation}\label{eq:conv}
|x^{\alpha}-y^{\alpha}|\le \alpha(|x|\vee|y|)^{\alpha-1}|x-y|\,,\qquad x,y\in \R\,.
\end{equation}
Thus, combined with $0<\overline F_0(\overline F_1^{-1}(e^{-E/\beta}))\vee e^{-R(\bm x,p)E/\beta} \le 1$ and \Cref{eq:grandTau}, we obtain
\begin{align}\label{eq_inter}
|\overline F_0(\overline F_1^{-1}(e^{-E/\beta}))^{\alpha}-e^{-\alpha R(\bm x,p)E/\beta}|\le \alpha/\beta\, O(E^{1+1/d}/\beta^{1/d})\,.
\end{align}
Therefore given the assumption $\alpha/\beta=O(1)$, since $\E[E^{1+1/d}]<\infty$ we can apply reverse Fatou's Lemma to conclude the proof.  
\end{proof}

\subsubsection{Proof of \Cref{Prop:E[EynnZ1]}}\label{sub: proof-of-Prop5.12}

Assume Conditions \textbf{(D)}, \textbf{(X*)} and \textbf{(G1)} hold. 
\begin{enumerate}
    \item [$i)$] We first show that 
$$
I:= |\E[T_{1,0}^{s_0,s_1}] -\mu^*(\bm x) | = O(s^{-1/d}).
$$

\noindent From \Cref{Prop:EynnZ1^0} i), we have
\begin{align*}
    \E[T_{1,0}^{s_0,s_1}]&:=\E[\E[Y^{S_0 \cup S_1}_{nn(\bm x)}\mid Z_1^0]]=s_1\E[U_1^{s_1-1}\overline F_0(\overline F_1^{-1}(U_1))^{s_0}]\\
&=\E[\overline F_0(\overline F_1^{-1}(e^{-E/s_1}))^{s_0}].
\end{align*}
Besides, 
\begin{align*}
I & =|\E[\overline F_0(\overline F_1^{-1}(e^{-E/s_1}))^{s_0}]-\E[e^{-s_0R(\bm x,p)E/s_1}] + \E[e^{-s_0R(\bm x,p)E/s_1}]\\
& \qquad\qquad - \mu^*(\bm x)|\\
& \leq | \E[\overline F_0(\overline F_1^{-1}(e^{-E/s_1}))^{s_0}]-\E[e^{-s_0R(\bm x,p)E/s_1} ] |\\
& \qquad\qquad  + | \E[e^{-s_0R(\bm x,p)E/s_1}]- \mu^*(\bm x)|\\
& = R_1+R_2.
\end{align*}
Applying \Cref{lem:ProofFV} with $\alpha=s_0$ and $\beta=s_1$ we get $R_1=O(s_1^{-1/d})$.
Besides, we have 
\begin{align*}
\E\big[(\exp(-(s_0/s_1)R(\bm x,p)E)\big]&=\int_0^\infty \exp(-(1+(s_0/s_1)R(\bm x,p)))x)dx\\
&=\big(1+(s_0/s_1)R(\bm x,p)\big)^{-1}\,.
\end{align*}
and 
\begin{align*}
\mu^*(\bm x) =\dfrac{p^*}{p^*+(1-p^*)R(\bm x,p)}.
\end{align*}
Then, under the assumption $|s_1-p^*s|=O(s^{-(1+1/d)})$  as $s\to \infty$, we have 

\begin{align*}
R_2= |\big(1+(s_0/s_1)R(\bm x,p)\big)^{-1}- (1+\frac{1-p^*}{p^*}R(\bm x,p))^{-1}| = O(s_1^{-1/d}).
\end{align*}
Since $s_1=O(s)$ that implies $I=O(s^{-1/d})$ and concludes the proof. One may proceed similarly to prove $
\E[T_{1,1}^{s_0,s_1}]= \mu^*(\bm x)\, + O(s^{-1/d})$.

\item[$ii)$] Since $Y_i$ is a Bernoulli random variable, $(Y^{S_0 \cup S_1}_{nn(\bm x)})^2=Y^{S_0 \cup S_1}_{nn(\bm x)}$ and the result follows from $i)$.

\end{enumerate}

\subsubsection{Proof of \Cref{prop:1NN_V10^s}}\label{subsec:proof_prop:1NN_V10^s}

Assume Conditions \textbf{(D)}, \textbf{(X*)} and \textbf{(G1)} hold.
We start by proving that for $\bm x\in\mathcal X$
  \begin{align*}
I=\Big|V_{1,0}^{s_0,s_1}-  \dfrac{\mu^*(\bm x)^2(1-\mu^*(\bm x))}{2(1-p^*)s}\Big| =O(s^{-(1+1/d)})
\end{align*}
where $V_{1,0}^{s_0,s_1}=\v(T_{1,0}^{s_0,s_1})$ with $T_{1,0}^{s_0,s_1}=\E[Y^{S_0 \cup S_1}_{nn(\bm x)} \mid Z_1^0]$.  From  \Cref{Prop:EynnZ1^0} and \Cref{rem:Renyi} we have the identities
\begin{align}\label{eq:inter0}
T_{1,0}^{s_0,s_1}& \nonumber =\int_{\overline F_1(\overline F_0^{-1}(U_1^0))}^1\overline F_0(\overline F_1^{-1}(v))^{s_0-1}s_1v^{s_1-1}dv\\ \nonumber
& =\E[\overline F_0(\overline F_1^{-1}(e^{-E/s_1}))^{s_0-1}\1\{\overline F_0(\overline F_1^{-1}(e^{-E/s_1}))>U_1^0\}\mid U_1^0]\\\nonumber
& =\E[\overline F_0(\overline F_1^{-1}(e^{-E/s_1}))^{s_0-1}]\\
&\qquad -\E[\overline F_0(\overline F_1^{-1}(e^{-E/s_1}))^{s_0-1}\1\{\overline F_0(\overline F_1^{-1}(e^{-E/s_1}))<U_1^0\}\mid U_1^0]\,.
\end{align}
 The first term in  the r.h.s \Cref{eq:inter0} is deterministic and will not affect the variance calculation. Using \Cref{lem=barF0barF}, we can dominate the expectation of the second term in the r.h.s \Cref{eq:inter0} by 
\begin{align}\label{eq5:proof_final}
   \E[\P(\overline F_0(\overline F_1^{-1}(e^{-E/s_1}))<U_1^0)\mid U_1^0]& \nonumber= \E[1-\overline F_0(\overline F_1^{-1}(e^{-E/s_1}))]\\ \nonumber
   &=R(\bm x,p)\E[1-e^{-E/s_1}]\\
   & \qquad +O((1-\E[(1-e^{-E/s_1})^{1+1/d}])\\
     &=\dfrac{R(\bm x,p)}{s_1}+O(s_1^{-(1+1/d)})
\end{align}
 where the last line comes from a Taylor approximation \\
 $1-e^{-E/s_1}=E/s_1+O((E/s_1)^2)$. Therefore

\begin{align*}
V_{1,0}^{s_0,s_1}&=\E[(\E[Y_{nn(\bm x)}^*\mid Z_1^0]-\E[Y_{nn(\bm x)}^*])^2]\\
&=\E[(\E[\overline F_0(\overline F_1^{-1}(e^{-E/s_1}))^{s_0-1}\1\{\overline F_0(\overline F_1^{-1}(e^{-E/s_1}))<U_1^0\}\mid U_1^0])^2]\\
& \qquad +O(s^{-2})\\
&= \E[(\E[\exp(-(1-p^*)R(\bm x,p)/p^*E(1+O((E/s)^{1/d})))\wedge 1\\
&\quad \times\1\{E(1+O((E/s)^{1/d})>s_1(1-U_1^0)/R(\bm x,p)\}\mid U_1^0])^2]+O(s^{-2})\,.
\end{align*}

\begin{align*}
    V_{1,0}^{s_0,s_1}&=\E[(T_{1,0}^{s_0,s_1}-\E[T_{1,0}^{s_0,s_1}])^2]\\
&=\E[(\E[\overline F_0(\overline F_1^{-1}(e^{-E/s_1}))^{s_0-1}\1\{\overline F_0(\overline F_1^{-1}(e^{-E/s_1}))<U_1^0\}\mid U_1^0])^2]\\
& \qquad +O(s_1^{-(1+1/d)})\\
&= \E \Big[(U_1^0)^{2(s_0-1)}\E \Big[\Big(\dfrac{\overline F_0(\overline F_1^{-1}(e^{-E/s_1}))}{U_1^0}\Big) ^{s_0-1}\\
&\qquad \1\{\Big(\dfrac{\overline F_0(\overline F_1^{-1}(e^{-E/s_1}))}{U_1^0}\Big)^{s_0-1} \leq 1\} \Big | U_1^0\Big] ^2\Big]
 +O(s_1^{-(1+1/d)})\,.
\end{align*}
     Denote \begin{align}\label{eq2:proof_final}
    x=\overline F_0(\overline F_1^{-1}(e^{-E/s_1}))e^{\frac{E^\prime}{2s_0-1}}\mbox{ and }y=e^{-R(\bm x,p)E/s_1}e^{\frac{E^\prime}{2s_0-1}}
\end{align}
where the random variable $E' \sim \mathcal{E}(1)$ is independent of $E$. With this notation in mind, since $s_0/s_1=O(1)$, we obtain  
\begin{align*}
(2s_0-1)V_{1,0}^{s_0,s_1}&  \nonumber = O(s^{-1/d}) + \int_0^1\E \Big[\Big(\dfrac{\overline F_0(\overline F_1^{-1}(e^{-E/s_1}))}{u}\Big) ^{s_0-1}\\ \nonumber
&\qquad \times \1\Big\{\Big(\dfrac{\overline F_0(\overline F_1^{-1}(e^{-E/s_1}))}{u}\Big)^{s_0-1} \leq 1\Big\} \Big]^2(2s_0-1)u^{2s_0-1}du  \\ \nonumber
&=  \E\Big[\E\Big[\Big(\overline F_0(\overline F_1^{-1}(e^{-E/s_1}))e^{\frac{E^\prime}{2s_0-1}}\\
& \times \1 \{\overline F_0(\overline F_1^{-1}(e^{-E/s_1})) e^{\frac{E^\prime}{2s_0-1}}\le 1 \}\Big)^{s_0-1}\Big | E^\prime\Big] ^2 \Big] +O(s^{-1/d})\\  
& =  \E\big[\E\big[ x^{s_0-1}\1\{x\le1 \} | E^\prime\big] ^2  \Big]+O(s^{-1/d})\,.
\end{align*}
Denoting $I'=(2s_0-1)I$ we have 
\begin{align*}
    I' 
     & \leq \big|(2s_0-1) V_{1,0}^{s_0,s_1}-   \E\big[\E\big[ y^{s_0-1}\1\{y\le1 \} | E^\prime\big] ^2  \big]  \big| \\
     &\qquad + \Big| \E\big[\E\big[ y^{s_0-1}\1\{y\le1 \} | E^\prime\big] ^2  \big]  -  \dfrac{(2s_0-1)\mu^*(\bm x)^3R(\bm x,p)}{2p^*s}\Big|\\
     & =R_1+R_2.
\end{align*}
Applying the identity $a^2-b^2=(a-b)(a+b)$ with $a=\E\big[ x^{s_0-1}\1\{x\le1 \} \big| E^\prime\big] $ and $b=\E\big[ y^{s_0-1}\1\{y\le1 \}\big| E^\prime\big] $, since $a \vee b =1$ we obtain 
\begin{align*}
    R_1 
     & \leq 2 \E\big[ \big|\E\big[ x^{s_0-1}\1\{x\le1 \} | E^\prime\big]  - \E\big[ y^{s_0-1}\1\{y\le1 \}  | E^\prime\big]  \big| \big] +O(s^{-1/d})\\
     & \leq 2 \E\big[  \big| x^{s_0-1}\1\{x\le1 \}   -   y^{s_0-1}\1\{y\le1 \} \big|   \big] +O(s^{-1/d}).
     \end{align*}
We compute 
\begin{align*}
\big| x^{s_0-1}\1\{x\le1 \}   -   y^{s_0-1}\1\{y\le1 \} \big|=&    \big| x^{s_0-1}\1\{x\le1,y\le 1\}\\
& \qquad+  x^{s_0-1}\1\{x\le1,y> 1\}\\
&\qquad \quad -   y^{s_0-1}\1\{x\le 1,y\le1 \}\\
&\qquad \qquad-y^{s_0-1}\1\{x>1,y\le1 \}\big|\\
\le &\big| x^{s_0-1}-y^{s_0-1}\big|\1\{x\le 1,y\le1 \}\\
&+\1\{x\le1,y> 1\}+\1\{x>1,y\le1 \}\,.
\end{align*}
Applying the convex inequality \eqref{eq:conv} for $s_0-1\ge 1$, we upper bound the first term as
\begin{align*}
  |x^{s_0-1} -y^{s_0-1}|\1\{x\le 1,y\le1 \}&\le s_0|x-y|\,.
\end{align*}
Combining those inequalities we obtain 
\begin{align*}
     R_1 & \leq 2\Big(s_0\E\big[e^{\frac{E^\prime}{2s_0-1}}\big|\overline F_0(\overline F_1^{-1}(e^{-E/s_1}))-e^{-R(\bm x,p)E/s_1}\big| \big]\\
     &+\P(x\le1,y> 1)+\P(x>1,y\le1)\Big)+O(s^{-1/d})\\
     & =2(A_1+A_2+A_3)+O(s^{-1/d})\,.
     \end{align*}
Using the independence between $E$ and $E'$ and the uniform integrability of $e^{\frac{E^\prime}{2s_0-1}}$ for $s_0\ge 2$, from \Cref{lem:ProofFV}
     \begin{align*}
 A_1= s_0\E\big[ e^{\frac{E^\prime}{2s_0-1}}\big]\E\big[\big|\overline F_0(\overline F_1^{-1}(e^{-E/s_1}))-e^{-R(\bm x,p)E/s_1}\big|\big]=O(s^{-1/d})\,.
\end{align*}
It remains to deal with $A_2$, the control of $A_3$ is similar. One rewrites
\begin{align*}
    A_2&=\P((2s_0-1)R(\bm x,p)E/s_1<E'\le (2s_0-1)\log (\overline F_0(\overline F_1^{-1}(e^{-E/s_1}))))\\
    &=\big((\overline F_0(\overline F_1^{-1}(e^{-E/s_1})))^{2s_0-1}-e^{-(2s_0-1)R(\bm x,p)E/s_1}\big)_+=O(s^{-1/d})\,,
\end{align*}
by applying  again \Cref{lem:ProofFV} for $2s_0-1\ge 1$. We finally obtain  $R_1= O(s^{-1/d})$. It remains to control the term $R_2$. We have

\begin{align*}
\E\big[&\E\big[ y^{s_0-1}\1\{y\le1 \} | E^\prime\big] ^2  \big]\\ 
& =\E\Big[\E\Big[\exp\Big(-\frac{s_0-1}{s_1} R(\bm x,p) E+\frac{s_0-1}{2s_0-1} E^\prime\Big)\1 \Big\{\frac{R(\bm x,p) E}{s_1} >\frac{E^\prime}{2s_0-1}  \Big\}\Big | E^\prime\Big] ^2 \Big]\\
&= \E\Big[\exp\Big( \dfrac{2s_0-2}{2s_0-1}E^\prime\Big)\\
& \qquad \times \Big(\int_{s_1E^\prime/((2s_0-1)R(\bm x,p))}^\infty \exp(-((s_0-1) R(\bm x,p)/s_1+1) x )dx\Big)^2\Big]\\
&= \Big(\dfrac{s_1}{(s_0-1) R(\bm x,p)+s_1}\Big)^2\E\Big[\exp\Big( \dfrac{2s_0-2}{2s_0-1}E^\prime-\dfrac{2((s_0-1) R(\bm x,p)+s_1)}{(2s_0-1)R(\bm x,p)}E^\prime \Big)\Big]\\
&= \Big(\dfrac{s_1}{(s_0-1) R(\bm x,p)+s_1}\Big)^2\E\Big[\exp\Big( -\dfrac{2s_1}{(2s_0-1)R(\bm x,p)}E^\prime \Big)\Big]\\
&= \Big(\dfrac{s_1}{(s_0-1) R(\bm x,p)+s_1}\Big)^2 \int_{0}^\infty \exp(-(2s_1/(2s_0-1)R(\bm x,p)+1)x )dx \\
&= \Big(\dfrac{s_1}{(s_0-1) R(\bm x,p)+s_1}\Big)^2 \dfrac{(2s_0-1)R(\bm x,p)}{(2s_0-1)R(\bm x,p) +2s_1 }\,.
\end{align*}

Under the assumption $|s_1-p^*s|=O(s_1^{-(1+1/d)})$ we obtain 
$$
\big|\E\big[\E\big[ y^{s_0-1}\1\{y\le1 \} | E^\prime\big] ^2  \big]-\mu^*(\bm x)^2(1-\mu^*(\bm x))\big|=O(s^{-1/d})\,.
$$
Then, $I'=O(s^{-1/d})$ and 
\begin{align*}
\Big| V_{1,0}^{s_0,s_1}- \dfrac{\mu^*(\bm x)^2(1-\mu^*(\bm x))}{2s_0}\Big|=O(s^{-(1+1/d)})\,.
\end{align*}
Therefore, since $|s_0-(1-p^*)s|=O(s^{-(1+1/d)})$ we obtain $I=O(s^{-(1+1/d)})$ and concludes the proof. We treat similarly the term $V_{1,1}^{s_0,s_1}= \v(T_{1,1}^{s_0,s_1})$ with $T_{1,1}^{s_0,s_1}=\E[Y^{S_0 \cup S_1}_{nn(\bm x)} \mid Z_1^1]$ and details are left for the reader.  Recall  \Cref{Prop:EynnZ1^0}
\begin{align*}
    T_{1,1}^{s_0,s_1} =(U_1^1)^{s_1-1}\overline F_0(\overline F_1^{-1}(U_1^1))^{s_0}+\int_{U_1^1}^1\overline F_0(\overline F_1^{-1}(v))^{s_0}(s_1-1)v^{s_1-2}dv
\end{align*}
where from \Cref{rem:Renyi} we have  
\begin{multline}\label{eq3:final_proof}
\int_{U_1^1}^1\overline F_0(\overline F_1^{-1}(v))^{s_0}(s_1-1)v^{s_1-2}dv\\
 =\E[\overline F_0(\overline F_1^{-1}(e^{-E/(s_1-1)}))^{s_0}\1\{e^{-E/(s_1-1)}>U_1^1\}\mid U_1^1]\\ 
=\E[\overline F_0(\overline F_1^{-1}(e^{-E/(s_1-1)}))^{s_0}]\\
 -\E[\overline F_0(\overline F_1^{-1}(e^{-E/(s_1-1)}))^{s_0}\1\{e^{-E/(s_1-1)}<U_1^1\}\mid U_1^1]\,.
\end{multline}
Again the first term in the r.h.s  \Cref{eq3:final_proof} is deterministic, and similarly to \Cref{eq5:proof_final}, we can dominate the expectation of the second term in the r.h.s of \Cref{eq3:final_proof}  by $
   \E[1-e^{-E/(s_1-1)}]= 1/s_1+O(s^{-(1+1/d)})$.  Therefore, we have 
\begin{align*}
        V_{1,1}^{s_0,s_1}&=\E[(T_{1,1}^{s_0,s_1}-\E[T_{1,1}^{s_0,s_1}])^2]\\
        & = \E \big[\big(\big( (U_1^1)^{s_1-1}\overline F_0(\overline F_1^{-1}(U_1^1)\big)^{s_0} \\
        & \qquad - \E[\overline F_0(\overline F_1^{-1}(e^{-E/(s_1-1)}))^{s_0}\1\{e^{-E/(s_1-1)}<U_1^1\}\mid U_1^1] \big)^2 \big]+  O(s_1^{-2})\\
        & = \E\big[ \big( ( (U_1^1)^{s_1-1}\overline F_0(\overline F_1^{-1}(U_1^1))^{s_0}\\
        & \qquad - (U_1^1)^{s_0}\E\Big[\Big(\dfrac{\overline F_0(\overline F_1^{-1}(e^{-E/(s_1-1)})}{U_1^1}\Big)^{s_0}\1\{e^{-E/(s_1-1)}<U_1^1\}\mid U_1^1\Big] \big)^2 \big]\\
        & \qquad \qquad + O(s_1^{-2})\\
        &  =: \E[ (B_1-B_2)^2] + O(s_1^{-2}).
\end{align*}

Observe
\begin{align*}
(2s_1-1)\E[B_1^2]  &= \int_0^1\overline F_0(\overline F_1^{-1}(u))^{2s_0}(2s_1-1)u^{2s_1-2}du \\
&=\E[\overline F_0(\overline F_1^{-1}(e^{-E/(2s_1-1)}))^{2s_0}].
\end{align*}
Similarly, we have
\begin{align*}
(2s_0+1)\E[B_2^2]& =\E\Big[e^{\frac{2s_0 E'}{2s_0+1}}\E[\overline F_0(\overline F_1^{-1}(e^{-E/(s_1-1)}))^{s_0}\\
& \qquad \times \1\{e^{-E/(s_1-1)}<e^{-\frac{E'}{2s_0+1}}\}\mid E']^2\Big], 
\end{align*}
and 
\begin{align*}
& s\E[B_1B_2]=\E\Big[\overline F_0(\overline F_1^{-1}(e^{-E'/s}))^{s_0}e^{s_0 E'/s}\E[\overline F_0(\overline F_1^{-1}(e^{E/(s_1-1)}))^{s_0}\\
&\qquad \qquad \times  \1\{E/(s_1-1)>E'/s\}\mid E']\Big].
\end{align*}
\color{black}
Denote
\begin{itemize}
    \item $z_1=\exp(-(2s_0/(2s_1-1))R(\bm x,p)E)$
    \item $z_2=e^{\frac{2s_0 E'}{2s_0+1}}e^{-(s_0/(s_1-1))R(\bm x,p)E}$
    \item $z_3=\exp((s_0/s)(1-R(\bm x,p))E')\exp(-(s_0/(s_1-1))R(\bm x,p)E)$
\end{itemize}
Using \Cref{lem:ProofFV} similar arguments as in the proof for $V_{1,0}^{s_0,s_1}$,  we have
\begin{align*}
|(2s_1-1)\E[B_1^2]- \mu^*(\bm x) | & \leq |(2s_1-1)\E[B_1^2]- \E[z_1] | + |\E[z_1]-  \mu^*(\bm x) | \\
&=O(s_1^{-1/d}).
\end{align*}
In the same fashion, we have

\begin{align*}
 |(2s_0+1)\E[B_2^2] - \dfrac{\mu^*(\bm x)^3(1-p^*)}{p^*} | & \leq |(2s_0+1)\E[B_2^2]\\
 & - \E\Big[z_2\1\{E>E'\frac{s_1-1}{2s_0+1}\}\mid E']^2]\Big] | \\
& \qquad + |\E\Big[z_2\1\{E>E'\frac{s_1-1}{2s_0+1}\}\mid E']^2]\Big]\\
& \qquad - \dfrac{\mu^*(\bm x)^3(1-p^*)}{p^*} |\\  
&=O(s_1^{-1/d})
\end{align*}
and
\begin{align*}
|2s\E[B_1B_2]-\dfrac{\mu^*(\bm x)^2}{p^*}| &  \leq |2s\E[B_1B_2] -2\E[z_3\1\{(E>(s_1-1)E'/s\}\mid E']|\\
& \qquad + |2\E[z_3\1\{(E>(s_1-1)E'/s\}\mid E']-\dfrac{\mu^*(\bm x)^2}{p^*}|\\
&=O(s_1^{-1/d}).
\end{align*}

Combining these three terms, we obtain
\begin{align*}
V_{1,1}^{s_0,s_1}& =  \dfrac{\mu^*(\bm x)}{2s_1}+\dfrac{\mu^*(\bm x)^3(1-p^*)}{p^*2s_0}-\dfrac{\mu^*(\bm x)^2}{p^*s} +O(s_1^{-(1+1/d)})\\
& = \dfrac{\mu^*(\bm x)(1-\mu^*(\bm x))^2}{2p^*s} + O(s_1^{-(1+1/d)}).
\end{align*}
This concludes the proof. 

\subsubsection{Proof of \Cref{cor:TCL-under_1NN}} \label{subsec:cor:TCL-under_1NN}

Assume Conditions \textbf{(D)}, \textbf{(X*)} and \textbf{(G1)} hold.  \Cref{cor:TCL-under_1NN} is a direct application of  \Cref{th: mu_start_TCL}.  We need to check that  \textbf{(H1')}, \textbf{(H2')} are satisfied. From \Cref{Prop:E[EynnZ1]} ii) we have
$$
\v(T^{s_0,s_1})=\v(Y^{S_0 \cup S_1}_{nn(\bm x)}) \to  \mu^*(\bm x)(1-\mu^*(\bm x)),\qquad n\to \infty. $$
We now check   {\bf(H1')} by computing  
\begin{align*}  
\frac{\v(T^{s_0,s_1})}{n_1(s_0/s_1) V_{1,0} ^{s_0,s_1} +n_0(s_1/s_0)V_{1,1} ^{s_0,s_1}}&\sim \dfrac{s}{n}\Big(\frac{2(1-p^*)p\mu^*(\bm{x})}{p^*(1-p)(1-\mu^*(\bm{x}))+(1-p^*)p\mu^*(\bm{x})}\Big)\\
&=O(s/n)=o(1).
\end{align*}

 Unless otherwise stated $\E[\widehat \mu_{NN}^{s_0,s_1}(\bm x))]= \E[Y^{S_0 \cup S_1}_{nn(\bm x)}]$ thanks to H\'ajek's projection. Using \Cref{Prop:E[EynnZ1]}  \, i) $|\E[Y^{S_0 \cup S_1}_{nn(\bm x)}] -\mu^*(\bm x)\,|= O(s^{-1/d})$, so that {\bf (H2')} is satisfied since $n/s^{1+2/d}\to 0 $.  Applying  \Cref{degenerate-ust-2sample} and \Cref{lem:ratio_var_2D} we have 
 \begin{align*}
     \v( \widehat \mu_{NN}^{s_0,s_1}(\bm x))&= \dfrac{s_0^2}{n_0}V_{1,0}^{s_0,s_1} + \dfrac{s_1^2}{n_1} V_{1,1}^{s_0,s_1}\\
     &\sim \dfrac{s}{2n}\Big(\dfrac{(1-p^*) \mu^{*}(\bm{x})^2 (1-\mu^{*}(\bm{x}))}{(1-p)} +  \dfrac{p^* \mu^{*}(\bm{x}) (1-\mu^{*}(\bm{x}))^2}{p}\Big).
 \end{align*}
 This is held by \Cref{prop:1NN_V10^s} and also,  $s_0\sim (1-p^*)s$,  $n_0\sim (1-p)n.$

\begin{align*}
  \v( \widehat \mu_{NN}^{s_0,s_1}(\bm x))&\sim \dfrac{s\mu^{*}(\bm{x}) (1-\mu^{*}(\bm{x}))}{2n}\Big(\dfrac{(1-p^*) \mu^{*}(\bm{x})}{(1-p)} +  \dfrac{p^*(1-\mu^{*}(\bm{x}))}{p}\Big)\\
  & \sim \dfrac{s\mu^{*}(\bm{x}) (1-\mu^{*}(\bm{x}))}{2n}\Big(\dfrac{(1-p^*)p \mu^{*}(\bm{x}) + (1-p)p^*(1-\mu^{*}(\bm{x}))}{(1-p)p} \Big).
\end{align*}
We have that   
\begin{multline*}
\sqrt{\dfrac {2n}s}(\widehat \mu_{NN}^{s_0,s_1}(\bm x)-\mu^*(\bm x))\\
\overset{d}{\longrightarrow} 
\mathcal N\Big(0,\mu^{*}(\bm{x}) (1-\mu^{*}(\bm{x}))\Big(\dfrac{(1-p^*) \mu^{*}(\bm{x})}{(1-p)} 
 + \dfrac{p^*(1-\mu^{*}(\bm{x}))}{p}\Big)\Big), \,
\end{multline*}
as $n\to \infty.$
\begin{flushright}
  $\Box$  
\end{flushright}

\subsubsection{Proof of \Cref{cor:TCL-IS_1NN}} \label{subsec:cor:TCL-IS_1NN}
Since we have verified   \textbf{(H1')} and \textbf{(H2')} in  \Cref{subsec:cor:TCL-under_1NN}.  We still need  to verify  \textbf{(H2'')}. By applying   H\'ajek's projection, we get  
$\E[\widehat \mu_{NN}^{s_0,s_1}(\bm x))]= \E[Y^{S_0 \cup S_1}_{nn(\bm x)}]$. Therefore ,  we compute 
$$ g_n( \E [\widehat \mu_{NN}^{s_0,s_1}(\bm x)]))=\frac{n_1s_0(\mu^*(\bm x)\, + O(s^{-1/d}))}{n_0s_1(1- \mu^*(\bm x)\, + O(s^{-1/d}))+n_1s_0(\mu^*(\bm x)\, + O(s^{-1/d}))}.$$
Under \textbf{(G0)} and $|s_1-p^*s|=O(s^{-(1+1/d)})$, we use \Cref{est:mu_{I_n}S} to infer that
\begin{multline}\label{eq: bias_gn}
   g_n( \E [\widehat \mu_{NN}^{s_0,s_1}(\bm x)]))-\mu(\bm x) \\=\frac{(1-p^*)p(\mu^*(\bm x)\, + O(s^{-1/d}))}{(1-p)p^*(1- \mu^*(\bm x)\, + O(s^{-1/d}))+(1-p^*)p(\mu^*(\bm x)\, + O(s^{-1/d}))}\\
 \qquad -  \frac{(1-p^*)p\mu^*(\bm{x})}{(1-p)p^*(1-\mu^*(\bm{x}))+(1-p^*)p\mu^*(\bm{x})}\,.
\end{multline}
Denote 
\begin{itemize}
    \item $C_1=(1-p^*)p\mu^*(\bm{x})$
    \item $C_2=(1-p)p^*(1- \mu^*(\bm x))$
    \item $C_3=O(s^{-1/d}).$
\end{itemize}
Hence, we can rewrite \Cref{eq: bias_gn} as follows 
\begin{align*}
    | g_n( \E [\widehat \mu_{NN}^{s_0,s_1}(\bm x)]))-\mu(\bm x)| &= \Big| \dfrac{C_1 + (1-p^*)pC_3}{C_2+ C_1+ C_2C_3+ C_1C_3}- \dfrac{C_1}{C_2+C_1} \Big|\,.
\end{align*}
For small values of $C_3$ we can use a first-order of Taylor's approximation. It follows that the first fraction  simplifies as  
\begin{align*}
    \dfrac{C_1 + (1-p^*)pC_3}{C_2+ C_1+ C_2C_3+ C_1C_3}& = \dfrac{C_1 + (1-p^*)pC_3}{C_2+ C_1} \Big(1-\dfrac{ C_2C_3+ C_1C_3}{C_2+ C_1}\Big)\,.
\end{align*}
Then, the difference between the two fractions terms yields
 \begin{align*}
     & \dfrac{C_1 + (1-p^*)pC_3}{C_2+ C_1} \Big(1-\dfrac{ C_2C_3+ C_1C_3}{C_2+ C_1}\Big)- \dfrac{C_1}{C_2+C_1}\\
     & = \Big(\dfrac{C_1 + (1-p^*)pC_3}{C_2+ C_1} - \dfrac{C_1}{C_2+C_1}\Big)-\dfrac{ (C_1 + (1-p^*)pC_3)(C_2C_3+ C_1C_3)}{(C_2+ C_1)^2}\,.
 \end{align*}
 Finally we obtain
 \begin{align*}
    | g_n( \E [\widehat \mu_{NN}^{s_0,s_1}(\bm x)]))-\mu(\bm x)| &= \dfrac{(1-p^*)pC_3}{C_2+ C_1} \\
    & = \frac{(1-p^*)pO(s^{-1/d})}{(1-p)p^*(1-\mu^*(\bm{x}))+(1-p^*)p\mu^*(\bm{x})}.\\
\end{align*}
 Since, the  second part becomes a second-order approximation and can  be neglected for large values of $s$. Pre-multiplying by variance term,  it follows that 
 \begin{align*}
      \sqrt{\dfrac{1}{v_n}} | g_n( \E [\widehat \mu_{NN}^{s_0,s_1}(\bm x)]))-\mu(\bm x)|&= \frac{\sqrt{2(1-p)p}(1-p^*)pO(\sqrt{n}/s^{(\frac{1}{d}+ \frac{1}{2})})}{\sqrt{(C_2+ C_1)^3\mu^{*}(\bm{x}) (1-\mu^{*}(\bm{x}))}}.\\
 \end{align*}
 Therefore, \textbf{(H2'')} is satisfied, since $s=o(n)$, then  $s^{-(\frac{1}{d}+ \frac{1}{2})}\to 0$,  as $\sqrt{n} \to \infty$, i.e. $s\gg n^{\frac{2d}{d+2}}$.  We now compute  $ w= v_n V^*(\bm{x})$ in order to apply \Cref{cor:Two_case_th} and recall that 
\begin{align*}
    V^* (\bm{x}) &= \frac{(p p^{*}(1-p^{*}) (1-p)) ^2}{((1-p)p^{*}(1- \mu(\bm{x}))+p
		(1-p^{*}) \mu(\bm{x}))^4}\,,
\end{align*}
so that, as $n\to \infty$,
 \begin{align*}
     w\sim& \dfrac{s\mu^{*}(\bm{x}) (1-\mu^{*}(\bm{x}))}{2n}\Big(\dfrac{(1-p^*)p \mu^{*}(\bm{x}) + (1-p)p^*(1-\mu^{*}(\bm{x}))}{(1-p)p} \Big) \\
     &\times \frac{(p p^{*}(1-p^{*}) (1-p)) ^2}{((1-p)p^{*}(1- \mu(\bm{x}))+p
		(1-p^{*}) \mu(\bm{x}))^4}\\
  \sim& \dfrac{s\mu^{*}(\bm{x}) (1-\mu^{*}(\bm{x}))}{2n \mu(\bm{x})}\Big(\frac{(1-p)(p p^{*})^2(1-p^{*})^3 }{((1-p)p^{*}(1- \mu(\bm{x}))+p
		(1-p^{*}) \mu(\bm{x}))^4}\Big).
  \end{align*}
  Now, we can express the result in terms of $\mu(\bm x) $,  using \Cref{eq: inverse_g} we obtain
  \begin{align*}
  w&\sim \frac{s(p(1-p))^3(p^{*}(1-p^{*}))^4\mu(\bm{x})(1- \mu(\bm{x}))}{2n((1-p)p^{*}(1- \mu(\bm{x}))+p
		(1-p^{*}) \mu(\bm{x}))^7}\,,\qquad n\to \infty.
 \end{align*}
Combining all those bounds we achieve 
\begin{align*}
    \sqrt{\dfrac {2n}s}\,(\widehat \mu_{NN}^{IS}(\bm x)-\mu(\bm{x}))\overset{d}{\longrightarrow} \mathcal{N}\Big(\, 0,   \frac{(p(1-p))^3(p^{*}(1-p^{*}))^4\mu(\bm{x})(1- \mu(\bm{x}))}{((1-p)p^{*}(1- \mu(\bm{x}))+p
		(1-p^{*}) \mu(\bm{x}))^7}\Big), \, \,
\end{align*}
$ n\to \infty.$
\begin{flushright}
    $\Box$
\end{flushright}

\end{appendix}

 \textbf{Acknowledgments}. We would like to express our deep gratitude to the Reviewers for their valuable feedback and insightful questions regarding our work. We are also immensely grateful to Francesco Bonacina for his invaluable assistance in conducting these simulations. Additionally, we would like to specifically acknowledge the members of the T-REX project for funding our numerous workshops in Vienna and other locations, where productive discussions significantly enhanced this paper. We extend our thanks to the "International Emerging Actions 2022" and the "Centre National de la Recherche Scientifique" for supporting Moria Mayala and Charles Tillier's travel to Vienna.

\bibliographystyle{plainnat}
\bibliography{biblio-RF1}

\end{document}